\documentclass[11pt]{amsart}

\usepackage{amsmath,amsthm}
\usepackage {latexsym}
\usepackage{amssymb}
\usepackage{xcolor}

\usepackage[left=2.5cm,top=3cm,right=2.5cm,bottom=3cm,bindingoffset=0.5cm]{geometry}

\usepackage[utf8]{inputenc}
\newcommand{\beal}{\begin{align}}
\newcommand{\enal}{\end{align}}
\newcommand{\bealn}{\begin{align*}}
\newcommand{\enaln}{\end{align*}}
\newcommand{\bear}{\begin{eqnarray}}
\newcommand{\eear}{\end{eqnarray}}
\newcommand{\beeq}{\begin{equation}}
\newcommand{\eneq}{\end{equation}}

\newcommand{\eps}{{\varepsilon}}
\newcommand{\R}{{\mathbb R}}

\def\bm{\left[ \begin{array}{cc}}
\def\endm{\end{array}\right]}

\def\eps{\varepsilon}

\def\bm{\left[\begin{matrix} }
\def\endm{\end{matrix}\right]}

\def\R{{\mathbb R}}

\newtheorem{theorem}{Theorem}
\newtheorem{lemma}[theorem]{Lemma}
\newtheorem{defi}[theorem]{Definition}

\newtheorem{proposition}[theorem]{Proposition}

\newtheorem{remark}[theorem]{Remark}

\renewcommand{\Im}{\,{\rm Im}\,}
\renewcommand{\Re}{\,{\rm Re}\,}

\renewcommand{\hat}{\widehat}
\renewcommand{\epsilon}{\eps}
\renewcommand{\tilde}{\widetilde}
\numberwithin{equation}{section}
\numberwithin{theorem}{section}

\begin{document}

\title[Stabilization of focusing Klein-Gordon equations]{Boundary  stabilization of  focusing NLKG near unstable equilibria: radial case}

\author{Joachim Krieger}
\address{Bâtiment des Mathématiques, EPFL\\Station 8, CH-1015 Lausanne, Switzerland}
\email{\texttt{joachim.krieger@epfl.ch}}
\thanks{}

\author{Shengquan Xiang}
\address{Bâtiment des Mathématiques, EPFL\\
 Station 8, CH-1015 Lausanne, Switzerland}
\email{\texttt{shengquan.xiang@epfl.ch.}}
\thanks{}


\begin{abstract}
We investigate the stability and stabilization of  the cubic focusing Klein-Gordon equation around static solutions on the closed ball of radius L   in $\mathbb{R}^3$.  First we show that the system is  linearly unstable near the static solution $u\equiv 1$ for any dissipative boundary condition $u_t+ au_{\nu}=0, a\in (0, 1)$. Then by means of boundary controls  (both open-loop and closed-loop) we  stabilize  the system around this equilibrium exponentially under the condition  $\sqrt{2}L\neq \tan \sqrt{2}L$. Furthermore, we show that the equilibrium can be stabilized with any rate less than $ \frac{\sqrt{2}}{2L} \log{\frac{1+a}{1-a}}$, provided $(a,L)$ does not belong to a certain zero set. This rate is sharp.
\end{abstract}
\subjclass[2010]{35B34,  35B37,   35B40}
\thanks{\textit{Keywords.} cubic Klein-Gordon, focusing,  stabilization.}
\maketitle
\smallskip


\tableofcontents

\section{Introduction}
In this paper we  consider a model customarily studied on $\mathbb{R}^3$, the focusing energy subcritical Klein-Gordon equation
\begin{equation}\label{eq:KG}
\Box u+u = u^3, \; \Box=\partial_{tt}- \Delta. 
\end{equation}
This model admits static solutions, namely solitons  $W(x)$ (see for example \cite{coffman, Strauss-1977}), 
which are unstable equilibria \cite{Schlag-Nakanishi}.  The instability comes from the linearization around $W$ which admits a negative eigenmode, leading generically to exponential growth for the corresponding wave flow.  As one of the  core problems of dispersive PDEs, the long time behavior of solutions  of a number of related problems has been extensively  studied in the last twenty years and has led to fruitful theories.   Let us mention some of them:  the work on scattering by the Kenig-Merle method \cite{Kenig-Merle} on the energy critical focusing nonlinear wave equations,   the stability of special solutions using the center manifold method by Krieger-Nakanishi-Schlag \cite{Krieger-Nakanishi-Schlag-2012, Krieger-Nakanishi-Schlag-2013} in the context of the one-dimensional Klein-Gordon equation,  and  the soliton resolution for the energy critical wave equation by \cite{Thomas-Jia-Kenig-Merle}. See also the  book by Tao \cite{Tao-book-pde} on an excellent introduction on dispersive equations.

\subsection{The damping stabilization of the wave equation and the NLKG equation}
It is   natural to ask if unstable static solutions can be stabilized with the help of control terms: indeed, in control theory, this is the so called stabilization problem.  When we restrict wave equations to bounded domains,  control terms can be naturally implemented on  the boundary of the domain,  resulting in boundary control (or boundary stabilization).     The controllability of wave equations has been extensively studied in the literature, in particular the linear systems that will be dealing with are exactly controllable. This property can be shown via different approaches, for example the multiplier method \cite{Zuazua-1990}, Carleman estimates \cite{Tataru-JMPA-1994, Tataru-AMO-1995}, as well as microlocal analysis \cite{Bardo-Lebeau-Rauch}, and the references therein. We also note that the closely related unique continuation problem for wave operators is studied by Tataru, Hörmander, Robbiano--Zuily \cite{Hormander-97, Tataru-95, Robbiano-Zuily}  as well as others.   

To study the issue of stabilizing an unstable solution in the context of \eqref{eq:KG}, we briefly discuss some static solutions for it {\it{in the context of a bounded domain}}.   For example, imposing a Neumann boundary condition, a positive static solution solves the following elliptic Klein-Gordon equation,
\begin{gather}
-\Delta u+ u-u^3=0  \textrm{ in } \Omega, \label{N1}\\
u_{\nu}=0  \textrm{ on } \partial\Omega, \label{N2}\\
u>0  \textrm{ in } \Omega. \label{N3}
\end{gather}
To simplify the model and  reduce the difficulty, in this paper we shall only work with the equation in  a  {\it radial setting}, and from now on we let
\begin{gather}
\textrm{the domain  $\Omega$ denote the closed ball $B_{L/\sqrt{2}}(0)$ as subset of $\mathbb{R}^3$,} \notag
\end{gather}  
where the extra $\sqrt{2}$ is added for  normalization reasons to simplify the calculations.  Clearly \eqref{N1} - \eqref{N3} admits the constant solution  $u\equiv 1$, but in general it is not unique. We refer to    \cite{Brezis-Nirenberg, Bahri-Coron, Lions-1, Lions-2, Struwe-book} for the same problem with Dirichlet boundary conditions. In our specific case, one has the following result due to Lin-Ni-Takagi:
\begin{proposition}[Lin-Ni-Takagi, \cite{Lin-Ni-Takagi}]
There exist numbers $R_0$ and $R_1$ satisfying   $0<R_0<R_1$ such that,
\begin{itemize}
\item[(i)] if $L< R_0$, then  Equations \eqref{N1}--\eqref{N3}
only admit the  unique  solution $u= 1$;
\item[(ii)] if $L> R_1$, then  Equations  \eqref{N1}--\eqref{N3}
 also admit another non-trivial solution $u(x)=W_L(x)$. 
\end{itemize}
\end{proposition}
 Indeed, in \cite{Lin-Ni-Takagi} the authors   considered the more general system $-\Delta u+ \mu u-u^3=0$ with  a general  parameter $\mu>0$.  Moreover, when the nonlinearity is close enough to the  critical value,  Rey-Wei \cite{Rey-Wei} showed that  this non-trivial solution $W_L$ can be obtained by departing from the unique soliton solution $W$ of the same equation  in the unbounded domain.\\
  
In this paper,  we focus on the stabilization problem  around the simplest  solution $u=1$.  It  is easy to prove the instability around this equilibrium at the linear level by considering the linearization
\begin{gather}
\Box u- 2u=0   \textrm{ in } \Omega,    \notag\\
u_{\nu}=0   \textrm{ on } \partial\Omega.   \notag
\end{gather}
See Theorem~\ref{thm-unstab} below concerning a variant of this instability result.  

Various methods have been introduced to stabilize different wave equations, including the Gramian method \cite{Komornik}, the Riccati method \cite{Ammari-Duyckaerts-Shirikyan},  the backstepping method \cite{Krstic-wave},  the Lyapunov approach \cite{Coron-Trelat-2006, Lasiecka-Tataru-1993},  the duality method  \cite{Russell-1978},  the damping method \cite{Bardo-Lebeau-Rauch,  Burq-wave,  Burq-Gerard-wave, Lagnese,   Laurent-2011, Zuazua-stab-wave-1990}, to just name a sample. Let us comment more on the so called damping method, which is probably the most investigated  method on the stabilization of dispersive equations.  
Its idea  is to replace the Neumann boundary condition by a dissipative boundary condition, 
\begin{equation}
u_t+ au_{\nu}=0  \textrm{ on } \partial\Omega,   \notag
\end{equation}
and to show that the new system is stable.  To explain this in more detail, consider the model equation
\begin{gather}
\Box u=0   \textrm{ in } \Omega,    \notag\\
u_t+ a u_{\nu}=0   \textrm{ on } \partial\Omega.   \notag
\end{gather}
Define the energy of the wave equation by 
\begin{equation}\label{E0}
E_1(u[t]):= \frac{1}{2}\int_{\Omega} \left(|\partial_t u|^2+ |\nabla u|^2\right)(x) dx, \forall u[t]:= (u, u_t)\in H^1(\Omega),
\end{equation}
and perform a simple integration by parts. We thereby get the following ``decay'' property
\begin{equation}\label{iden-11}
\frac{d}{dt} E_1(u[t])= -\int_{\partial \Omega} \frac{1}{a(x)} |u_t(x)|^2 dx\leq 0.
\end{equation}
 It remains to understand how ``fast" the solution decays. For example, does the solution decay exponentially? This is equivalent to showing the following observability inequality: there exist some $T>0$ and $C>0$ such that the solution satisfies,
\begin{equation*}
\int_0^T \int_{\partial \Omega} \frac{1}{a(x)} |u_t(x)|^2 dx dt\geq C E_1(u[0]), \forall u[0]\in H^1(\Omega)\times L^2(\Omega).
\end{equation*}  
The proof of such an observability inequality is non-trivial, and different methods have been introduced for this purpose, for example the multiplier method. Based on the latter Lagnese characterized sufficient conditions on the choice of the ``damping function" $a(x)$ for exponential stabilization in \cite{Lagnese}. Other related works include for example  Zuazua \cite{Zuazua-stab-wave-1990} on the local internal stabilization of semilinear wave equations.    Another method, based on microlocal analysis, relates this problem to the propagation of singularities:  in the seminal work \cite{Bardo-Lebeau-Rauch}   Bardos-Lebeau-Rauch introduced the Geometric Control Condition (GCC) which provides an almost sharp condition for the controllability and the dissipative boundary stabilization  of the linear wave equation (see also Burq-Gérard \cite{Burq-Gerard-wave} and Burq \cite{Burq-wave} for improvements and simplifications). This powerful method has been widely used for the study of other linearized dispersive equations and nonlinear dispersive equations (but in a defocusing setting; as we shall see later on our focusing system is unstable even with dissipative boundary control). For example, Laurent  \cite{Laurent-2011} proved the  local stabilization of critical defocusing Klein-Gordon equations using internal control,  and Dehman-Lebeau-Zuazua \cite{Dehman-Lebeau-Zuazua}  proved the local boundary  stabilization of the subcritical semilinear wave equation. 
Furthermore, even when  GCC is not satisfied, Lebeau-Robbiano \cite{Lebeau-Robbiano-stabilization} proved asymptotic stability. Finally, we remark that all these results are proved in a general non-radial setting.  
\\

Let us now return to the equation arising upon linearizing \eqref{eq:KG} around the equilibrium $u = 1$, with dissipative boundary condition: 
\begin{gather}
\Box u- 2u=0   \textrm{ in } \Omega,    \notag\\
u_t+ a u_{\nu}=0   \textrm{ on } \partial\Omega.   \notag
\end{gather} 
Defining the energy
\begin{equation}\label{E01}
E_0\left( u[t]\right):= \frac{1}{2}\int_{\Omega} \left(|\partial_t u|^2+ |\nabla u|^2-2 |u|^2\right)(x) dx, \; \forall u[t]\in H^1(\Omega),
\end{equation}
we  obtain the ``decay''  property,
\begin{equation}\label{iden}
\frac{d}{dt} E_0(u(t))= -\int_{\partial \Omega} \frac{1}{a(x)} |u_t(x)|^2 dx\leq 0.
\end{equation}
However, since the energy $E_0$ is not positive definite (this comes from the focusing nature of \eqref{eq:KG}), the decrease of the value of $E_0$ does not imply the decay of the solution.  Indeed, as we shall prove later on in Theorem~\ref{thm-unstab}, this system is also unstable despite the dissipative boundary condition. Heuristically, we observe that  the instability ought to come from low-frequency modes of the system, since the high-frequency modes are expected to be  stable according the above mentioned theories on the wave equation.    
\\

To remedy this, our idea is to add a suitable control on the boundary, 
$$  u_t+ a u_{\nu}= b \textrm{ in  } \Omega$$ 
 to stabilize those finitely many unstable modes.  In general terms,  
\begin{itemize}
\item based on  explicit calculations we show that the spectrum of the operator generating the wave flow under the dissipative boundary condition has an asymptotic line in the left half complex domain (see Section \ref{subsec-spectral},   Lemma \ref{lem-eigen} and especially its proof for details);
\item using resolvent estimates we further prove that the high-frequency part of the system is exponentially stable with explicit decay rates (see Sections \ref{subsec-asym}--\ref{subsec-boum} for details);
\item we show that by adding suitable controls the low frequency part can be stabilized. This  step is achieved via a holomorphic extension approach (see Section \ref{sec-lin} and Section \ref{sec:constructionfeedback} for details).
\end{itemize}
This strategy will be further explained in Section \ref{sec-lin}.
Remark that the idea of exploiting the high-frequency stability of a system has been previously used by other authors. For example,  Ammari-Duyckaerts-Shirikyan proved  a general exponential stabilization result on damped defocusing like equations with  internal controls in \cite{Ammari-Duyckaerts-Shirikyan}, where they adapted the Riccati approach to stabilize the low frequency part. We emphasize that our stabilization approach appears different from the existing ones and provides explicit feedback laws with sharp decay rate.

\subsection{The main results}

In this paper, our goal is to investigate the controllability of the unstable static solution $u = 1$ of the focussing NLKG equation by means of a very direct and explicit method related to resolvent estimates, yielding an essentially sharp result.

First, the fact that this solution is (lineary) unstable follows from 
\begin{theorem}\label{thm-unstab}
Let $L> 0$. Let $a\in (0, 1)$.  The cubic focusing  Klein-Gordon system 
\begin{gather}
\Box u+ u-u^3=0, \textrm{ in } \Omega,    \notag\\
u_t+ au_{\nu}=0, \textrm{ on } \partial\Omega,  \notag
\end{gather}
is linearly unstable around the equilibrium $(1, 0)$.
\end{theorem}

As illustrated above the  main purpose of this paper  is to add some control term on the dissipative boundary,  $i. e. $ $u_t+ au_{\nu}=b(t)$,   to stabilize those unstable modes.
To reduce the difficulty,  in this paper we only consider the radial case, $i. e.$ both  initial data and boundary conditions  are radial, which  forces solutions  be radial.  It is to be expected, though, that the methods and results of this paper can be extended to the non-radial framework, at considerable technical expense. 
\\

Let us denote by $\mathcal{H}^1$ the space of {\it radial function} pairs 
\begin{equation}
\mathcal{H}^1:= \left\{
\begin{pmatrix}
u\\v
\end{pmatrix}: u\in H^1_{rad}(\Omega), v\in  L^2_{rad}(\Omega)
\right\},
\end{equation}
 and endow it  with the norm
\begin{equation}
\lVert (u, v)\lVert_{\mathcal{H}^1}^2:= \int_{\Omega} \left(|\nabla u|^2+ |v|^2+ |u|^2\right)(x) dx.   
\end{equation} 
The solution of the nonlinear Klein-Gordon equation with boundary control term $b(t)$ and given initial data $(u_0, v_0)$ is a function $u(t, x)\in C([0, T); \mathcal{H}^1)$ that satisfies the equation in the classical  {\it transportation sense}, where $T\in \mathbb{R}\cup \{+\infty\}$ is the blow up time, and throughout we use the boundary condition from before $u_t+ au_{\nu}=b(t)$.  This definition of solution,  introduced by Lions \cite{Lions-Hilbert}, is commonly used in control theory,  see the book by Coron \cite{coron} for an nice introduction to this subject.

The first main result of this paper is the following one concerning open-loop stabilization of  NLKG.
\begin{theorem}\label{thm-stab}
Let $L>0$ such that  $L\neq \tan L$. Let $a\in (0, 1)$.   There exist some  effectively computable constants  $\beta_*>0, C_{\beta_*}>0, \varepsilon_{\beta_*}>0, N_{\beta_*}\in \mathbb{N}$ and smooth functions $\{b_k(t)\}_{k=1}^{N_{\beta}}$ compactly supported on the time interval  $(2, 4)$  such that, for any radial initial state  $(u_0, v_0)^T\in \mathcal{H}^1$ satisfying 
\begin{equation}
\lVert (u_0, v_0)^T- (1, 0)^T\lVert_{\mathcal{H}^1}\leq \varepsilon_{\beta_*},  \notag
\end{equation}
 we are able to find a smooth real control function $b(t)$ as a linear combination of $\{b_k(t)\}_{k=1}^{N_{\beta}}$, 
\begin{equation}
b(t):= \sum_{k=1}^N \tilde l_k(u_0, v_0) b_k(t),   \notag
\end{equation}
with $\tilde l_k(u_0, v_0)\in \mathbb{R}$ depending continuously on $(u_0, v_0)^T\in \mathcal{H}^1$ satisfying
\begin{equation}
 \sum_{k=1}^N |\tilde l_k(u_0, v_0)|\leq 2 C_{\beta_*}\lVert (u_0, v_0)- (1, 0)\lVert_{\mathcal{H}^1},  \notag
\end{equation}
such that the unique  radial  solution of the nonlinear equation 
\begin{gather}
\begin{cases}
\Box u+ u-u^3=0 \textrm{ in } \Omega,      \notag\\
u_t+ au_{\nu}=b(t), \textrm{ on } \partial\Omega    \notag\\
u(0, x)= u_0, \; u_t(0, x)=v_0,  \notag
\end{cases}
\end{gather}
satisfies 
\begin{equation}
\lVert (u(t), u_t(t))^T-(1, 0)^T\lVert_{\mathcal{H}^1}\leq 2C_{\beta_*} e^{-\beta_* t} \lVert (u_0, v_0)^T-(1, 0)^T\lVert_{\mathcal{H}^1}, \; \forall t\geq 0.   \notag
\end{equation}

Moreover, the same conclusion obtains for any $\beta \in (0,  \frac{\sqrt{2}}{2L} \log{\frac{1+a}{1-a}})$ replacing $\beta_*$, provided $a$ does not belong to a certain zero measure set $\mathcal{A}(L)$ (see \eqref{def:AL} for its definition). 
\end{theorem}
\begin{remark}
The condition $L\neq \tan L$ is assumed to ensure that 0 is not an eigenvalue of the wave operator, and is a limitation of our method for technical reasons.  We  do not know whether it is possible to stabilize the system with the help of nonlinear terms even if the linearized system is not stable, as it is the  case for many other models, for example  phantom tracking for Euler equation \cite{coron-1999-euler-stabilization, 2005-Glass-Euler}, and power series expansion for KdV with critical length \cite{coron-rivas-xiang-kdv-16} etc.
\end{remark}

\begin{remark}
In terms of the distribution of eigenvalues in relation to the asymptotic line $ \frac{\sqrt{2}}{2L} \log{\frac{1+a}{1-a}}$, our result is presumably sharp. Moreover, observing that when $a$ tends to $1^-$ the value of $ \frac{\sqrt{2}}{2L} \log{\frac{1+a}{1-a}}$ tends to $+\infty$; thus in some sense we can understand this result as rapid stabilization,  $i. e.$ exponential stabilization with the decay rate being arbitrarily large.  Indeed, by taking $a=1$  this paper already presents the rapid stabilization.   The case  $a>1$ can be treated similarly  leading to an analogous result. \\
It is noteworthy that the study of the essential spectrum for the damped wave equation in a general geometric setting is more involved. This is related to the stability of the damped wave  equation: by contrast to our situation there is no need to add extra control terms to stabilize the system.  We refer to  the works of Koch--Tataru  \cite{Koch-Tataru},  Lebeau \cite{Lebeau-1996}  and  the references therein in this direction.  
 
\end{remark}
\begin{remark}
We can replace the support $(2, 4)$ by any other compact interval, which only affects the choices of the constants $C_{\beta}, \varepsilon_{\beta}$, and the functions $\{b_k(t)\}_{k=1}^{N_{\beta}}$, but not  the dimension of the control $N_{\beta}$ that we use.
\end{remark}

Notice that Theorem \ref{thm-stab} is an open-loop stabilization result, which is less robust to disturbances compared with closed-loop stabilization for engineering applications and realistic situations.  Generally speaking, it is hard to pass from an  open-loop stabilization result to a closed-loop stabilization result, since the open-loop control that we choose may be non-local in time while closed-loop feedback (except for some time-delay feedbacks) should only depend on the current state and time.   However,  we observe from   Theorem \ref{thm-stab} that  the control is chosen from a finite dimensional space which, in particular, is also compactly supported.  This essential observation makes it possible to derive  from Theorem \ref{thm-stab}  a closed-loop stabilization theorem   by means of time periodic feedback laws. 
\\
For this purpose we need to add some ``observers'', which are introduced in order to \textit{observe} the current state, more precisely,  to determine the value of $\tilde{l}_k$ that appears in Theorem \ref{thm-stab}.  This is a standard, and in many circumstances a necessary trick  for stabilization problems, see for example \cite{coron-1991-adding-integral, Coron-Xiang-2018, 2017-Xiang-SCL}.  

Therefore the state becomes $(u, u_t; l_1, l_2,..., l_{N_{\beta}})\in \mathcal{H}^1\times \mathbb{R}^{N_{\beta}}$, and  one needs to stabilize the following new system 
\begin{gather}
\begin{cases}
\Box u(t, x)+ u(t, x)-u^3(t, x)=0, \; t\in(s, +\infty), x\in \Omega      \notag\\
u_t(t, x)+ au_{\nu}(t, x)=b_f(t), \; t\in(s, +\infty),  x\in \partial\Omega,    \notag\\
\dot l_k(t)=0,  \;\forall k\in \{1, 2,..., N_{\beta}\},  t\geq s, t\in [NT, (N+1)T), N\in \mathbb{Z},\\
l_k(NT)=  \tilde{l}_k(u(NT), u_t(NT)), \; \forall k\in \{1, 2,..., N_{\beta}\},  NT\geq s, N\in \mathbb{Z}, 
\end{cases}
\end{gather}
with the help of some time periodic feedback law $b_f(t)$  that depends  on the state $(u, u_t; l_1, l_2,..., l_{N_{\beta}})$ and time $t$.  In fact, observing  from the last two condition in the preceding formula, one only needs to stabilize $(u, u_t)(t)$.   More precisely, we have the following closed-loop stabilization result.
\begin{theorem}\label{thm-stab-closed}
Let $L>0$ such that  $L\neq \tan L$. 
Let $a\in (0, 1)$.   Take the values of $\beta_*$ and $N_{\beta_*}$ from Theorem  \ref{thm-stab}.  For any $\varepsilon_0>0$ small enough, there exist   effectively computable constants $\tilde C>0, \tilde \varepsilon>0, \tilde T>0$, smooth functions $\{b_k(t)\}_{k=1}^{N_{\beta_*}}$ compactly supported on $(2, 4)$, and $\tilde T$-periodic feedback laws that depends on the value of $(l_1, l_2,..., l_{N_{\beta_*}})(t)$ and time $t$:
\begin{equation}
b_f(t)= b_f\big(l_1(t), l_2(t),..., l_{N_{\beta_*}}(t)\big)(t):= \sum_{k=1}^{N_{\beta_*}} l_k(t) b_k\left(t-[\frac{t}{\tilde T}]\tilde T\right), \forall t\in \mathbb{R},  \notag
\end{equation} 
such that, for any $s\in \mathbb{R}$, for any radial initial state  $\left(u_0, v_0; l_1^0, l_2^0,..., l^0_{N_{\beta_*}}\right)\in \mathcal{H}^1\times \mathbb{R}^{N_{\beta_*}}$ satisfying the smallness condition 
\begin{gather}
\lVert (u_0, v_0)^T- (1, 0)^T\lVert_{\mathcal{H}^1}\leq \tilde \varepsilon,    \notag
\end{gather}
and the compatible condition 
\begin{gather}
l^0_k= \tilde{l}_k(u_0, v_0), \;\forall k\in \{1, 2,..., N_{\beta_*}\},   \notag
\end{gather}
where the continuous function $\tilde{l}_k$ and the smooth functions  $\{b_k(t)\}_{k=1}^{N_{\beta}}$ are chosen directly from  Theorem \ref{thm-stab}, the unique radial solution $(u, u_t; l_1, l_2,..., l_{N_{\beta_*}})(t)$ of the nonlinear equation,
\begin{gather}
\begin{cases}
\Box u(t, x)+ u(t, x)-u^3(t, x)=0, t\in(s, +\infty), x\in \Omega      \notag\\
u_t(t, x)+ au_{\nu}(t, x)=b_f(t),  t\in(s, +\infty),  x\in \partial\Omega,    \notag\\
\dot l_k(t)=0,  \forall k\in \{1, 2,..., N_{\beta_*}\},  t\geq s, t\in [N \tilde T, (N+1)\tilde T), N\in \mathbb{Z},\\
l_k(N\tilde T)=  \tilde{l}_k(u(N\tilde T), u_t(N\tilde T)), \forall k\in \{1, 2,..., N_{\beta_*}\},  N\tilde T\geq s, N\in \mathbb{Z}, \\
u(s, x)= u_0, u_t(s, x)=v_0, \\
l_k(s)= l_k^0, \forall k\in \{1, 2,..., N_{\beta_*}\},
\end{cases}
\end{gather}
satisfies 
\begin{gather}
\lVert (u(t), u_t(t))^T-(1, 0)^T\lVert_{\mathcal{H}^1}\leq \tilde C e^{-(\beta_*-\varepsilon_0)  (t-s)} \lVert (u_0, v_0)-(1, 0)\lVert_{\mathcal{H}^1}, \; \forall t\geq s,   \notag\\
\sum_{k=1}^{N_{\beta_*}} |l_k(t)| \leq 2 \tilde C e^{-(\beta_*-\varepsilon_0) (t-s)} \lVert (u_0, v_0)-(1, 0)\lVert_{\mathcal{H}^1}, \; \forall t\geq s.  \notag
\end{gather}
Moreover, the same conclusion obtains for any $\beta \in (0,  \frac{\sqrt{2}}{2L} \log{\frac{1+a}{1-a}})$ replacing $\beta_*$, provided $a$ does not belong to a certain zero measure set $\mathcal{A}(L)$.
\end{theorem}
\begin{remark}
We do not know whether adding $N_{\beta}$ integral terms is ``\textit{optimal}'', as the controllability and stabilization with reduced control terms is among one of the central problems in control theory, especially for nonlinear systems for which the nonlinearity  may provide plenty of extra information rather than linear systems, see for example \cite{Coron-Lissy} concerning local exact controllability of three dimensional Navier-Stokes equations with only one  controlled component.
\end{remark}

To the best of our knowledge, this appears to be the first attempt to stabilize  multi dimensional unstable focusing dispersive equations  using resolvent estimates, and this stabilization result also shares the advantage of explicit feedback with a sharp decay rate.

We note that under the radial assumption on solutions  the NLKG  can be more or less regarded as a one dimensional wave equation,  for which the understanding of controllability, stability and stabilization  is more complete, to be compared with the more complicated  higher dimensional case. For example local controllability and global controllability even with different type of nonlinearities has been obtained by Zuazua \cite{Zuazua-1990, Zuazua-1993} (see also \cite{coron} for an introductory proof of the local controllability),  exponential  stabilization using  Lyapunov functions is obatined   by Coron-Trélat  \cite{Coron-Trelat-2006} , or by the backstepping method  \cite{Krstic-wave} (this general approach has been applied to different types of models, see for example the heat equation  \cite{2017-Coron-Nguyen-ARMA} and KdV equation \cite{2019-xiang-SICON}).

 Nevertheless, though benefiting from the simplicity of the radial setting,  we do not use any other specific one dimensional structures in this paper, which suggests the possibility to extend things to the multi-dimensional non-radial case.  \\

This paper is structured as follows. In Section \ref{sec-lin} we introduce some general facts concerning the linear inhomogeneous problem as well as our strategy of stabilizing unstable modes.  Section \ref{sec-stab-li}  is the main part of the paper; we prove that under the radiality assumption via explicit resolvent estimates the static equilibrium is unstable (Theorem \ref{thm-unstab}) and that with the help of some control term on its linearized system the solution  will decay exponentially (Theorem \ref{thm-li-key}). This is followed by a section on open-loop stabilization of the nonlinear system concerning Theorem \ref{thm-stab}  as well as a section on closed-loop stabilization concerning Theorem \ref{thm-stab-closed}.  In the end,  in Section \ref{sec-com} we comment on some interesting further questions, and furnish some technical proofs in  Appendix \ref{App-A} and Appendix \ref{App-B}.

\section{Inhomogeneous linear problem and our stabilization strategy}\label{sec-lin}
In this section we introduce our stabilization strategy for  general linear wave equations with potential terms and boundary controls
\begin{gather}
\Box u-V u=h(t, x)  \textrm{ in } \Omega, \label{inl1}\\
(u_t+ a u_{\nu})(t)= b(t)   \textrm{ on } \partial\Omega, \label{inl2}  \\
u(0, x)= u_0, \; u_t(0, x)= v_0, \label{inl3}
\end{gather}
where the potential $V$ is assumed to be radial, bounded and smooth, and $a\in (0, 1)$ throughout. Also, the function $b(t)$ is always assumed to be $C^\infty$ and compactly supported. 

\subsection{Our stabilization strategy}\label{subsec-strategy}
 In this subsection we briefly comment on  our strategy of getting stabilization, and leave the more detailed explanations to Section \ref{subsec:elli}--\ref{subsec-str}.  It is essentially  composed of two parts:
\begin{itemize}
\item Transform the evolution problem into an elliptic one via Fourier transformation in time. This procedure gives some function $U(\omega, x)$ that is well-defined if $-\Im \omega$ is large enough;
\item  Extend the function $U(\omega, x)$ holomorphically to a larger domain with the help of control terms if necessary, and perform resolvent estimates for $U(\omega, x)$.
\end{itemize}

We note that the idea of using the Fourier transform in time has been used for the study of the stability problem of damped wave equations in different settings. Let us mention, for example,  the works by Lagnese \cite{Lagnese}, Burq--Zworski \cite{Burq--Zworski},  Duyckaerts-Miller \cite{Duyckaerts-Miller} and the references therein.  Remark that in most of the literature one deals with stable linearized equations, thus holomorphic extension is always possible, and one mainly concentrates on the resolvent estimates.  Indeed, in this circumstance the damped wave equation is exponentially stable.  However, in our framework, as we shall see later on in Section \ref{thm-unstab}, the damped wave equation is unstable and holomorphic extension is only guaranteed by working with well-chosen controls.    \\

Let us now briefly present the idea of Fourier transform in time.  Let us define the  partial Fourier transform
\begin{equation}\label{def_U0}
U_0(\omega, x):= \int_0^{+\infty}e^{-i\omega t} u(t, x) dt,       
\end{equation}
where $u$ is assumed to be a solution of \eqref{inl1}--\eqref{inl3}.   Straightforward calculation implies that $U_0$ satisfies an elliptic equation 
\begin{gather}
\Delta U_0+ (V+\omega^2)U_0= -H_0 \textrm{ in } \Omega,  \label{U0eq1} \\
i\omega U_0+ aU_{0, \nu}= B_0 \textrm{ on } \partial \Omega,   \label{U0eq2}
\end{gather}
 
Then, at least formally, for $\omega= \alpha+ i\beta$ with some fixed negative valued $\beta$  we can define for $t>0$ 
\begin{equation}
e^{\beta t}u_{\beta}(t, x)= \frac{1}{2\pi} \int_{-\infty}^{+\infty} e^{i\alpha t}U_0(\alpha+i \beta, x)d\alpha=: F(\beta, t, x).   \label{u:def:inverse}
\end{equation}
That the preceding expressions make sense for $\beta$ a negative number of sufficiently large absolute value follows from Lemma \ref{lem-a-es}. Therefore, in this circumstance the preceding  formula is exactly the inverse Fourier transformation, and  the value of $u_{\beta}(t, x)$ coincides the unique solution of the wave equation $u(t, x)$.   \\

If furthermore, we are able {\it{to extend}} the definition of $U_0(\omega, x)$, $F(\beta, t, x)$, to much larger sets of the complex parameter $\omega$, in particular to the region where $\beta>0$, 
and to prove for some $\beta_0>0$ that 
\begin{equation}
\lVert F(\beta_0, t, \cdot)\lVert_{H^1}\leq C(\beta_0), \; \forall t\in [0, +\infty),    \label{F:estimate}
\end{equation}
then we get the required exponential decay of $u$.  

 This requires us to extend $U_0(\omega, x)$ {\it holomorphically} from the region of $\beta$ sufficiently small to the region $\beta\leq \beta_0$.  Note that generally we are not allowed to extend $U_0$ via the definition  \eqref{def_U0}, since the integral may not be well-defined.   Instead we shall perform the extension by solving \eqref{U0eq1}--\eqref{U0eq2}.  We shall see later in subsection \ref{subsec-str} that these equations need not  admit a solution for certain specific values of $\omega$ (namely the poles), unless $H_0$ and $B_0$ satisfy a suitable compatibility condition.  Notice that the value of $B_0$ is directly related to the values of the control term $b(t)$ (see Equation \eqref{ellip-4}), while $H_0$ is directly related to the initial states (see Equation \eqref{ellip-3}).  Our strategy is to exhibit this compatibility condition on $(H_0, B_0)$, or equivalently to find suitable controls $b(t)$,   such that  one can solve $U_0$ holomorphically in the complex region $\Im \omega\leq \beta_0$. Consequently, the function $u_{\beta}(t, x)$ defined in \eqref{u:def:inverse}  is analytically extended to the larger region $\beta\leq \beta_0$, and still coincides  with the unique solution of the wave equation $u(t, x)$.    In the following when there is  no  confusion we shall simply denote $u_{\beta}(t, x)$ by $u(t, x)$.
 
 Finally, we obtain the required estimate \eqref{F:estimate} using resolvent estimates (see subsections \ref{subsec-inh}--\ref{subsec-boum} for details).    \\

\subsection{Relating the stability question to elliptic problems}\label{subsec:elli}
As a simple starting point, we state the following lemma.  Let us define the energy of the system,
\begin{equation}
E(t):= \frac{1}{2}\int_{\Omega} \left(|\partial_t u|^2+ |\nabla u|^2 +|u|^2\right)(t, x) dx     \notag
\end{equation}
and 
\begin{equation}
\tilde{E}(t):=E(t)+ \frac{1}{a}\int_0^t \int_{\partial \Omega}u_t^2(s, x)d\sigma ds.   
\end{equation}
\begin{lemma}\label{lem-a-es}
There exists some  $C=C(V, \Omega)$ such that for any given $(u_0, v_0)^T \in \mathcal{H}_1$,  the unique solution  of \eqref{inl1}--\eqref{inl3} satisfies
\begin{equation}
 \left(\tilde{E}(t)\right)^{\frac{1}{2}}\lesssim e^{Ct}(E(0))^{\frac{1}{2}}+ \int_0^t e^{C(t-s)}\left(\lVert h(s, \cdot)\lVert_{L^2}+ |b(s)| + |b''(s)|\right)ds, \; \forall t\in [0, +\infty). \notag
\end{equation}
\end{lemma}
\begin{proof}
At first we consider the case $b(t)=0$.
Direct calculation yields 
\begin{equation}
\frac{d}{dt}\tilde{E}(t)= \int_{\Omega}\left((1+V)u_tu+ u_t h\right)dx\leq C \tilde{E}(t)+\sqrt{\tilde{E}(t)} \lVert h(s, \cdot)\lVert_{L^2(\Omega)},    \notag
\end{equation}
thus 
\begin{equation}
\frac{d}{dt}\sqrt{\tilde{E}(t)}\leq C \sqrt{\tilde{E}(t)}+ \lVert h(s, \cdot)\lVert_{L^2(\Omega)},     \notag
\end{equation}
which implies the lemma via Gronwall's inequality.

As for the case $b(t)\neq 0$, by taking the difference of $u$ and $g(x)b(t)$ with some smooth function $g(x)$ satisfying $g(L)=0, g_x(L)=a^{-1}$ we transform the term $b(t)$ into the source term, which can then be handled by the above inhomogeneous free boundary control case. This concludes the required estimate.
\end{proof}
\begin{remark}
Observe from the energy estimate that the boundary trace has temporal derivative bounded in the $L^2$ sense, while the usual trace formula only provides its continuity. This is indeed a hidden inequality that comes from the boundary value problems, and sometimes such a kind of hidden inequality is the key to prove controllability results (see for example \cite{rosier97, Krieger-Xiang-kdv} for KdV).  Moreover, we shall also benefit from this trace estimate in our stabilization problem.
\end{remark}

In order to  take advantage of the inverse Fourier transformation, we need to  extend $u(t)$ by 0 on $t<0$ via a smooth truncation. 
Let $\chi\in C^{\infty}(\mathbb{R})$ satisfies $\chi(t)=0$ for $t\leq 1$, and $\chi(t)=1$ for $t\geq 2$. Define  $w:= \chi u$, it satisfies
 \begin{gather}
 \Box w- Vw= \chi_{tt} u+ 2\chi_t u_t+ \chi  h= h_0  \textrm{ in }  \Omega, \label{wave-op-1} \\
 w_t+ aw_{\nu}= \chi b+ \chi_t u= b_0  \textrm{ on } \partial \Omega, \label{wave-op-2}\\
 w(0, x)=w_t(0, x)=0. \label{wave-op-3}
 \end{gather}
Let us assume that the control $b(t)$ that we will choose later on satisfies that 
\begin{equation}\label{assum-b-02}
\textrm{ supp } b(t)\Subset (2, +\infty).
\end{equation} 
Thanks to this assumption, we know that $u(t, x)|_{t\in (0, 2)}$  is uniquely determined by $(u_0, v_0)(x)$ and $h(t, x)|_{t\in (0, 2)}$. Therefore, the boundary term $b_0(t)|_{t\in (0, 2)}$  is given by   $\chi_t u(t, x)|_{t\in (0, 2)}$, and  the source term $h_0(t, x)$ is fixed irrespective of the choice of $b(t)$.  Moreover, since the energy and the trace  of $u(t)$ are bounded by 
\begin{align*}
E(t)&\lesssim E(0)+ \left(\int_0^t \lVert h(s, \cdot)\lVert_{L^2}ds\right)^2, \forall t\in (0, 2),\\
|u(t, L)|^2&\lesssim E(0)+ \left(\int_0^t \lVert h(s, \cdot)\lVert_{L^2}ds\right)^2, \forall t\in (0, 2),\\
\int_0^t u_t^2(s, L)ds&\lesssim E(0)+ \left(\int_0^t\lVert h(s, \cdot)\lVert_{L^2}ds\right)^2, \forall t\in (0, 2),
\end{align*}
we know that
\begin{align*}
\lVert h_0(t)\lVert_{L^2}^2\lesssim \lVert h(t)\lVert_{L^2}^2+ \lVert u_0\lVert_{H^1}^2+  \lVert v_0\lVert_{L^2}^2+ \left(\int_0^t \lVert h(s, \cdot)\lVert_{L^2}ds\right)^2, \forall t\in (0, 2).
\end{align*}

It suffices to stabilize $w= \chi u$ via a good choice of $b_0(t)$ keeping in mind that on the time interval $[0, 2]$  this function is given by $w(t)= \chi'(t)u(t, L)$. 
\\

By defining 
\begin{equation}\label{eq:Uomega}
U(\omega, x):= \int_0^{+\infty}e^{-i\omega t} w(t, x) dt= \int_{-\infty}^{+\infty}e^{-i\omega t} w(t, x) dt,  
\end{equation}
thanks to Lemma \ref{lem-a-es}, we know that for any complex number $\omega$ satisfying  $\Im  \omega< -C(V, \Omega)$  both $U_0(\omega, x)$ and $U(\omega, x)$ are well-defined and belong to the  $H^1$ space.  Thanks to the inverse Fourier transform,  for $\omega= \alpha+ i\beta$ satisfying $\beta\leq -C(V, \Omega)$, the function $w(t, x)$ satisfies
\begin{align} \label{four-inverse}
e^{\beta t} w(t, x)&= \frac{1}{2\pi} \int_{-\infty}^{+\infty} e^{i\alpha t}U(x, \alpha+i\beta)d\alpha,\\
e^{\beta t} w_t(t, x)&= \frac{1}{2\pi} \int_{-\infty}^{+\infty} (i\alpha- \beta) e^{i\alpha t}U(x, \alpha+i\beta)d\alpha. \label{four-inverse2}
\end{align}
 In the following if there is no confusion we will call $w(t)$ (or sometimes $u(t)$) the function corresponding to $U(\omega)$.
Then, we  derive the following relations via integration by parts:
\begin{align*}
\int_0^{+\infty}e^{-i\omega t}w_{tt}dt= e^{-i\omega t}&w_t|_{0}^{+\infty}+ i\omega \int_0^{+\infty}e^{-i\omega t}w_{t}dt=-\omega^2 U(\omega, x),\\
aU_{\nu}=a\int_0^{+\infty} e^{-i\omega t}w_{\nu}(t, L)dt&=\int_0^{+\infty} e^{-i\omega t} (-w_t(t, L)+ b_0(t))dt\\
&= -i\omega U+ \int_0^{+\infty}e^{-i\omega t} b_0(t) dt.
\end{align*}
All of these relations are valid as long as $\Im  \omega< -C(V, \Omega)$.
In that same region for the parameter value $\omega$, direct calculation shows that $U$ can be characterised by an elliptic boundary value problem: 
\begin{gather}
\Delta U+ (V+\omega^2)U= -H \textrm{ in } \Omega,    \label{ellip-1}\\
i\omega U+ aU_{\nu}= B \textrm{ on } \partial \Omega,   \label{ellip-2}
\end{gather}
where $H$ and  $B$ are the Fourier transforms of $h_0$ and $b_0$ respectively:
\begin{gather}
H(\omega, x):= \int_0^{+\infty}e^{-i\omega t} h_0(t, x) dt, \label{ellip-3}\\
 B(\omega):=  \int_0^{+\infty}e^{-i\omega t} b_0(t) dt.  \label{ellip-4} 
\end{gather}
The preceding elliptic  boundary value problem  is  to be understood in the following variational sense:  $U\in H^1(\Omega)$ is a solution if for every $W\in H^1(\Omega)$ we have that 
\begin{equation}\label{eq:sesqform}
\int_{\Omega}\nabla U \nabla \overline{W}dx-\int_{\Omega}(V+\omega^2)U \overline{W}dx+ \frac{i\omega}{a}\int_{\partial \Omega}U \overline{W}d\sigma= \int_{\Omega}H \overline{W}dx+ \frac{1}{a}\int_{\partial \Omega}B\overline{W}d\sigma.  
\end{equation}

Consider now the unbounded operator $P_V(\omega)$ defined as follows,
\begin{align*}
P_V(\omega): D(P_V(\omega))\subset L^2_{rad}(\Omega)&\rightarrow L^2_{rad}(\Omega), \\
U&\mapsto (\Delta + V+\omega^2)U,   \notag
\end{align*}
with domain 
\begin{equation}
D(P_V(\omega)):= \{U\in H^2_{rad}(\Omega): i\omega U+ aU_{\nu}= 0\}. \notag
\end{equation}
 As in \cite[page 174]{Lagnese},  $P_V(\omega)$ is an $m$-sectorial operator associated with the sesquilinear form on the left hand side of \eqref{eq:sesqform}, which is sectorial for any $\omega\in \mathbb{C}$ according to  the following Lemma. Its proof can be found in Appendix \ref{App-A}.
\begin{lemma}\label{lem:sesqsect} For any $\omega\in \mathbb{C}$, we define the sesquilinear form $p_{\omega}$ as follows,
\begin{align*}
p_{\omega}: H_{rad}^1(\Omega)\times H_{rad}^1(\Omega)&\rightarrow \mathbb{C}, \\
(U, W)&\mapsto \int_{\Omega}\nabla U \nabla \overline{W}dx-\int_{\Omega}(V+\omega^2)U \overline{W}dx+ \frac{i\omega}{a}\int_{\partial \Omega}U \overline{W}d\sigma,
\end{align*}
where the same assumptions on the potential $V$ as before are in force. Then the sesquilinear form  $p_{\omega}$ is sectorial. 

\end{lemma}

We conclude (see e.g. \cite{Kato-book}) that the operator valued function $P_V(\omega)$ is {\it{resolvent-holomorphic}}, which means that for any $\lambda\in \mathbb{C}$, the operator valued function 
\[
\big(\lambda - P_V(\omega)\big)^{-1}
\]
is holomorphic in a sufficiently small neighborhood of any $\omega_0\in \mathbb{C}$ with the property that $\lambda\in \rho(\omega_0)$. In particular, this holds for $\lambda = 0$. 
\\
The resolvent being a compact operator (see. e. g. \cite{Lagnese}), we infer that the non-invertibility of $P_V(\omega)$ is equivalent to the existence of a non-trivial function in its kernel. The elementary Lemma~\ref{lem-underline} below (and proved in the appendix) implies that all such values of $\omega$ in the lower complex half plane for which $P_V(\omega)$ is not invertible lie on the imaginary axis,  and standard theory implies that they form a discrete set there. 
\begin{defi}\label{def:poles} 
We call $\omega\in \mathbb{C}$ a pole, provided $P_V(\omega)$  is not boundedly invertible on $L^2(\Omega)$. If $\omega$ is not a pole, we call it regular.   Moreover, a pole $\omega_0\in \mathbb{C}$ is called of order $n$ if  the operator valued function $(\omega- \omega_0)^n\big(P_V(\omega)\big)^{-1}$ is holomorphic around $\omega_0$ and $(\omega- \omega_0)^{n-1}\big(P_V(\omega)\big)^{-1}$ is not holomorphic around $\omega_0$.
\end{defi}

As we shall see later on, the first inequality of Lemma~\ref{lem-underline} implies that all the poles {\it{in the lower half plane}} are in fact simple(of order one).

\begin{remark}\label{rmk-spec}
Let us remark here that the poles coincide with the eigenvalues of the following unbounded operator $\mathcal{A}$ defined on the Hilbert space $\mathcal{H}^1$, which is the infinitesimal generator of the semigroup associated to the damped wave equation: 
\begin{align*}
\mathcal{A}: D(\mathcal{A})\subset \mathcal{H}^1&\rightarrow \mathcal{H}^1,\\
\begin{pmatrix}
u\\v
\end{pmatrix}
&\mapsto 
\begin{pmatrix}
0 & 1\\
\Delta+ V & 0
\end{pmatrix}
\begin{pmatrix}
u\\v
\end{pmatrix},
\end{align*}
with 
\begin{gather*}
D(\mathcal{A}):=  \left\{\begin{pmatrix}
u\\v
\end{pmatrix}
\in H^2_{rad}(\Omega)\times H^1_{rad}(\Omega): v+ a u_{\nu}= 0 \textrm{ on } \partial \Omega \right\}.
\end{gather*}
Since the operator $\mathcal{A}$ has compact resolvent, its spectrum reduces to the set of eigenvalues. 
Suppose that for some $\omega\in \mathbb{C}$ and $(w, w_t)^T\in H^2(\Omega)\times H^1(\Omega)$ there is 
\begin{equation*}
\mathcal{A} \begin{pmatrix}
w\\w_t
\end{pmatrix}= 
\omega \begin{pmatrix}
w\\w_t
\end{pmatrix},
\end{equation*}
namely, 
\begin{gather*}
\dot{w}= w_t= i\omega w  \textrm{ in } \Omega, \\
\dot{w_t}= w_{tt}=\Delta w+ Vw= i\omega w_t  \textrm{ in } \Omega,\\
w_t+ aw_{\nu}=0 \textrm{ on } \partial \Omega,
\end{gather*} 
then simple calculation yields
\begin{gather*}
\Delta w+ Vw+ \omega^2 w=0  \textrm{ in } \Omega,\\
i\omega w+ aw_{\nu}=0 \textrm{ on } \partial \Omega.
\end{gather*}
This observation  automatically gives the criterion for the stability of the linearized  equation: stable if and only if there is no pole in the lower half  plane.
\end{remark}

\subsection{Holomorphic extension.}\label{subsec-str}
From the preceding section, assuming that the function $w$ satisfies \eqref{wave-op-1} - \eqref{wave-op-3}, for some  $b_0(t)$ compactly supported and $h_0$ bounded\footnote{In fact later we make much stronger exponential decay assumptions on $h_0$.} in $H^1(\Omega)$, say, we know that, at least for $\Im \omega< -C(V, \Omega)$ the function$\big(U(\omega, x), H(\omega, x), B(\omega, x)\big)$, defined via \eqref{eq:Uomega}, \eqref{ellip-3}, \eqref{ellip-4}, respectively, is well-defined in $H^1\times H^1\times \mathbb{C}$, and holomorphic with respect to $\omega$, and its components solve \eqref{ellip-1}--\eqref{ellip-2}.  As illustrated in Section \ref{subsec-strategy}, the relations \eqref{four-inverse}--\eqref{ellip-4} remain valid for $\beta\leq \beta_0$ provided that  there is a holomorphic (with respect to $\omega$) function $U(\omega, x)$ solving  \eqref{ellip-1}-\eqref{ellip-2} in the region $\Im \omega\leq \beta_0$.      Now we intend to holomorphically extend the solution of this elliptic equation to more general complex values $\omega$, and more specifically, we strive to extend it up to the line $\Im \omega=\beta_0$ with some $\beta_0$ positive. This will imply the exponential stability of the system.

However, it is not always possible to analytically extend the solution of the elliptic problem \eqref{ellip-1}--\eqref{ellip-2}, due to the presence of poles.  Indeed,  suppose that  $\omega_0$ is a pole, $i. e. $ the preceding operator $P_V(\omega_0)$ is not invertible, and   there exists some  non-trivial  solution $U_1$ of the following equation, 
 \begin{gather*}
\Delta U_1+ (V+\omega_0^2)U_1= 0  \textrm{ in } \Omega,\\
i\omega_0 U_1+ aU_{1\nu}= 0 \textrm{ on } \partial \Omega.
\end{gather*}
The equations \eqref{ellip-1}--\eqref{ellip-2} may not even admit a solution at this pole $\omega_0$. Indeed,  if  $W$ is a radial function on $\Omega$ satisfying 
\begin{equation}
i \omega W+ a W_{\nu}= B \textrm{ on } \partial \Omega,
\end{equation}
then $U$ being a solution of \eqref{ellip-1}--\eqref{ellip-2} is equivalent to $X:= U-W$ being a solution of 
\begin{gather}
\Delta X+ (V+\omega_0^2)X= Y \textrm{ in } \Omega,    \\
i\omega_0 X+ a X_{\nu}= 0 \textrm{ on } \partial \Omega, \\
Y:= -H- \Delta W+ (V+\omega_0^2)W.
\end{gather}
The preceding equation need not admit a solution since $P_V(\omega_0)$ is not invertible.   On the other hand, there is a subset $\mathcal{I}(\omega_0)$ of $L^2_{rad}(\Omega)$ such that for any $Y$ belonging to $\mathcal{I}(\omega_0)$, the preceding equation admits a solution $X$ (but the solution will not be unique).\\

Observe from the preceding equation that  $W$ only depends on the value of $b(t)$ (which determines $B$), which is exactly the control term.  Therefore it is natural and reasonable to ask whether there is a good choice of $b(t)$ such that for any pole $\omega_0$  the  above defined  function $Y$ takes its value from the set $\mathcal{I}(\omega_0)$.  This guarantees the existence of a solution for Equations \eqref{ellip-1}--\eqref{ellip-2}  for any pole $\omega_0$. In fact, the precise condition required is given by (iv) in the {\bf{Assumption 1}} on $b$ below. It remains to show that under this assumption there is a holomorphic extension for $U(\omega, x)$ beyond any pole. This will be done in Lemma \ref{lem-holoextension}.  Note that the motivation for (iv) below is that if $\omega_0$ is a pole, then for any given radial function $H$,  all  radial $H^1$-solutions $U$  of the inhomogenous problem 
\begin{equation}\label{ob--1}
\Delta U+ (V+\omega_0^2)U= -H, \textrm{ in } \Omega,
\end{equation}
have the same boundary value
\begin{equation}\label{ob--2}
i\omega_0 U+ aU_{\nu} \textrm{ on }\partial \Omega.   
\end{equation}

We now specify the precise technical assumptions we shall make on $b$. An actual construction of a function $b$ satisfying these conditions, and in the context of a nonlinear iterative scheme, will be accomplished in Lemma~\ref{lem:bconstruction}.\\

\noindent {\bf Assumption 1:} 
We say that a function $b(t)$ satisfies this assumption if the following conditions  on $b_0(t)$ defined as $\chi b+ \chi_t u$ hold: 
\begin{itemize}
\item[(i)] Compatibility condition:
\begin{equation*}
b_0(t)= \chi'(t)u(t, L) \textrm{ on } [0, 2].
\end{equation*} 
(Remark:  this condition comes from the assumption \eqref{assum-b-02} on $b(t)$).
\item[(ii)] Real valued condition: $b_0(t)$ is \textit{real} valed. \\
(Remark: notice that for any pole $\omega= \alpha+ i\beta$ with $\beta>0$ the value of $\alpha$ is not necessarily  0, which may cause some difficulty for finding real controls $b$, but this is resolved by the fact that  $\omega$ and $-\bar \omega$ appear in pairs).
\item[(iii)]  Support property: $b_0(t)$ is compactly supported \\ 
(Remark: this condition is imposed  to ensure  that $B$ is always defined, as well as the decay).
\item[(iv)]  Compatibility on poles: for any simple pole $\omega_j$ satisfying $\Im \omega_j<\beta$, we have that $B(\omega_j)$ coincides with the value of $i\omega_j U+ aU_{\nu}$ for any one (and hence for all solutions according to the observation on  \eqref{ob--1}--\eqref{ob--2}) of the  radial $H^1$-solutions $U$ of the elliptic equation
\begin{equation}
\Delta U+ (V+\omega_j^2)U= -H(\omega_j)  \textrm{ in } \Omega. \notag
\end{equation} 
\end{itemize}
Indeed, once a function $b_0(t)$ that satisfies such four properties  is chosen, one shall select $b(t)$ as follows
\begin{equation}
b(t):= 
\begin{cases}
0, \textrm{ if } t\in (0, 2), \\
b_0(t), \textrm{ if } t\in [2, +\infty).
\end{cases}
\end{equation}
 \\
 
Below we shall see that if we assume that there is no pole on the line $\Im \omega=\beta_0$ and the above assumptions hold, and furthermore all poles are simple,  then the solution $U = U(\omega, x)$ of the above elliptic problem can be holomorphically extended to a neighbourhood of the domain $\Im \omega\leq \beta_0$. Again, we emphasize   that when $\omega= \omega_j$ is a pole, the problem
 \begin{gather*}
\Delta U_1+ (V+\omega_j^2)U_1= -H(\omega_j) \textrm{ in } \Omega,\\
i\omega_j U_1+ aU_{1\nu}= 0 \textrm{ on } \partial \Omega
\end{gather*}
may not  admit any solution. This can also be compared to a similar issue for the simple ODE model: 
\begin{equation}
u_{xx}=f(x), 
u(0)=0, u(\pi)=0,    \notag
\end{equation}
which does not necessarily admit a $C^2$ solution. The preceding strategy shall be rendered rigorous below in Lemma \ref{lem-holoextension}, and the precise choice of the control term  $b_0$ that satisfies Assumption 1 shall be found  in Section \ref{sec:constructionfeedback}.\\

 We now present the details, starting with some properties concerning the kernel of $P_V(\omega_0)$:  any potential pole that lies in the lower half plane can only appear on the imaginary axis, moreover, these poles cannot accumulate in this region, and are simple. The following lemmas closely mimic ones in \cite{Lagnese}, and the proofs can be found in Appendix \ref{App-A}:
\begin{lemma}\label{lem-underline}
The point $\omega= \alpha+ i\beta$ is regular  if  $\beta\leq 0, \alpha\neq 0$.  Furthermore, for $\alpha\neq 0$ and $\beta\in (-C(V, \Omega), 0)$ the unique solution of 
\begin{gather}
\Delta U+ (V+\omega^2)U= -H \textrm{ in } \Omega, \label{eq-lem-1}\\
i\omega U+ aU_{\nu}= 0 \textrm{ on } \partial \Omega, \label{eq-lem-2}
\end{gather}
satisfies
\begin{equation}
\lVert U \lVert_{L^2(\Omega)}^2+|\beta|^{-1} \lVert U \lVert_{L^2(\partial\Omega)}^2\lesssim \frac{1}{|\alpha\beta|^2}\lVert H \lVert_{L^2(\Omega)}^2,    \notag
\end{equation}
and 
\begin{equation}
\lVert U\lVert_{H^1(\Omega)}+ \lVert U \lVert_{L^2(\partial\Omega)}\lesssim \lVert U\lVert_{L^2(\Omega)}.    \notag
\end{equation}     
\end{lemma}
We also record the following basic
\begin{lemma}\label{lem-fin}
There are finitely many poles $\omega_j= i\beta_j$ with $-\beta_j\in (0, C(V, \Omega))$, and they  are simple poles. 
\end{lemma}
In fact, the simplicity of the poles in the lower half plane is a direct consequence of Lemma~\ref{lem-underline}.

Armed with the preceding lemmas and a good choice of $b(t)$, we can holomorphically extend $U$, interpreted as function of  $\omega$, to all of  the lower half plane or even further.
\begin{lemma}\label{lem-holoextension}
 Suppose that there is only a finite number of simple poles\footnote{In particular, we assume there are no higher order poles} under the line $\Im \omega=\beta_0$.
Assume that $H(\omega, x)$ depends holomorphically on $\omega$ in a neighbourhood of the domain $\Im \omega\leq \beta_0$. Then,  for any fixed $b_0(t)$ satisfying Assumption 1, the unique function $U(\omega,x)$ which solves \eqref{ellip-1}--\eqref{ellip-2} for regular points $\omega$  can be  holomorphically extended to all $\omega$  in the domain $\Im \omega\leq \beta_0$, thus also including the poles, namely, $U(\omega, x)$ is holomorphic in this domain and solves Equations  \eqref{ellip-1}--\eqref{ellip-2}.
\end{lemma}
\begin{proof}
Thanks to the choice of $b(t)$, we know that for any regular $\omega$ such that $\Im \omega\leq \beta$ the elliptic equation \eqref{ellip-1}--\eqref{ellip-2} admits a unique solution $U(\omega, x)$. It suffices to show that $U(\omega)$ can be holomorphically extended to any simple pole $\omega_j$.  Since $\omega_j$ is in the spectrum,  we can expand the resolvent   $(\Delta+V+ (\omega_j+ \omega-\omega_j)^2)^{-1}$ with the boundary condition  $i\omega U+ aU_{\nu}= 0$ as 
\begin{equation}
(\omega-\omega_j)^{-1}A+ Q(\omega),  \notag
\end{equation}
where the image of $A$ is of one dimension (the kernel of the operator for $\omega_j$),  and the regular part $Q(\omega)$ is bounded and holomorphic near $\omega_j$.

Now let $U(\omega_j)$ be any one $H^1$-solution of the problem 
\[
\Delta U+ (V+\omega_j^2)U= -H(\omega_j).
\]
According to (iv) of the basic conditions for $b$, we then have 
\[
i\omega_j U(\omega_j)+ aU_{\nu}(\omega_j) = B(\omega_j).
\]
For any $\omega$ close to $\omega_j$, we set the difference 
\begin{equation}
\Delta_{\omega}:= U(\omega)-U(\omega_j),   \notag
\end{equation}
which satisfies
\begin{gather*}
(\Delta + V+\omega^2)\Delta_{\omega}=-(\omega^2-\omega_j^2)U(\omega_j)-H(\omega)+H(\omega_j), \textrm{ in } \Omega,\\
i\omega \Delta_{\omega}+ a(\Delta_{\omega})_{\nu}= i(\omega-\omega_j)U(\omega_j, x)+ B(\omega)-B(\omega_j), \textrm{ on } \partial \Omega.
\end{gather*}
As usual we split $\Delta_{\omega}$ by $V_{\omega}+ W$ with $W$ given by
\begin{equation}
W(\omega, x):= g(x)\big(i(\omega-\omega_j)U(\omega_j, x)|_{\partial \Omega}+ B(\omega)-B(\omega_j)\big),    \notag
\end{equation}
where the smooth function $g(x)$ verifies $g|_{\partial\Omega}(L)=0, g'|_{\partial\Omega}(L)=a^{-1}$.  Thus $V_{\omega}$ satisfies
\begin{align*}
(\Delta + V+\omega^2)V_{\omega}&=-(\omega^2-\omega_j^2)U(\omega_j)-H(\omega)+H(\omega_j)\\
&\;\;\; \;-\big(i(\omega-\omega_j)U(\omega_j, L)+ B(\omega)-B(\omega_j)\big)(\Delta + V+\omega^2)g(x),\\
&= (\omega-\omega_j)R(\omega),\\
&\;\;\;\,i\omega V_{\omega}+ a(V_{\omega})_{\nu}=0,
\end{align*}
where $R(\omega)$ is holomorphic near $\omega_j$. Thanks to the resolvent expression of  $(\Delta+V+ \omega^2)^{-1}$, we know that 
\begin{equation}
V_{\omega}= AR(\omega)+ Q(\omega)(\omega-\omega_j)R(\omega),    \notag
\end{equation}
which is of course holomorphic around $\omega_j$.
\end{proof}

\begin{remark}
The lemmas stated in this section actually hold beyond the radial context. 
\end{remark}
According to Lemma \ref{lem-underline}, we only have information about poles that are below the real line.  However, in the  next section for a specific potential function $V=1$ with the help of  explicit calculations we will be able to extend the function $U(\omega, x)$ significantly above the real line. In particular, the preceding construction can be extended to any strip in the upper half-plane where only finitely many poles exist, under a suitable non-degeneracy condition.

\section{Stabilization around unstable equilibria}\label{sec-stab-li}
This section is devoted to the stabilization of the system around the static solution $u\equiv 1$, more precisely, the stabilization of  its linearized system, 
\begin{gather*}
\Box u-2 u=h, \textrm{ in } \Omega,\\
u_t+ au_{\nu}=b(t), \textrm{ on } \partial\Omega.
\end{gather*}
Let us perform the following simple change of variables 
\begin{equation}
\overline{u}(t, x):= u\left(\frac{t}{\sqrt{2}},\; \frac{x}{\sqrt{2}}\right), \overline{b}(t):= \frac{1}{\sqrt{2}}b\left(\frac{t}{\sqrt{2}}\right), \; \overline{\Omega}:= B_L(0).    \notag
\end{equation}
For ease of notations, if there is no risk of confusion we still denote the new variables by $u$ and $f$, which therefore verify
\begin{gather*}
\Box u- u=h, \textrm{ in } \overline{\Omega},\\
u_t+ au_{\nu}=b(t), \textrm{ on } \partial\overline{\Omega}, \\
u(0, x)=u_0, u_t(0, x)=v_0.
\end{gather*}
For any given $\beta\in\mathbb{R}$, we define and works with the following space for the triple $(h, u_0, v_0)$,
\begin{gather}
\mathcal{B}= \mathcal{B}(\beta):= \{(h, u_0, v_0): h(t, x)\in L^1(0,+\infty; L^2(\overline{\Omega} )), (u_0, v_0)\in \mathcal{H}^1\},    
\end{gather}
with its norm given by 
\begin{gather}
\lVert (h, u_0, v_0)\lVert_{\mathcal{B}}:= \int_0^{+\infty} e^{\beta t}\lVert h(t, \cdot)\lVert_{L^2(\overline{\Omega} )} dt+ \lVert (u_0, v_0)\lVert_{\mathcal{H}^1}.\label{norm-B}
\end{gather}

In the end of this section, we will  achieve the theorem that we present in the following.   Remark that  this result only deals with  the case that  $a\in (0, 1)\setminus \mathcal{A}(L)$ (see \eqref{def:AL} for the precise definition of this set that contains countably many numbers), while when $a$ takes its value from $\mathcal{A}(L)$ the same argument leads to  exponential stabilization with decay rate $\beta_*$, for some number  $\beta_*$ that is sufficiently close to 0 such that there is no pole in the strip $\Im \omega\in [0, \beta_*]$ (as stated in Theorem \ref{thm-stab}).  
\begin{theorem}\label{thm-li-key}
Let $L>0$ such that  $L\neq \tan L$.  Let $a\in (0, 1)\setminus \mathcal{A}(L)$ (see \eqref{def:AL} for a precise definition).   For any $\beta \in (0,  \frac{1}{2L} \log{\frac{1+a}{1-a}})$, there exists a constant $C_{\beta}$ effectively computable such that, for any radial triple $(h, u_0, v_0)\in \mathcal{B}$ we are able to find a smooth real function $b(t)$ compactly supported in  interval $(2, 4)$ satisfying 
\begin{equation}
|b(t)|, |b'(t)|\leq C_{\beta} \lVert (h, u_0, v_0)\lVert_{\mathcal{B}}       \notag
\end{equation}
 such that the solution $u$ verifies
 \begin{align*}
\lVert u(t)\lVert_{\mathcal{H}^1}&\leq C_{\beta}e^{-\beta t}\left(\int_0^{t} \lVert h(t, \cdot)\lVert_{L^2(\overline{\Omega} )} dt+ \lVert (u_0, v_0)\lVert_{\mathcal{H}^1}\right),  \textrm{ for } t\leq 2,\\
\lVert u(t)\lVert_{\mathcal{H}^1}&\leq C_{\beta}e^{-\beta t}\left(\int_0^{+\infty} e^{\beta t}\lVert h(t, \cdot)\lVert_{L^2(\overline{\Omega} )} dt+ \lVert (u_0, v_0)\lVert_{\mathcal{H}^1}\right), \textrm{ for } t> 2.
\end{align*}
Moreover, the real function $b(t)$ is in fact chosen from a fixed finite dimensional space $\mathcal{N}:= \textrm{span } \{b_1(t),..., b_N(t)\}$
\begin{equation}
b(t)= \sum_{k=1}^N l_k b_k(t)     \notag
\end{equation}
with $l_k= l_k(h, u_0, v_0)\in \mathbb{R}$  depending linearly and continuously  on $(h, u_0, v_0)\in \mathcal{B}$:
\begin{equation}
 \sum_{k=1}^N |l_k(h, u_0, v_0)|\leq C_{\beta} \lVert (h, u_0, v_0)\lVert_{\mathcal{B}}.     \notag
\end{equation}
\end{theorem}

As it  corresponds to the case of $V\equiv1$ in the preceding section, we  study  the operator $P(\omega): = P_1(\omega)$ that  we introduced in \eqref{ellip-1}--\eqref{ellip-4}: for any function  $U\in D(P_1(\omega))$,
\begin{gather*}
P_{\omega}U:= (\Delta + 1+ \omega^2) U \; \textrm{ in } \overline\Omega, \\
i\omega U+ a U_{\nu}=0 \; \textrm{ on } \partial \overline\Omega.
\end{gather*}
Since we are working with   radial functions, it is standard to consider the following function on $r\in [0, L]$:
\begin{equation}
 \psi(r):= rU(r).
\end{equation}
Then, the preceding operator on $U$ becomes the following one acting on $\psi$:
\begin{align}
Q(\omega): D(Q(\omega))\subset L^2(0, L)&\rightarrow L^2(0, L), \notag\\
\psi&\mapsto \partial_r^2 \psi + (1+ \omega^2)\psi, \label{op-1} 
\end{align}
with its domain of definition as
\begin{gather}
D(Q(\omega)):= \{\psi\in H^2(0, L): \psi(0)=0, \; (iL\omega-a)\psi(L)+ aL \psi'(L)=0\}. \label{op-2}
\end{gather}\\

We again observe that all solutions of  \eqref{ob--1} have the same boundary value \eqref{ob--2} if $\omega_0$ is a pole: suppose that $\omega_0$ is a pole and that Equation \eqref{ob--1} has two different solutions $U_1$ and $U_2$ having different boundary behaviors, then $U:= U_1- U_2$ is a solution of Equation \eqref{ob--1} with $H=0$. Thus $\psi:= rU$ satisfies 
\begin{gather*}
 \psi_{rr} + (1+ \omega_0^2)\psi=0, \\
\psi(0)=0, \; (iL\omega_0-a)\psi(L)+ aL \psi'(L)\neq 0.
\end{gather*}
Therefore, the unique solution of the second order ODE 
\begin{equation*}
\partial_r^2 \psi + (1+ \omega_0^2)\psi=0, \; \psi(0)=0, \psi_r(0)= 1,
\end{equation*}
satisfies 
\begin{equation*}
(iL\omega_0- a)\psi(L)+ aL \psi'(L)\neq 0.
\end{equation*}
This is in contradiction with the assumption that $\omega_0$ is a pole.

\subsection{Spectral properties}\label{subsec-spectral}
We want to find those non trivial pairs $(\omega, \psi)\in  \mathbb{C}\times D(Q(\omega)$ that satisfy  $Q(\omega)\psi= 0$. Observe that this is related to the following unbounded operator $A$ defined on the Hilbert space $L^2(0, L)$,
\begin{gather*}
A: D(A)\subset L^2(0, L)\rightarrow L^2(0, L), \\
\begin{aligned}
A\begin{pmatrix}\psi_1\\ \psi_2 \end{pmatrix}= \begin{pmatrix} & 1\\ \partial_r^2+ 1& \end{pmatrix}\begin{pmatrix}\psi_1\\ \psi_2 \end{pmatrix},
\end{aligned}
\end{gather*}
with its domain of definition given by 
\begin{equation}
D(A):= \{(\psi_1, \psi_2)^T\in H^2(0, L)\times H^1(0, L): \psi_2(0)=0,\; -a \psi_1(L)+ aL \psi_{1}'(L)= -L \psi_2(L)    \}. \notag
\end{equation}
Note that the spectrum of $A$ coincides the spectrum of the operator $\mathcal{A}$ that is defined in Remark \ref{rmk-spec}.

Since the operator $A$ has real coefficients, its eigenvalues  appear in pairs.  
Suppose that $(i\omega, \psi_1, \psi_2)$ is a pair of eigen modes for the preceding operator, then we easily get from its definition  that
\begin{gather*}
\psi_2= i\omega \psi_1,\; \textrm{ in } \overline\Omega\\
(\partial_r^2+ 1)\psi_1= i\omega \psi_2= -\omega^2 \psi_1, \; \textrm{ in } \overline\Omega\\
\psi_1(0)= \psi_2(0)=0,\\
-a \psi_1(L)+ aL \psi_{1}'(L)= -L \psi_2(L)= -iL\omega \psi_1(L)
\end{gather*} 
Hence, $(\omega, \psi_1)$ is a pair constituting a non trivial solution of \eqref{op-1}--\eqref{op-2}.  It suffices to find all eigenvalues of $A$; in fact, we are more interested in their asymptotic behavior, $i.e.$ when $k$ tends to $\infty$, the value of $i\omega_k$.  

Asymptotically, formally neglecting lower order terms, we reduce to study the simpler operator $\tilde{A}$,
\begin{equation*}
\begin{aligned}
\tilde{A}\begin{pmatrix}\psi_1\\ \psi_2 \end{pmatrix}= \begin{pmatrix} & 1\\ \partial_r^2& \end{pmatrix}\begin{pmatrix}\psi_1\\ \psi_2 \end{pmatrix},
\end{aligned}
\end{equation*}
\begin{equation}
\psi_2(0)=0,\;  a \psi_{1}'(L)= - \psi_2(L),     \notag
\end{equation}
whose spectrum is explicit: 
\begin{equation}
i\tilde \omega_k= \frac{1}{2L} \log \frac{1-a}{1+a}+ i \frac{k \pi}{L},\;  k\in \mathbb{Z},      \notag
\end{equation}
thus
\begin{equation}
\tilde \omega_k= i \frac{1}{2L} \log \frac{1+a}{1-a}+  \frac{k \pi}{L}, \;  k\in \mathbb{Z}.   \notag
\end{equation}
Hence, the following lemma is quite natural. 
In fact, exploiting the radiality assumption, we are able to sharpen Lemma \ref{lem-underline} and Lemma \ref{lem-fin} in the preceding section as follows.
\begin{lemma}\label{lem-eigen}
In the radial case, the operator $P(\omega)$ is  invertible  except on a discrete  set $\{\omega_k\}$, 
\begin{itemize}
\item[(i)]  there exist only finitely many poles $\omega_k$ such that $\textrm{Im}(\omega_k)<0$, and these poles are simple and purely imaginary.  Moreover, if   $\omega_k$ is a pole  then so is $-\bar \omega_k$;
\item[(ii)] there is no pole $\omega_k$ on the  real line if $L\neq \tan L$;
\item[(iii)]  for any $\beta<\frac{1}{2L} \log \frac{1+a}{1-a}$, there are only finitely many poles which are under the line $\Im \omega=\beta$. Indeed, $\beta= \frac{1}{2L} \log \frac{1+a}{1-a}$ is the asymptotic line for poles.
\\
In particular, there exists $\gamma> 0$ such that in the strip $0< \textrm{Im}(\omega)\leq \gamma$ there is no $\omega_k$.
\item[(iv)]  for any $L$ such that $\tan L\neq L$, there exists a set $\mathcal{A}(L)$ containing at most countably many elements such that for any $a\in (0, 1)\setminus \mathcal{A}(L)$ there are only simple poles. 
\end{itemize}
\end{lemma}

\begin{proof}
The first property and (i) follow from Lemma \ref{lem-underline} and Lemma \ref{lem-fin}. We continue with the proof of properties $(ii)$, $(iii)$ and  $(iv)$.

Recall that a complex value $\omega$ being a  pole is equivalent to the existence of a non-trivial solution of the equation 
\begin{gather}
\partial_r^2 \psi + (1+ \omega^2)\psi= 0,  \label{eq:nt1} \\
\psi(0)=0, \; (iL\omega-a)\psi(L)+ aL \psi'(L)=0. \label{eq:nt2}
\end{gather}
From now on we shall denote 
$\langle\omega\rangle^2:= \omega^2+1$, and define $\langle\omega\rangle$ as a square root of $(\omega^2+ 1)$ as follows: if $\omega\notin [-i,i]$, let $\langle\omega\rangle$ be the unique square root of $(\omega^2+ 1)$ closest to $\omega$. Furthermore, for $\omega\in [-i,i]$, given by $\omega = i\sin \rho$, $\rho\in \R$, set $\langle\omega\rangle = |\cos\rho|$. Then $\big|\omega - \langle\omega\rangle\big|\leq 1$. Also, for $\omega\notin [-i, i]$, we have 
\[
\langle-\bar{\omega}\rangle = -\overline{\langle\omega\rangle}.
\]
Now  assume\footnote{It is easily seen that $\omega = \pm i$ are not poles}  $\omega\neq \pm i$, and
 notice that the solutions of
\begin{gather*}
\partial_r^2 \psi + (1+ \omega^2)\psi= 0, \; 
\psi(0)=0
\end{gather*}
are given by $C \psi$ with 
\begin{equation*}
\psi_1(r):= e^{i \langle\omega\rangle r}- e^{-i \langle\omega\rangle r}.
\end{equation*}
Therefore,  \eqref{eq:nt1}--\eqref{eq:nt2} admits a non-trivial solution if and only if 
\begin{equation*}
(iL\omega-a)\psi_1(L)+ aL \psi_1'(L)=0,
\end{equation*}
which is further equivalent to 
\begin{equation}
 (iL \omega-a)  \left(e^{i \langle\omega\rangle L}- e^{-i \langle\omega\rangle L}\right) +  i aL\langle\omega\rangle \left(e^{i \langle\omega\rangle L}+ e^{-i \langle\omega\rangle L}\right)=0. \notag
\end{equation}
This implies that $\omega\in \mathbb{C}\setminus\{i, -i\}$ is a pole \textit{if and only if}
\begin{equation}\label{eq-omega-eig}
e^{2i\langle \omega\rangle L}= \frac{iL\omega-a-iaL\langle \omega\rangle}{iL\omega-a+iaL\langle \omega\rangle},
\end{equation}
or equivalently, 
\begin{equation}\label{eq-omega-eig-2}
(iL\omega- a)\sin (\langle \omega\rangle L)+ aL \langle \omega\rangle \cos (\langle \omega\rangle L)=0.
\end{equation}
Setting $\omega = 0$, we see that the condition $L\neq \tan L$ implies that it is not a pole, as asserted in (ii). 
\\
Next, we prove (iii). 
At first we show that there is a  uniform upper and lower bound for the imaginary part of poles.  
Assume  that $\omega= \alpha+ i\beta$ is a pole such that $|\omega|>N$ for some $N$ large enough to be chosen later on.  Since  $|\omega- \langle \omega\rangle|\leq 1$, we know that on the left-hand side of  Equation \eqref{eq-omega-eig} 
\begin{equation}\label{eq:LHS:1}
e^{-2L(1+ \beta)}\leq  |e^{2i\langle \omega\rangle L}|\leq e^{2L(1- \beta)}.
\end{equation}
Moreover, by the definition of $\langle \omega\rangle$ there is 
\begin{equation*}
\frac{\langle \omega\rangle}{\omega}= 1+ r_{\omega}  \textrm{ with } |r_{\omega}|\leq \frac{1}{N}.
\end{equation*}
Hence the right-hand side of \eqref{eq-omega-eig}  is equivalent to 
\begin{equation}
\frac{iL\omega-a-iaL\langle \omega\rangle}{iL\omega-a+iaL\langle \omega\rangle}= \frac{iL \left(1- a \frac{\langle \omega\rangle}{\omega}\right)- \frac{a}{\omega}}{iL \left(1+ a \frac{\langle \omega\rangle}{\omega}\right)- \frac{a}{\omega}}= \frac{iL \left((1-a)- a r_{\omega}\right)- \frac{a}{\omega}}{iL \left((1+ a)+ a r_{\omega}\right)- \frac{a}{\omega}}.
\end{equation}
Inspired by the preceding formula we shall choose $N:= N_{a, L}$ as
\begin{equation}
N_{a, L}:= \frac{2a(1+L)}{(1-a)L}.
\end{equation}
Such a  choice of $N_{a, L}$ indeed guarantees 
\begin{align*}
|iL \left((1-a)- a r_{\omega}\right)- \frac{a}{\omega}|&\geq L(1- a)- \frac{aL}{N_{a, L}}- \frac{a}{N_{a, L}}\geq \frac{L(1- a)}{2}, \\
|iL \left((1+a)+ a r_{\omega}\right)- \frac{a}{\omega}|&\geq L(1+ a)- \frac{aL}{N_{a, L}}- \frac{a}{N_{a, L}}\geq \frac{L(1+ 3a)}{2}.
\end{align*}
Thus
\begin{equation}
\frac{L(1- a)}{2L(1+ 2a)+ 2a} \leq \left|\frac{iL\omega-a-iaL\langle \omega\rangle}{iL\omega-a+iaL\langle \omega\rangle}\right|\leq \frac{2L(1+ 2a)+ 2a}{L(1+ 3a)},
\end{equation}
which combined with \eqref{eq:LHS:1}, yields an effectively computable constant $C_{a, L}$ such that  
\begin{equation}
|\beta|\leq C_{a, L} \textrm{ for any pole $\omega$ satisfying } |\omega|\geq N_{a, L}.
\end{equation}
Consequently, there exists some constant $C_U>0$ such that 
\begin{equation}
|\beta|\leq C_{U} \textrm{ for any pole $\omega\in \mathbb{C}$. }
\end{equation}
\\

Next  we give an asymptotic expansion for poles such that $\alpha \geq N_{a, L}$. Keep in mind that for any pole $\omega= \alpha+ i\beta$ which satisfies this condition, its imaginary part is smaller than $C_{a, L}$.   By writing $\langle\omega\rangle= \tilde{\alpha}+ i\tilde{\beta}$, one easily gets\footnote{Recall our choice of $\langle\omega\rangle$.}
\begin{gather*}
\tilde{\alpha}= \alpha+ \frac{1}{2\alpha}+ O(\frac{1}{\alpha^2}), \; \tilde{\beta}= \beta+ O(\frac{1}{\alpha^2}),\;
\langle\omega\rangle= \omega+ \frac{1}{2\alpha}+O(\frac{1}{\alpha^2}), \\
e^{i \langle\omega\rangle r}= e^{i \omega r}+ e^{i \omega r}O(\frac{1}{\alpha})= e^{i \omega r}+ O(\frac{1}{\alpha}).
\end{gather*} 
Since $\langle\omega\rangle= \alpha+ \frac{1}{2\alpha}+ i\beta+ O(\frac{1}{\alpha^2})$ we further get
\begin{align}\label{as:es:om}
e^{i\langle\omega\rangle r}= e^{i (\alpha+ \frac{1}{2\alpha}+ i\beta+ O(\frac{1}{\alpha^2}))r}
= e^{i \omega r}+ e^{i \omega r} \frac{i r}{2\alpha}+ \mathbf{r}(r, \alpha)
\end{align}
with 
\begin{equation}
|\mathbf{r}(r, \alpha)|\lesssim \frac{r}{\alpha^2}, \; \forall r\in [0, 2L], \alpha\geq A.      \notag
\end{equation} 

Thanks to \eqref{as:es:om}, the left-hand side of \eqref{eq-omega-eig} is 
\begin{align}
e^{2i\langle\omega\rangle L}= e^{2iL (\alpha+ \frac{1}{2\alpha}+ i\beta)}+ O(\frac{1}{\alpha^2}).
\end{align}
On the other hand, since  $\langle \omega\rangle/ \omega= 1+ O(1/\alpha^2)$,
 the right-hand side verifies 
 \begin{equation}
  \frac{iL\omega-a-iaL\langle \omega\rangle}{iL\omega-a+iaL\langle \omega\rangle}= \frac{iL(1- a)- \frac{a}{\alpha}}{iL(1+ a)- \frac{a}{\alpha}}+ O(\frac{1}{\alpha^2})= \frac{1-a}{1+a}+ O(\frac{1}{\alpha}).
 \end{equation}
  
 By combining the preceding two equations we conclude that 
 \begin{equation}\label{eq:compare}
 e^{2iL (\alpha+ \frac{1}{2\alpha})} e^{-2L\beta}=  \frac{1-a}{1+a}+ O(\frac{1}{\alpha}).
 \end{equation}
 By comparing the absolute values of both sides of the equation we immediately obtain
 \begin{equation}
 \beta=  \frac{1}{2L}\log \frac{1+a}{1-a}+ O(\frac{1}{\alpha}),
\end{equation}  
which implies (iii). 
\\

Finally, we present the proof of Property (iv). 
Notice from  Equation \eqref{eq-omega-eig-2} $\omega$ is a pole if and only if
 \begin{equation}\label{eq:ver-1}
S(a, \omega):= (iL\omega- a)\sin (\langle \omega\rangle L)+ aL \langle \omega\rangle \cos (\langle \omega\rangle L)=0,
\end{equation}
thus it is a double pole if and only if in addition 
the derivative $S_{\omega}(a, \omega)$ at $\omega$ satisfies
\begin{equation}
(i- aL \omega)\sin (\langle \omega\rangle L)+ \frac{iL \omega^2}{\langle \omega\rangle} \cos (\langle \omega\rangle L)=0.
\end{equation}
See also the definition of $c_g$ and $\eta$ by Equation \eqref{def:eta:fun} in the next subsection. 
Therefore, $\omega$ is a double pole if and only if it satisfies both \eqref{eq:ver-1} and 
\begin{equation}\label{eq:ver-2}
L(a^2-1)\omega^3- 2ia \omega^2+ a^2 L\omega- ai=0.
\end{equation}

Take $\omega= iy$ and deduce  that $\omega$ is a solution of \eqref{eq:ver-2} if and only if $y$ is a solution of 
\begin{equation*}
f(y):= (1- a^2)L y^3+ 2a y^2+ a^2 L y- a=0
\end{equation*}
Clearly for any fixed $a\in (0, 1)$ the preceding equation admits three complex valued solutions $y_1(a), y_2(a)$ and $y_3(a)$.  We first notice that $f(0)= -a<0$ and  that $f(1)= L+ a>0$, thus there exists at least one solution in the interval $(0, 1)$. Indeed, there is exactly one solution in this interval: suppose there are at least two solutions, then we deduce from the values of $f(0)$ and $f(1)$ that all the three  solutions should be inside $(0, 1)$, which is incompatible with the fact that 
\begin{equation*}
y_1+ y_2+ y_3= - \frac{2a}{(1-a^2)L}<0.
\end{equation*}
From now on we shall denote such a  unique solution in $(0, 1)$ by $y_1(a)$.   For any $a_0\in (0, 1)$ there exist a closed  ball around $a_0$, $B(a_0)\subset \mathbb{C}$, as well as a simply connected domain around $y_1(a_0)$, say  $D(a_0)\subset \mathbb{C}\backslash\{\pm 1\}$  such that  $y_1: (0, 1)\cap B(a_0)\rightarrow (0, 1)$ can be holomorphically extended to $y_1: B(a_0)\rightarrow D(a_0)$. \\
We claim that the other two solutions $y_2(a)$ and $y_3(a)$ both have negative real part. We know from the preceding equation that $\Re (y_2+ y_3)<0$ and $\Im (y_2+ y_3)=0$. If one of them is real valued, then both are real valued, and in this case  both are negative, since their product is positive. If $y_{2,3}$ are not real valued, then they are complex conjugate numbers having the same negative real part.  

Recall from Lemma \ref{lem-fin} that there are only  simple poles below the real axis; hence we know that $\omega_2:= i y_2(a)$ and $\omega_3:= i y_3(a)$ are not double poles. Consequently, $a\in (0, 1)$ admits double poles if and only if $\omega_1(a):= i y_1(a)$ satisfies Equation \eqref{eq:ver-1}. For any $a_0\in (0, 1)$ we shall define the following function on $B(a_0)$: 
\begin{equation*}
F(a):=  \langle i y_1(a)\rangle^{-1}\cdot S(a, i y_1(a)),
\end{equation*}
which is holomorphic by the choice of the simply connected domain $D(a_0)$. Therefore, there are at most finitely many $a$ inside $B(a_0)$  that admit double poles, and we shall define the set of these values of $a$ as $\mathcal{A}_{a_0}$.   

Finally the set $\mathcal{A}(L)$ is given by 
\begin{equation}\label{def:AL}
\mathcal{A}(L):= \bigcup\limits_{a_0\in (0, 1)} \mathcal{A}_{a_0}\cap (0, 1),
\end{equation}
which only admits countably many elements.  More concretely, some number $a\in (0, 1)$ belongs to $\mathcal{A}(L)$ if and only if it satisfies 
\begin{equation}
(L y_1^2(a)+ a)\sin \left(L \sqrt{1- y_1^2(a)}\right)- aL \sqrt{1- y_1^2(a)} \cos \left(L \sqrt{1- y_1^2(a)}\right),
\end{equation}
where $y_1(a)$ is the unique solution inside $(0, 1)$ of the cubic equation 
\begin{equation}
(1- a^2)L y^3+ 2a y^2+ a^2 L y- a=0.
\end{equation}

\end{proof}

\subsection{Asymptotic resolvent estimate: the inverse $R(\omega)= P(\omega)^{-1}$}\label{subsec-asym}  
In this Section we only work with $\omega\in\mathbb{C}$ such that $\Im \omega\in [-1, C_U+ 1]$ (recall from the preceding section that $C_U$ is the uniform upper bound of the imaginary part of poles). 
For $\omega= \alpha+ i\beta \notin \{\omega_k\}$, the operator $P(\omega)$ is invertible on $L^2_{rad}$.  One of the essential goals here is to characterize the resolvent asymptotically in terms of $\alpha$, which will be used later on in the following sections concerning energy estimates.   For $f(x)\in L^2_{rad}(\Omega)$, thus $rf(r)\in L^2$, and $R(\omega)f= U=r^{-1} \psi$ satisfies,
\begin{gather*}
\partial_r^2 \psi + (1+ \omega^2)\psi= rf(r), \\
\psi(0)=0, \; (iL\omega-a)\psi(L)+ aL \psi'(L)=0,
\end{gather*}
which can be obtained by solving Sturm--Liouville's problem with the help of Green's function:
\begin{equation}
R(\omega)f= r^{-1}\psi= r^{-1}\int_0^L \Gamma(r, s; \omega) sf(s) ds.     \notag
\end{equation}

Let us now concentrate on the  Green's function  which subject to the above ODE is of the form 
\begin{gather}
\Gamma(r, s):=
\begin{cases} \frac{1}{c_g} \phi_1(s) \phi_2(r), \textrm{ for } r\geq s, \\
\frac{1}{c_g} \phi_2(s) \phi_1(r),  \textrm{ for } r< s, 
\end{cases}
\end{gather}
where for $\omega\notin \{\pm i\}$
\begin{equation}\label{def:eta:fun}
\begin{cases}
\phi_1(r):= \frac{1}{2} \left(e^{i \langle\omega\rangle r}- e^{-i \langle\omega\rangle r} \right), \\
\phi_2(r):= \frac{1}{2} \left(e^{i \langle\omega\rangle r}- e^{-i \langle\omega\rangle r}\right) +  \frac{\eta}{2} \left(e^{i \langle\omega\rangle r}+ e^{-i \langle\omega\rangle r}  \right),\\
c_g:=\phi_1(0)\phi_2'(0)-  \phi_1'(0)\phi_2(0)= -i \langle\omega\rangle \eta,
\end{cases}
\end{equation}
with some value $\eta$ to be fixed later on such that $\phi_2(r)$ verifies the right-hand side boundary condition, namely, $(iL\omega-a)\phi_2(L)+ aL \phi_2'(L)=0$.
\begin{remark}
Let us remark here that $R(-\bar{\omega})=\overline{R(\omega)}$.  Indeed, by the definition of $\langle \omega\rangle$, we know that $\langle -\bar{\omega}\rangle=  -\overline{\langle \omega\rangle}$. Then successively, we have $\phi_i(-\bar{\omega}; r, s)= \overline{\phi_i}(\omega; r, s)$, $c_g(-\bar{\omega})= \overline{c_g}(\omega)$, $\Gamma(-\bar{\omega}; r, s)= \overline{\Gamma}(\omega; r, s)$, which implies the required conjugacy property.
\end{remark}
For ease of notations we only consider $\alpha\geq A$ for some $A$ big enough (actually it suffices to take $A=1$). 

\begin{remark}
In the following calculation we only deal with the case that  $\alpha\geq A$, for $\alpha\leq -A$ we also identify $\langle \omega\rangle$ with $\omega+ \frac{1}{2\alpha}+O(\frac{1}{\alpha^2})$ and all the  calculations that follows remain the same.    Moreover, throughout this paper we do not detail the case that $-A\leq \alpha\leq A$, since all the estimates that we will be dealing with are uniformly bounded on this compact interval.
\end{remark}

Let $\beta\in [-1, C_U+1]$. Since $\phi_2$ verifies the right-hand side boundary condition, we have 
\begin{align*}
0=& (iL \omega-a)\left( \left(e^{i \langle\omega\rangle L}- e^{-i \langle\omega\rangle L}\right) +  \eta \left(e^{i \langle\omega\rangle L}+ e^{-i \langle\omega\rangle L}  \right)\right)\\
&+ i aL\langle\omega\rangle\left( \left(e^{i \langle\omega\rangle L}+ e^{-i \langle\omega\rangle L}\right) +  \eta \left(e^{i \langle\omega\rangle L}- e^{-i \langle\omega\rangle L}  \right)\right),
\end{align*}
thus
\begin{equation}
(1- e^{i 2\omega L})- a(1+ e^{i 2\omega L})+ O(\frac{1}{\alpha})=\left((1+ e^{i 2\omega L}) -  a (1- e^{i 2\omega L})+ O(\frac{1}{\alpha})\right) \eta.    \notag
\end{equation}
It is easy to obtain the existence of  $c, C>0$ such  that $\eta\in (c, C)$.
Unfortunately, we are not able to present   $\eta$ by series expansion  with respect to $\frac{1}{\alpha}$, due to its periodicity.  For some further reasons, we need to perform explicit first order expansion of $\alpha$ on $\eta$.  Thanks to  precise calculations, the details of which can be found in Appendix \ref{App-B},    there exist $\pi/L$ periodic functions (with respect to the $\alpha$ variable) $\eta_0(\alpha), \eta_1(\alpha)$ such that 
\begin{align*}
\eta= \eta_0+ \frac{\eta_1}{\alpha}+ O(\frac{1}{\alpha^2})
=  \frac{-1+d_0 e^{-i2L\alpha}}{1+d_0 e^{-i2L\alpha}}+ \frac{d_1 e^{-i2L\alpha}}{(
1+ d_0 e^{-i2L\alpha})^2}\frac{1}{\alpha}+ O(\frac{1}{\alpha^2}),
\end{align*}
where  $d_0$ and $d_1$ are constants depending on the value of $\beta$ with $d_0$ satisfying
\begin{equation}
d_0= \frac{1-a}{1+a} e^{2L\beta}<1.   \notag
\end{equation}

\begin{remark}
The above value of  $d_0$ can be compared with property (iii) of Lemma \ref{lem-eigen}, in fact, from the definition of $d_0$  we easily observe an asymptotic line of poles: $\beta= \frac{1}{2L}\log{\frac{1+a}{1-a}}$, which corresponds to the case that $d_0=1$.  We also emphasize here that the fact that $d_0$ be smaller than 1 is crucial, as this allows us to perform series expansion on $\eta_i$.  Indeed, simple calculation yields
\begin{align*}
-\frac{1}{\eta_0}&= \frac{1+d_0 e^{-i2L\alpha}}{1-d_0 e^{-i2L\alpha}}= 1+2\sum_{n=1}^{\infty} d_0^n e^{-2iL n\alpha},\\
\frac{\eta_1}{\eta_0^2}&= d_1 e^{-i2L\alpha}\frac{1+ d_0e^{-i2L\alpha}}{(1-d_0e^{-i2L\alpha})^2}= d_1\left(1+ \sum_{n=1}^{\infty} (2n+1)d_0^n e^{-2iL n\alpha}\right).
\end{align*}
\end{remark}

Next, we characterize the other terms asymptotically. To simplify the notations we denote
\begin{gather}
e^{\pm}= e^{\pm}(r):= e^{i \omega r}\pm e^{-i\omega r}.     \notag
\end{gather}
Hence
\begin{gather*}
e^{i \langle \omega\rangle r}-e^{-i \langle \omega\rangle r}= e^{-}+ \frac{ir}{2\alpha}e^++ O(\frac{r}{\alpha^2}), \;
e^{i \langle \omega\rangle r}+e^{-i \langle \omega\rangle r}= e^++\frac{ir}{2\alpha}e^-+ O(\frac{r}{\alpha^2}).
\end{gather*}
Then, for $\alpha\in [A, +\infty)$, $s,r\in [-L, L]$, via some direct calculation, successively we are able to obtain 
\begin{align*}
 \eta&=  \eta_0+\frac{\eta_1}{\alpha}+ O(\frac{1}{\alpha^2}),\\
 c_g&=-i\eta \alpha\left(1+\frac{i\beta}{\alpha}+  O(\frac{1}{\alpha^2})\right)=-i\alpha\left(\eta_0+ \frac{i\beta\eta_0+\eta_1}{\alpha}+O(\frac{1}{\alpha^2}) \right),\\
 2\phi_1(r)
 &= e^-+ \frac{ir}{2\alpha}e^+ +O(\frac{r}{\alpha^2}),\\
 2\phi'_{1}(r)& 
 = i \omega \left(e^++\frac{ir}{2\alpha}e^-\right)+O(\frac{1}{\alpha}),\\
 2\phi_2(r)
 &= e^-+ \frac{ir}{2\alpha}e^+ + \eta \left(e^++\frac{ir}{2\alpha}e^-\right)+ O(\frac{r}{\alpha^2}), \\
2\phi'_2(r)&
=  i \omega \left((e^++\frac{ir}{2\alpha}e^-)+ \eta   (e^-+ \frac{ir}{2\alpha}e^+)\right) +O(\frac{1}{\alpha}).
\end{align*}
Thanks to the preceding asymptotic expansion, we can further expand the kernel function $\Gamma$ by
\begin{align}\label{def-Gamma}
\begin{cases}
4 \Gamma(r, s)&= \frac{i}{\alpha \eta}(1-\frac{i \beta}{\alpha}) \phi_1(s)\phi_1(r)+ \frac{i}{\alpha}(1-\frac{i \beta}{\alpha})\phi_1(s)
(e^++\frac{ir}{2\alpha}e^-)(r)+ O(\frac{1}{\alpha^3}),\\
&= \frac{i}{\alpha \eta_0}e^-(s)\left(e^-(r)+\eta_0e^+(r)\right)+O(\frac{1}{\alpha^2}),  \;  \textrm{ for } s\leq r,\\
4 \Gamma(r, s)&= \frac{i}{\alpha \eta}(1-\frac{i \beta}{\alpha}) \phi_1(s)\phi_1(r)+ \frac{i}{\alpha}(1-\frac{i \beta}{\alpha})\phi_1(r)(e^++\frac{is}{2\alpha}e^-)(s)+ O(\frac{1}{\alpha^3}),\\
&= \frac{i}{\alpha \eta_0}e^-(r)\left(e^-(s)+\eta_0e^+(s)\right)+O(\frac{1}{\alpha^2}),  \;   \textrm{ for } r<s,
\end{cases}
\end{align}

\subsection{Instability despite dissipative boundary condition: Theorem \ref{thm-unstab}}
We have seen in Remark \ref{rmk-spec} that the linearized system is stable if and only if there is no pole in the lower half plane.  

Here we are interested in the low frequency case, $i. e.$  the existence of $\omega=-i s$ with $s>0$ such that  equation \eqref{eq-omega-eig} holds, which clearly implies the instability of the nonlinear system.  Note that it is not sufficient to show the existence of  $\omega=-i s, s\geq 0$. For example,  if the unique non-positive  eigenvalue  is given by $s=0$, then this eigen mode  may generate a center manifold  for the nonlinear system, which stabilizes the system. \\

Equation \eqref{eq-omega-eig}, for $\omega=-i s$ with $s>0$ reads 
\begin{equation}\label{eq-arg}
e^{i2L\sqrt{1-s^2}}= \frac{Ls-a-iaL\sqrt{1-s^2}}{Ls-a+iaL\sqrt{1-s^2}}.     
\end{equation}
Observe that in \eqref{eq-arg} both sides have absolute value 1 if $s<1$, hence it suffices to find $s\in (0, 1)$ such that the argument of both complex numbers coincide, which is further equivalent to
\begin{equation}
L\sqrt{1-s^2}= \arg(Ls-a-iaL\sqrt{1-s^2}),    \notag
\end{equation}
or also
\begin{equation}\label{uns-ex-key}
\tan (L\sqrt{1-s^2})= \frac{aL\sqrt{1-s^2}}{a- Ls}.
\end{equation}
Let us define
\begin{equation}
f_1(s):= \tan (L\sqrt{1-s^2}), \; f_2(s):= \frac{aL\sqrt{1-s^2}}{a- Ls}.  \notag 
\end{equation}

Here we have to exclude the case $s = 1$, which corresponds to $\omega = -i$, which is not a pole anyways.
\\

In the following we will prove the existence of unstable eigenmodes,  $\omega=-is$ with $s\in (0, 1)$, by considering different cases.
\begin{itemize}
\item[(1)]
When $L=1$, we look at  equation \eqref{eq-arg}, the goal is to find a solution $s$ that is  between 0 and $a$.  Indeed, when $s$ varies from $0$ to $a$,  the argument of the right hand complex number increases from $\frac{\pi}{2}$ to $\pi$ and the argument of the left decreases from 2 to $2\sqrt{1-a^2}$.  Therefore,  there exists some $s\in (0, a)$ such that these two arguments coincide.  Actually, this strategy easily adapts to the case when $2L\in [\frac{\pi}{2}, \pi]$ mod $2\pi$.  
\item[(2)]
For the general case that  $L>1$, again by  looking at  Equation \eqref{eq-arg} we let $s$ vary from $0$ to $1$.  Simple calculation shows that when $s$ varies from $0$ to $1$,  for \eqref{eq-arg} the argument of the right hand complex number turns from $\frac{\pi}{2}$ for $(s=0)$ to $\pi$ for $ (s=a) $, then to $2\pi$ for $ (s=1) $, and the argument of the left hand side decreases from $2L$ to $0$.  As they are rotating from different directions, thanks to the fact that 
\begin{equation}
2L+ \frac{3\pi}{2}\geq 2\pi,   \notag
\end{equation}
we get  the existence of some  $s\in (0, 1)$ such that Equation \eqref{eq-arg} holds.  
\item[(3)]  When  $0<a<L<1$,  we  turn to the  formula \eqref{uns-ex-key} instead. Since 
\begin{equation}
f_1(0)-f_2(0)= \tan(L)-L>0, \; \lim_{s\rightarrow \frac{a}{L}^-}  f_1(s)- f_2(s)=-\infty,   \notag
\end{equation}
there exists some point $s_0\in (0, \frac{a}{L})$ such that $f_1(s_0)=f_2(s_0)$.
\item[(4)] When  $0<a=L<1$, since
\begin{equation}
f_2(s)= L \sqrt{\frac{1+s}{1-s}}\rightarrow +\infty, \textrm{ when } s\rightarrow 1^-,   \notag
\end{equation}
we also concludes the existence of $s_0$ such that $f_1(s_0)=f_2(s_0)$.
\item[(5)]    When  $0<L< a< 1$, we notice that 
\begin{equation}
f_1(1)-f_2(1)=0.   \notag
\end{equation}
and in fact both $f_1$ and $f_2$ tend to 0 as $s\rightarrow 1$. However, considering the convexity of $f_1-f_2$, we have
\begin{align*}
f_1'(s)-f_2'(s)&= \frac{- \frac{Ls}{\sqrt{1-s^2}}}{\cos^2(L\sqrt{1-s^2})}- aL\left(  -\frac{aLs}{\sqrt{1-s^2} (a-Ls)}+ \frac{L\sqrt{1-s^2}}{(a-Ls)^2}  \right),\\
&= \frac{Ls}{\sqrt{1-s^2}} \left(\frac{a}{a-Ls}-\frac{1}{\cos^2(L\sqrt{1-s^2})}\right)- \frac{aL^2\sqrt{1-s^2}}{(a-Ls)^2}.
\end{align*}
Letting $s$ tend to $1^{-}$ we notice that the preceding expression tends to $+\infty$.  Therefore $(f_1-f_2)'(s)>0$ for $s\in (1-\varepsilon, 1)$, which, combined with the fact that  $(f_1-f_2)(0)>0, (f_1-f_2)(1)=0$, implies the  existence of  some point $s_0\in (0, 1)$ such that $(f_1-f_2)(s_0)=0$.
\end{itemize}
\begin{remark}
This approach actually gives the existence of poles that lie on the interval $i(-1, 0)$, and with the help of more explicit (which is of course much more complicated) study we are even able to completely characterize  the number of poles on this interval.  One still needs to investigate the situation on $i(-C, -1)$ to  obtain the exact number of poles in the lower half plane.
\end{remark}

\subsection{Energy estimates: inhomogeneous problems}\label{subsec-inh}
Both this section and the next are devoted to the energy estimates for the solution $u$ corresponding to the elliptic equation \eqref{ellip-1}--\eqref{ellip-4}.  In this section we are interested in  inhomogeneous problems with null boundary condition (while homogeneous problems with controlled boundary are reserved for the next section),  which  will eventually lead to the decay property that we are seeking. 
\subsubsection{\textbf{The main estimate}}
More precisely,  we want to bound the energy of  $u$  as  follows.
\begin{proposition}\label{prop-key-es}
Let $\beta>0$ such that there is no pole on the line $\Im z= \beta$. If the radial function $U$ verifies
\begin{gather*}
\Delta U + U+ \omega^2 U= \int_0^{+\infty} e^{-it\omega} f(t, x)dt \; \textrm{ in }  \Omega, \\
i\omega U+ a U_{\nu}=0 \; \textrm{ on } \partial \Omega,
\end{gather*}
for $\omega=\alpha+ i\beta$ on the line $\Im z=\beta$,  then the related radial function $w(t, x)= w(t, r)$  given by the inverse Fourier transformation \eqref{four-inverse}  satisfies
\begin{align}
\lVert w(t)\lVert_{\mathcal{H}^1( \Omega)}  \leq C e^{-\beta t} \int_0^{+\infty} e^{t_0 \beta} \lVert (rf)(t_0)\lVert_{L^2_r} dt_0.
\end{align}
\end{proposition}
The  following lemma is essential to the proof of the preceding result.
\begin{lemma}\label{lem-key-es}
Let $\beta>0$ such that there is no pole on the line $\Im z= \beta$. If the radial function $U$ verifies
\begin{gather*}
\Delta U + U+ \omega^2 U= e^{i\alpha t_0}f(x) \; \textrm{ in } \Omega, \\
i\omega U+ a U_{\nu}=0  \; \textrm{ on } \partial \Omega,
\end{gather*}
for $\omega=\alpha+ i\beta$ on the line $\Im z=\beta$, then the radial function  $w(t, x)= w(t, r)$ given by the inverse Fourier transformation \eqref{four-inverse}  satisfies
\begin{equation}
\int_0^L \Big((rw)^2+ w^2+  (rw)_t^2+ (rw)^2_r\Big)(t, r) dr
\leq C e^{-2\beta t}\int_0^L (rf)^2 dr, \notag
\end{equation}
where the left hand side is equivalent to the $\mathcal{H}^1(\Omega)$ norm, and where the constant $C$ is independent on the choice of $f$ and $t_0$.
Moreover, if in the elliptic equation we replace  the function $e^{i\alpha t_0}$ by $\frac{e^{i\alpha t_0}}{\omega}$, then the same estimate still holds.
\end{lemma} 
\begin{proof}[Proof of Proposition \ref{prop-key-es}]
Suppose that $U$ is the unique  solution of the equation (recall that  $\omega= \alpha+ i\beta$ is not a pole, thus the solution is unique). We decompose $U$ by
\begin{equation}
U:= \int_0^{+\infty} U_{t_0} dt_0    \notag
\end{equation}
with $U_{t_0}$ the solution of 
\begin{gather*}
\Delta U_{t_0} + U_{t_0}+ \omega^2 U_{t_0}= e^{-i t_0 \omega}f(t_0, x)= e^{\beta t_0}  e^{-i t_0 \alpha}f(t_0, x), \; \textrm{ in } \overline\Omega, \\
i\omega U_{t_0}+ a (U_{t_0})_{\nu}=0, \; \textrm{ on } \partial \overline\Omega.
\end{gather*}
Thanks to Lemma \ref{lem-key-es} we know that 
\begin{equation}
\lVert w_{t_0}\lVert_{\mathcal{H}^1}
\leq C e^{-\beta t}  e^{t_0 \beta} \lVert (rf)(t_0)\lVert_{L^2_r} ,\;\forall t, t_0\in \mathbb{R}^+,   \notag
\end{equation}
thus
\begin{equation}
\lVert w \lVert_{\mathcal{H}^1}
\leq C e^{-\beta t}  \int_0^{+\infty}e^{t_0 \beta} \lVert (rf)(t_0)\lVert_{L^2_r} dt_0,\;\forall t \in \mathbb{R}^+.   \notag
\end{equation}
\end{proof}
In the rest of this section we focus on the proof of Lemma \ref{lem-key-es}.  For  ease of notations, we only prove the case with $t_0=0$, while the other cases can be adapted by  
 the same proof.  
From now on, we denote $rf(r)$ by $F(r)$.  By assuming $f(x)\in L^2_{rad}$ we have $F(r)\in L^2_r$:
\begin{gather*}
\int_{B_L} f^2(x) dx= C\int_0^L r^2 f^2(r) dr= C \int_0^L F^2(r) dr, \\
\int_{B_L} |\nabla f(x)|^2 dx= C\int_0^L r^2 f_r^2 dr.
\end{gather*}
From the previous section, we know that 
\begin{gather}
R(\omega)f=  r^{-1}\int_0^L \Gamma(r, s; \omega) F(s) ds=: r^{-1}R_0(r; \omega).
\end{gather}
Therefore,  thanks to  the radiality of  the function, it only remains to prove the following proposition to conclude Lemma \ref{lem-key-es}.
\begin{proposition}
Let $a\in(0, 1)$. Let $\beta< \frac{1}{2L} \log{\frac{1+a}{1-a}}$ such that there is no pole on the line $\Im \omega=\beta$. There exists $C>0$ such that, for any $t\in\mathbb{R}^+$,
\begin{gather}
\int_0^L \left|\int_{-\infty}^{+\infty} e^{it \alpha} R_0(r; \omega) d\alpha \right|^2 dr \leq C\|F\|_{L^2}^2, \label{L2} \\
\int_0^L \left|\left(\int_{-\infty}^{+\infty} e^{it \alpha} R_0(r; \omega) d\alpha\right)_r\right|^2 dr \leq  C\|F\|_{L^2}^2, \label{H2-2}  \\
\int_0^L \left|\int_{-\infty}^{+\infty} \frac{1}{r} e^{it \alpha} R_0(r; \omega) d\alpha\right|^2 dr \leq C\|F\|_{L^2}^2, \label{H1-1} \\
\int_0^L \left|\int_{-\infty}^{+\infty} \alpha e^{it \alpha} R_0(r; \omega) d\alpha\right|^2 dr \leq C\|F\|_{L^2}^2. \label{L2-t}
\end{gather}
\end{proposition}
\noindent Observe that $(R_0)_r$ gives an extra $\alpha$ for high frequency, thus in some sense \eqref{H2-2} also implies \eqref{L2} and \eqref{L2-t}.  In the following parts of this section, we will prove those estimates one by one, with a main focus on the proof of \eqref{H2-2}, as this almost describes all technical difficulties. 

Moreover, because those kernels are  uniformly bounded on the domain $\alpha\in[-A, A], r,s\in[0, L]$, $ e.g.$  $\Gamma_r(r, s;\omega), \Gamma_r(r, s;\omega)/r$, we only concentrate on the proof of the related inequalities with integral domain $(-\infty, -A)\cup (A, +\infty)$.

\subsubsection{\textbf{A technical lemma}}
Look at    the expansion of $\Gamma$ by \eqref{def-Gamma}  and  the definition of the kernels that are given in  \eqref{L2}--\eqref{L2-t},  basically we are dealing with
\begin{equation}
\int_{-\infty}^{+\infty}  \frac{1}{\alpha^k}e^{i(t\pm r\pm s)\alpha}d\alpha, \textrm{ with } k= 0, 1, 2, 3...  \notag
\end{equation}
Thus it is natural to present the following lemma concerning the estimates on such a kernel.
\begin{lemma}\label{lem-alpha-osc}
Let $A\geq 1$.  There exists $C>0$ such that for any function $S(r, s)$ satisfying $|S|\leq 1$ for any $r, s\in [0, L]$, the following estimates hold,
\begin{align*}
\lVert \int_0^r S(r, s)\int_A^{+\infty} \frac{1}{\alpha^2}e^{i(t\pm r\pm s)\alpha}d\alpha F(s) ds\lVert_{L^2_r}&\leq C \lVert F\lVert_{L^2},\\
\lVert \int_0^r S(r, s)\int_A^{+\infty} \frac{1}{\alpha}e^{i(t\pm r\pm s)\alpha}d\alpha F(s)ds\lVert_{L^2_r}&\leq C \lVert F\lVert_{L^2},\\
\lVert \int_r^L S(r, s)\int_A^{+\infty} \frac{1}{\alpha}e^{i(t\pm r\pm s)\alpha}d\alpha F(s)ds\lVert_{L^2_r}&\leq C \lVert F\lVert_{L^2}.
\end{align*}

\end{lemma}
\begin{proof}
The first inequality is trivial, it suffices to consider the rest. 
 Let us  consider the following two special integrals $h(p)$ and $k(p)$ given by
\begin{gather}
h(p):= \int_A^{+\infty}\frac{1}{\alpha} e^{ip \alpha} d\alpha,  \label{ref-h} \\
k(p):= -\frac{i}{2} (\int_{-\infty}^{-A}+\int_A^{+\infty})\frac{1}{\alpha} e^{ip \alpha} d\alpha=
 \int_A^{+\infty}\frac{1}{\alpha} \sin(p\alpha) d\alpha.  \label{ref-k}
\end{gather} 
The function $k(p)$ will turn out to be useful later on. 
 Via a simple change of variable,  they satisfy
 \begin{gather*}
 |h(p)|\leq \max\{C_0, |\log|p||\}, \;h'(p)= -\frac{1}{p} e^{ipA},\\
 |k(p)|, \;|k'(p)|\leq C_0.
 \end{gather*}
Indeed, for $h(p)$ with $p>0$, we know from a simple change of variables that
\begin{align*}
\int_A^{+\infty} \frac{1}{\alpha} e^{i p \alpha}d\alpha= \int_{Ap}^{+\infty} \frac{1}{\alpha} e^{i\alpha} d\alpha
=\int_{pA}^1  \frac{1}{\alpha}e^{i\alpha}d\alpha+\int_1^{+\infty} \frac{1}{\alpha} e^{i \alpha}d\alpha,
\end{align*}
which concludes the properties of $h(p)$.  The same idea also holds for $k(p)$, while thanks to the fact that $\sin(p\alpha)/\alpha$ is continuous at $\alpha=0$, both $k(p)$ and $k'(p)$ are uniformly bounded.

Therefore, for example, by taking $t+r+s$ we have
\begin{equation}
|G_0(r, s)|:=| \int_A^{+\infty} \frac{1}{\alpha} e^{i (t+ r+ s) \alpha}d\alpha|\leq C_0+ |\log |t+r+s||.     \notag
\end{equation}
We extend $G_0(r, s)$ from $s\leq r$ to $0\leq s, r\leq L$, then we have  
\begin{equation}
g(r):=\int_0^r G_0(r, s) f(s) ds,\; \forall r\in [0, L]     \notag
\end{equation}
satisfies
\begin{align*}
\lVert g(r)\lVert_{L^2_r}
\leq\left(\int_0^L |\int_0^L G_0(r, s) f(s) ds|^2 dr\right)^{\frac{1}{2}}
\leq \int_0^L \left(\int_0^L G_0^2(r, s) dr\right)^{\frac{1}{2}} |f(s)|ds,  
\end{align*}
thus
\begin{align*}
\lVert g(r)\lVert_{L^2_r}\leq C  \lVert f\lVert_{L^1}, \textrm{ for } t\in [-4L, 4L].
\end{align*}
On the other hand, if $t\notin [-4L, 4L]$ then $|G_0|$ is uniformly bounded, hence
\begin{equation}
\lVert g(r)\lVert_{L^2_r}  \leq C \lVert f\lVert_{L^1},   \notag
\end{equation}
which combined with the preceding case yields
\begin{equation}
\lVert g(r)\lVert_{L^2_r}  \leq C \lVert f\lVert_{L^1}, \forall t\in \mathbb{R}.   \notag
\end{equation}
\end{proof}

\subsubsection{\textbf{On the estimate of \eqref{L2}.}}
This part  is devoted to  the proof of the $L^2$ estimate \eqref{L2}: 
\begin{equation}
\int_{-\infty}^{+\infty} e^{it \alpha} R_0(r; \omega) d\alpha\in L^2_r.   \notag
\end{equation}
Thanks to Lemma  \ref{lem-alpha-osc} and the expansion of $\Gamma$ given by Equation \eqref{def-Gamma}, we can remove the $O(\frac{1}{\alpha^2})$ terms  by regarding  $\Gamma(r, s)$  as $\mathbf{\tilde{\Gamma}(r, s)}+ O(\frac{1}{\alpha^2})$, where the function $\tilde{\Gamma}$ is defined by
\begin{gather*}
4\tilde \Gamma(r, s)
:= 
\begin{cases}
\frac{i}{\alpha \eta_0}e^-(s)e^-(r)+ \frac{i}{\alpha}e^-(s)e^+(r), \; \textrm{ if } s\leq r, \\
 \frac{i}{\alpha \eta_0}e^-(s)e^-(r)+ \frac{i}{\alpha}e^-(r)e^+(s), \; \textrm{ if } s> r.
 \end{cases}
\end{gather*}
Furthermore, thanks to Lemma  \ref{lem-alpha-osc} again, the kernel 
\begin{equation}
 \int_A^{+\infty}e^{it\alpha} \tilde \Gamma(r, s; \omega)d\alpha   \notag
\end{equation}
defines a bounded operator on $L^2$.   The other parts can be estimated in the same way.

\subsubsection{\textbf{On the estimate of \eqref{H2-2}.}}
 We study, in this part,  the derivative of the  kernel $K(r, s)$ given by
\begin{equation}
K(r, s):= \left(\int_{-\infty}^{-A}+\int_A^{+\infty}\right)e^{it\alpha} \Gamma(r, s; \omega)d\alpha.   \notag
\end{equation}
\noindent\textbf{Reduce  $K(r, s)$  as $\tilde{K}(r, s)$.}

Notice that $\Gamma$ contains several non-explicit parts, which prevent us from  using oscillate integration directly.  However,   observe that those terms are asymptotically small,  therefore, we select and only work with the main part of $K$: 
\begin{equation}
\tilde{K}(r, s):= \left(\int_{-\infty}^{-A}+ \int_A^{+\infty}\right)  e^{it\alpha} \tilde \Gamma(r, s; \omega)d\alpha.     \notag
\end{equation}

After some careful calculation, the details of which can be found in Appendix \ref{App-B}, we have  the following expansion of  $4\Gamma_r(r, s)-4\tilde{\Gamma}_r(r, s)$, {\small
\begin{align*}
-\frac{i}{2\eta \alpha}\left(s e^+(r)\left(e^+(s)+\eta  e^-(s)\right) +  r e^-(r)(e^-(s)+\eta e^+(s))    \right)+ \frac{1}{\alpha}\frac{\eta_1}{\eta_0^2}e^+(r)e^-(s)  +O(\frac{1}{\alpha^2}), 
\end{align*} }
 for the region  $r< s$, and {\small
\begin{align*}
-\frac{i}{2\alpha\eta}\Big(se^+(s)(e^+(r)+\eta e^-(r))+re^-(s)(e^-(r)+\eta e^+(r)) \Big)+ \frac{1}{\alpha}\frac{\eta_1}{\eta_0^2}e^+(r)e^-(s)+ O(\frac{1}{\alpha^2}),
\end{align*} }
for the region  $s\leq r$.

Now we prove that the kernel $\tilde \Gamma_r- \Gamma_r$
generates a bounded operator on $L^2$.  Indeed $\tilde \Gamma_r- \Gamma_r$ is composed of two parts:  we denote by $R_1$ for the  $O(\frac{1}{\alpha^2})$ part, and by $R_2$ for the rest part.  Suppose that $F(s)$ is a $L^2$ function, then consider
\begin{equation}
\int_0^L F(s)\int_A^{+\infty} e^{it\alpha} (R_1+R_2) d\alpha ds.     \notag
\end{equation}
At first, for $R_1$ we know that 
\begin{equation}
 \int_0^L F(s)\int_A^{+\infty} e^{it\alpha} O(\frac{1}{\alpha^2})d\alpha ds     \notag
\end{equation}
is uniformly bounded, which  implies the required   $L^2_r$ estimate.

Next, for $G_2$ we notice that all the terms are of the form 
\begin{equation}
S(r, s)d_0^n \int_A^{+\infty} \frac{1}{\alpha} e^{i (t- 2nL\pm r\pm s) \alpha}d\alpha,     \notag
\end{equation}
where $|S(r, s)|$ is uniformly bounded, and where the index $-i2Ln\alpha$ on the exponential term comes from the series expansion of $\frac{1}{\eta_0}$. 
Therefore, $G_2$ generates a bounded operator from $L^1$ to  $L^2$ with its norm given by 
\begin{equation}
4|S(r, s)|_{L^{\infty}} C (2+ 2C\sum_n d_0^n+ C \sum_n(2n+1)d_0^n)<+\infty.     \notag
\end{equation}

As a consequence,   in the following,  it suffices  to work with the reduced kernel 
\begin{equation}
\tilde K_r(r, s):= \left(\Big(\int_{-\infty}^{-A}+\int_A^{+\infty}\Big)e^{it\alpha} \tilde \Gamma(r, s; \omega)d\alpha\right)_r.      \notag
\end{equation}
Thanks to the explicit expression of  $\tilde{\Gamma}(r, s)$,  we directly calculate $\tilde{K}$ and $\tilde{K}_r$ by  decomposing $\tilde{\Gamma}$ into two parts $\tilde{\Gamma}_1$ and $\tilde{\Gamma}_2$:
\begin{align*}
\tilde{\Gamma}_1&:= \frac{i}{\alpha \eta_0} e^-(s)e^-(r)+\frac{i}{\alpha}\left(e^{i(\alpha+i\beta)(s+r)}-e^{-i(\alpha+i\beta)(s+r)}\right), \\
\tilde{\Gamma}_2&:=\frac{i}{\alpha}\left(-e^{i(\alpha+i\beta)(r-s)}+ e^{i(\alpha+i\beta)(s-r)}\right), \textrm{ if }s\leq r, \\
\tilde{\Gamma}_2&:= \frac{i}{\alpha}\left(e^{i(\alpha+i\beta)(r-s)}- e^{i(\alpha+i\beta)(s-r)}\right), \textrm{ if }s> r.
\end{align*}
In the following we shall  treat  $\tilde{\Gamma}_1$ and $\tilde{\Gamma}_2$ one by one.
\\

\noindent \textbf{On the estimate concerning $\tilde{\Gamma}_1(r, s)$.}\\
Concerning $\tilde{\Gamma}_1$, thanks to the series expansion of $\frac{1}{\eta_0}$, each of the composing  terms  has the  form
\begin{equation}
\frac{i}{\alpha}d_0^n e^{-i 2nL\alpha} e^{i(\pm r\pm s)(\alpha+i\beta)}.   \notag
\end{equation}
As the sum of $d_0^{n}$ converges,  thanks to an argument similar to the one used for $G_2$,  it suffices  to obtain a uniform bound on 
\begin{equation}
\frac{1}{\alpha} e^{i t_0\alpha} e^{i(\pm r\pm s)(\alpha+i\beta)}, \; \forall t_0\in \mathbb{R}.   \notag
\end{equation}
Let 
\begin{equation*}
H_1(r, s):= \int_A^{+\infty}e^{i (t+t_0)\alpha}\frac{1}{\alpha}  e^{i(\pm r\pm s)(\alpha+i\beta)} d\alpha=  \int_A^{+\infty} \frac{1}{\alpha} e^{i \tilde{t}\alpha} e^{i(\pm r\pm s)(\alpha+i\beta)} d\alpha.
\end{equation*}
Then, via  direct calculation, we obtain
\begin{equation}
H_1(r, s)= e^{-(\pm r\pm s)\beta} \int_A^{+\infty} \frac{1}{\alpha} e^{i(\tilde{t}\pm r\pm s)\alpha} d\alpha= e^{-(\pm r\pm s)\beta}\left(C_0+ h(\tilde t\pm r\pm s) \right). \notag
\end{equation}
In order to understand the derivative in the distributional sense, we write 
\begin{align*}
H_1(r, s)= \lim_{\epsilon\rightarrow 0+}e^{-(\pm r\pm s)\beta} \int_A^{+\infty} \frac{1}{\alpha} e^{i(\tilde{t}\pm r\pm s)\alpha - \epsilon\alpha} d\alpha,
\end{align*}
and so 
\begin{align*}
(H_1)_r(r, s)&= \mp \beta e^{-(\pm r\pm s)\beta}\left(C_0+ h(\tilde t\pm r\pm s) \right) \pm  \lim_{\epsilon\rightarrow 0+}e^{-(\pm r\pm s)\beta}\frac{1}{\tilde t\pm r\pm s + i\epsilon} e^{i(\tilde t\pm r\pm s)A}, \\
&= \mp \beta e^{-(\pm r\pm s)\beta}\left(C_0+ h(\tilde t\pm r\pm s) \right) \pm \lim_{\epsilon\rightarrow 0+}e^{i\tilde{t}A}  \frac{1}{\tilde t\pm r\pm s + i\epsilon} e^{i(\pm r\pm s)(A+i\beta)}, \\
&=\mp \beta e^{-(\pm r\pm s)\beta}\left(C_0+ h(\tilde t\pm r\pm s) \right)  \pm \lim_{\epsilon\rightarrow 0+}e^{i\tilde{t}A}  \frac{1}{\tilde t\pm r\pm s + i\epsilon} e^{(\pm r\pm s)z},
\end{align*}
where $A, \beta$, and $z=-\beta+iA$ are fixed numbers.  Thanks to Lemma \ref{lem-alpha-osc}, the first part of the preceding formula is a kernel that generates a bounded operator on $L^2$ with a uniform norm. It only remains to consider the second part, more precisely, the distribution valued kernel 
\begin{equation}\begin{split}
\lim_{\epsilon\rightarrow 0+}\frac{1}{\tilde t\pm r\pm s + i\epsilon} e^{(\pm r\pm s)z} &= P.V.\big(\frac{1}{\tilde t\pm r\pm s} e^{(\pm r\pm s)z}\big)\\
& - i\pi\cdot \delta_0(\tilde t\pm r\pm s),  r, s\in [0, L].  \notag
\end{split}\end{equation}
The second term on the right clearly leads to an operator which is $L^2$-bounded. 
As for the first term on the right, thanks to a  symmetric change of variable, it suffices to consider 
\begin{equation}
\frac{1}{t+ r\pm s} e^{(r\pm s)z},  r, s\in [0, L].   \notag
\end{equation}
For $F(s)\in L^2$, we define 
\begin{equation}
G(r)= \int_0^L\frac{1}{t+r\pm s} e^{(r\pm s)z}F(s)ds,   \notag
\end{equation}
thus
\begin{equation}
e^{-rz} G(r)= \int_0^L \frac{1}{t+r\pm s} \left(e^{-sz}F(s)\right) ds,   \notag
\end{equation}
which  can  be treated by means of the following lemma.
\begin{lemma}\label{lem-1-r-s}
There exists $C>0$ such that for any $t\in \mathbb{R}$, we have 
\begin{equation}
\left\lVert \int_0^L \frac{1}{t+r\pm s} f(s) ds\right\lVert_{L^2(0, L)}\leq C\lVert f(s)\lVert_{L^2(0, L)}.   \notag
\end{equation}
\end{lemma}
\begin{proof}
Observe that the preceding formula is of Hilbert transform type. Extend $f(s)$ trivially to the whole line and by abuse of notation denote the resulting function by $f(s)$.  Then  for $r\in [0, L]$ we have
\begin{equation}
g(r)= \int_0^L \frac{1}{t+r\pm s} f(s) ds= \int_{-\infty}^{+\infty} \frac{1}{t+r\pm s} f(s) ds= \int_{-\infty}^{+\infty} \frac{1}{s}f(t+r\mp s) ds.   \notag
\end{equation}
We extend $g(r)$ by the same formula to the whole line and still denote it by $g(r)$.  As this is exactly a Hilbert transform, we get 
\begin{equation}
\lVert g(r)\lVert_{L^2}= \pi \lVert f(s)\lVert_{L^2}.   \notag
\end{equation}
\end{proof}

\noindent\textbf{On the estimate concerning $\tilde{\Gamma}_2(r, s)$.}\\
Concerning $\tilde{\Gamma}_2(r, s)$, we know that (recall the proof of Lemma~\ref{lem-alpha-osc} for the definition of the function $k$)
\begin{align*}
\tilde{K}_2(r, s)&=\left(\int_{-\infty}^{-A}+\int_A^{+\infty}\right) e^{it\alpha} \tilde{\Gamma}_2(r, s) d\alpha, \\
(\textrm{for } s\leq r) &= \left(\int_{-\infty}^{-A}+\int_A^{+\infty}\right)e^{it\alpha} \frac{i}{\alpha} \left(-e^{i(\alpha+i\beta)(r-s)}+e^{-i(\alpha+i\beta)(r-s)} \right), \\
&= 2e^{-(r-s)\beta} k(t+r-s)-   2e^{(r-s)\beta}k(t-r+s),\\
(\textrm{for } s> r) &= \left(\int_{-\infty}^{-A}+\int_A^{+\infty}\right)e^{it\alpha}\frac{i}{\alpha}  \left(e^{i(\alpha+i\beta)(r-s)}-e^{-i(\alpha+i\beta)(r-s)} \right),\\
&= -2e^{-(r-s)\beta} k(t+r-s)+   2e^{(r-s)\beta}k(t-r+s).
\end{align*}
Simple calculation shows that  $|(\tilde{K}_2)_r(r, s)|$ is uniformly bounded in the domain $r,s\in [0, L]$, which of course generates a bounded operator on $L^2$.  This simple symmetric structure also holds for $\tilde{\Gamma}_1$.

\subsubsection{\textbf{On the estimate of \eqref{H1-1}.}}
To be compared with Equation \eqref{L2}, the estimate \eqref{H1-1} admits an extra $1/r$, however,  as we will demonstrate later on,  which does not produce any singularity at $r=0$.  Similar to the previous sections, we  reduce the kernel $\Gamma$ by $\tilde{\tilde{\Gamma}}$ that is given by
\begin{gather*}
\tilde{\tilde{\Gamma}}:= 
\begin{cases}
 \frac{1}{\alpha \eta_0}e^-(s) \left(e^-(r)+\eta_0 e^+(r)\right)= \tilde \Gamma, \; \textrm{ for } s\leq r,\\
 \frac{1}{\alpha \eta_0}e^-(r) \left(e^-(s)+\eta_0 e^+(s)+ \frac{is}{2\alpha}e^+(s)+ \eta_0 \frac{is}{2\alpha}e^-(s)\right), \; \textrm{ for } s> r.
\end{cases}
\end{gather*}
In fact, it is allowed to consider $\tilde{\tilde{\Gamma}}$ instead of $\Gamma$, thanks to those techniques that we performed in previous sections as well as the following facts,
\begin{itemize}
\item[(1)] in $\Gamma-\tilde{\tilde{\Gamma}}$, those terms of type $\frac{r}{\alpha^2}$ can be easily bounded;
\item[(2)] for $|\alpha| \geq A$, we have $|e^-(r)|\leq C |\alpha r|$;
\item[(3)] inequality \eqref{H1-es-3}, which will be proved later on.
\end{itemize}
Therefore, it suffices to consider the kernel
\begin{equation}
 L(r, s):= \frac{1}{r}\Big(\int_{-\infty}^{-A}+\int_A^{+\infty}\Big)e^{it\alpha} \tilde{\tilde{\Gamma}}(r, s; \omega)d\alpha.      \notag
\end{equation}
It remains to show that 
\begin{lemma}\label{lem-1r-key}
The following integrations 
\begin{align}
\frac{1}{r} \Big(\int_{-\infty}^{-A}+\int_A^{+\infty}\Big)e^{it\alpha} \frac{1}{\alpha} e^{-}(s) e^-(r)d\alpha, & \textrm{ for } s\leq r, \label{H1-es-0}\\
\frac{1}{r} \Big(\int_{-\infty}^{-A}+\int_A^{+\infty}\Big)e^{it\alpha} \frac{1}{\alpha} e^{-}(s) e^+(r)d\alpha, & \textrm{ for } s\leq r,\label{H1-es-1}\\
\frac{1}{r} \Big(\int_{-\infty}^{-A}+\int_A^{+\infty}\Big)e^{it\alpha} \frac{1}{\alpha} e^{-}(s) e^-(r)d\alpha, & \textrm{ for } r\leq s,\\
\frac{1}{r} \Big(\int_{-\infty}^{-A}+\int_A^{+\infty}\Big)e^{it\alpha} \frac{1}{\alpha} e^{-}(r) e^+(s)d\alpha, & \textrm{ for } r\leq s,\label{H1-es-2}\\
\frac{1}{r} \Big(\int_{-\infty}^{-A}+\int_A^{+\infty}\Big)e^{it\alpha} \frac{1}{\alpha^2} e^{-}(r) e^{\pm}(s)d\alpha, & \textrm{ for } r\leq s\label{H1-es-3}
\end{align}
 are uniformly bounded for $t\in \mathbb{R}, s, r\in (0, L)$.  
 \end{lemma}
 \begin{proof}[Proof of Lemma \ref{lem-1r-key}]
 Now we present the proof of Lemma \ref{lem-1r-key},
 the essential idea is to derive an extra $r$ from $e^-(r)$ (or $e^-(s)$ if $s\leq r$).   We only need to show \eqref{H1-es-1}, \eqref{H1-es-2} and \eqref{H1-es-3}. 
  
 For inequality \eqref{H1-es-2}, we have  
 \begin{equation}
 \frac{1}{r} \Big(\int_{-\infty}^{-A}+\int_A^{+\infty}\Big)e^{it\alpha} \frac{1}{\alpha} \left(e^{-\beta r}e^{i\alpha r}-e^{\beta r}e^{-i\alpha r}\right) \left(e^{-\beta s}e^{i\alpha s}+ e^{\beta s}e^{-i\alpha s}\right) d\alpha, \notag
 \end{equation}
 thus by symmetry we only need to treat 
  \begin{align*}
 &\frac{1}{r} \Big(\int_{-\infty}^{-A}+\int_A^{+\infty}\Big)e^{it\alpha} \frac{1}{\alpha} \left(e^{-\beta r}e^{i\alpha r}-e^{\beta r}e^{-i\alpha r}\right) e^{-\beta s}e^{i\alpha s} d\alpha\\
 =& e^{-\beta s} \frac{1}{r} \Big(\int_{-\infty}^{-A}+\int_A^{+\infty}\Big)e^{it\alpha} \frac{1}{\alpha} \left(e^{-\beta r}e^{i\alpha r}-e^{\beta r}e^{-i\alpha r}\right) e^{i\alpha s} d\alpha.
 \end{align*}
 Recalling the definition of $k(p)$, \eqref{ref-k}, we get 
\begin{align*}
&\;\;\;\;\frac{1}{r} \Big(\int_{-\infty}^{-A}+\int_A^{+\infty}\Big)e^{it\alpha} \frac{1}{\alpha} \left(e^{-\beta r}e^{i\alpha r}-e^{\beta r}e^{-i\alpha r}\right) e^{i\alpha s} d\alpha\\
&= \frac{1}{r} \left(e^{-\beta r}k(t+s+r)-e^{\beta r}k(t+s-r)\right),
\end{align*}
where 
\begin{equation}
k(p)= \int_{Ap}^{+\infty} \frac{\sin \alpha}{\alpha} d\alpha,\;  |k|\leq C, \;|k(p_1)-k(p_2)|\leq C|p_1-p_2|.      \notag
\end{equation}
Therefore
\begin{align*}
 &\;\;\;\;\frac{1}{r} \left(e^{-\beta r}k(t+s+r)-e^{\beta r}k(t+s-r)\right)\\
 &= \frac{1}{r} \Big(e^{-\beta r}k(t+s+r)-k(t+s+r) + k(t+s+r)-k(t+s-r)\\
 &\hspace{0.5cm}  + k(t+s-r)-e^{\beta r}k(t+s-r)\Big), \\
 &\leq C. 
\end{align*}

For inequality \eqref{H1-es-1}, we benefit from the fact that $s\leq r$,
\begin{equation}
 \frac{1}{r} \Big(\int_{-\infty}^{-A}+\int_A^{+\infty}\Big)e^{it\alpha} \frac{1}{\alpha} \left(e^{-\beta s}e^{i\alpha s}-e^{\beta s}e^{-i\alpha s}\right) \left(e^{-\beta r}e^{i\alpha r}+ e^{\beta r}e^{-i\alpha r}\right) d\alpha. \notag
 \end{equation}
Similar to \eqref{H1-es-2}, we only prove 
\begin{equation}
\frac{1}{r} \Big(\int_{-\infty}^{-A}+\int_A^{+\infty}\Big)e^{it\alpha} \frac{1}{\alpha} \left(e^{-\beta s}e^{i\alpha s}-e^{\beta s}e^{-i\alpha s}\right) e^{i\alpha r} d\alpha, \notag
\end{equation}
thus
\begin{align*}
&\;\;\;\;\frac{1}{r}\left(e^{-\beta s} K(t+r+s)-e^{\beta s}k(t+r-s)\right),\\
&= \frac{1}{r}\Big(e^{-\beta s} k(t+r+s)-k(t+r+s)   +k(t+r+s)- k(t+r-s)\\
 &\hspace{0.5cm}  + k(t+r-s)- e^{\beta s}k(t+r-s)\Big),\\
&\leq C \frac{s}{r} \leq C.
\end{align*}

As for the last estimate \eqref{H1-es-3}, thanks to the same reasons as the previous two inequalities,  we only consider 
\begin{equation}
\frac{1}{r} \Big(\int_{-\infty}^{-A}+\int_A^{+\infty}\Big)\frac{e^-(r)}{\alpha^2}e^{i\alpha s}d\alpha= \frac{1}{r} \left(e^{-\beta r}Q(t+s+r)- e^{\beta r}Q(t+s-r) \right),   \notag
\end{equation}
with
\begin{equation}
Q(p):=  \Big(\int_{-\infty}^{-A}+\int_A^{+\infty}\Big) \frac{1}{\alpha^2} e^{ip\alpha}d\alpha    \notag
\end{equation}
which, similar to $k(p)$ and $h(p)$, satisfies
\begin{equation}
|Q(p)|\leq C,\; |Q'(p)|= 2|k(p)|\leq C.    \notag
\end{equation}
This completes the uniform boundedness of \eqref{H1-es-0}--\eqref{H1-es-3}. 
\end{proof}

This  implies that the kernel $L(r, s)$ is uniformly bounded, thus generating a bounded operator on $L^2$.

\subsubsection{\textbf{On the estimate of \eqref{L2-t}.}}
Thanks to the preparation in  previous sections, the proof of \eqref{L2-t} is straightforward since the proof of the inequality \eqref{H2-2} already presented the main difficulty. Indeed, thanks to the decomposition of $\Gamma$ under the  $O(\frac{1}{\alpha^3})$ order term, the kernel 
\begin{equation}
K(r, s):= \left(\int_{-\infty}^{-A}+\int_A^{+\infty}\right)\alpha e^{it\alpha} \Gamma(r, s; \omega)d\alpha,  \notag
\end{equation}
can be written by
\[  k_1+ \frac{1}{\alpha} k_2+ O(\frac{1}{\alpha^2}), \]
while $k_1$ is similar to the leading term  that appear in $\tilde{K}_r$, $\frac{1}{\alpha} k_2$ and $O(\frac{1}{\alpha^2})$ can be handled by Lemma \ref{lem-alpha-osc}.

\subsection{Energy estimates: boundary  problems}\label{subsec-boum}
In this section, we further prove  the energy estimates  for the  homogeneous equation with boundary controls. The idea is to transform such a boundary controlled  problem into an inhomogeneous one. 
\begin{proposition}\label{prop-key-bd}
Let $\beta>0$ such that there is no pole on the line $\Im z= \beta$. If the radial function $U$ satisfies 
\begin{gather*}
\Delta U + U+ \omega^2 U= 0  \; \textrm{ in }  \Omega, \\
i\omega U+ a U_{\nu}=\int_0^{+\infty} e^{-i\omega t}b(t)dt  \; \textrm{ on } \partial  \Omega,
\end{gather*}
for $\omega=\alpha+ i\beta$ on the line $\Im z=\beta$,  then the radial function $w(t, x)= w(t, r)$ given by the inverse Fourier transformation \eqref{four-inverse}  satisfies
\begin{equation}
\lVert w\lVert_{\mathcal{H}^1}
\leq  Ce^{-\beta t} \int_0^{+\infty} e^{t_0\beta} \left(|b(t_0)|+|b'(t_0)|\right)d t_0+ C \left|\int_0^t b(s) ds\right|+ C|b(t)|,\; \forall t\geq 0.      \notag
\end{equation}
\end{proposition}
Let the radial smooth function $g(x)$ be satisfying $
g(L)=0,\; g'(L)=1/a$. Let as decompose  $U$ as $$U= V+ \hat{b}(\omega)g(x).$$ Then, the  function $V$ satisfies,
\begin{gather*}
\Delta V + V+ \omega^2 V= -\hat{b}(\omega)(\Delta g+ g)- \omega^2\hat{b}(\omega)g(x)  \; \textrm{ in } \Omega, \\
i\omega v+ a V_{\nu}=0 \; \textrm{ on } \partial \Omega.
\end{gather*}
Because  the preceding source term can be written as 
\begin{align*}
\hat{b}(\omega)(\Delta g+ g)(x)&= \int_0^{+\infty}e^{-it\omega}b(t)(\Delta g+ g)(x)dt, \\
\omega^2\hat{b}(\omega)g(x)&= -\int_0^{+\infty}e^{-it\omega}b''(t)dt g(x),
\end{align*}
thanks to Proposition \ref{prop-key-es},  we obtain
\begin{equation}
\lVert v(t)\lVert_{\mathcal{H}^1}\leq Ce^{-\beta t} \int_0^{+\infty} e^{t_0\beta} \left(|b(t_0)|+|b''(t_0)|\right)d t_0.       \notag
\end{equation}
Actually,  we  can further reduce the regularity of $b(t)$ in the preceding formula.  Indeed, one needs to consider some function $g(\omega, x)$ instead of $g(x)$, more precisely, let 
\begin{equation}
g(\omega, x):= \frac{\phi(x)}{i\omega },      \notag
\end{equation}
provided that the smooth function $\phi$ satisfies $\phi(L)=1, \phi'(L)=0$.   Let $$U= V+ \hat{b}(\omega)g(\omega, x).$$ Then, the function $V$ satisfies,
\begin{gather*}
\Delta V + V+ \omega^2 V= -\frac{\hat{b}(\omega)}{i\omega}(\Delta \phi+ \phi)+ i\omega\hat{b}(\omega)g(x)  \; \textrm{ in }  \Omega, \\
i\omega v+ a V_{\nu}=0 \; \textrm{ on } \partial  \Omega.
\end{gather*}
Because the source term can be written as 
\begin{align*}
\frac{\hat{b}(\omega)}{\omega}(\Delta \phi+ \phi)&= \int_0^{+\infty}\frac{e^{-it\omega}}{\omega}b(t)(\Delta \phi+ \phi)(x)dt, \\
i\omega\hat{b}(\omega)\phi&= \int_0^{+\infty}e^{-it\omega}b'(t)dt \phi,
\end{align*}
by applying Proposition \ref{prop-key-es} and Lemma \ref{lem-key-es} we get 
\begin{equation}
\lVert v(t)\lVert_{\mathcal{H}^1}\leq Ce^{-\beta t} \int_0^{+\infty} e^{t_0\beta} \left(|b(t_0)|+|b'(t_0)|\right)d t_0, \; \forall t\geq 0.    \notag
\end{equation}

On the other hand, for $l:= u-v$ that satisfies 
\begin{equation}
\hat{l}(\omega)= \frac{\hat{b}(\omega)}{i\omega} \phi,    \notag
\end{equation}
we have
\begin{equation}
e^{\beta t} l(t)=\phi \int_{-\infty}^{+\infty}e^{i\alpha t}\frac{\int_0^{+\infty} e^{-i\alpha s} e^{\beta s}b(s)ds}{i\alpha-\beta} d\alpha= \phi \int_0^{+\infty}\int_{-\infty}^{+\infty}\frac{e^{i\alpha(t-s)}}{i\alpha-\beta}d\alpha e^{\beta s} b(s) ds.  \notag
\end{equation}
Because 
\begin{equation}
\int_{-\infty}^{+\infty}\frac{e^{i\alpha(t-s)}}{i\alpha-\beta}d\alpha    \notag
\end{equation}
is uniformly bounded, we know that 
\begin{equation}
 |l(t)|\lesssim e^{-\beta t} \int_0^{+\infty} e^{\beta s} |b(s)| ds.   \notag
\end{equation}
Concerning the term $e^{\beta t}l_t(t)$, one needs to estimate 
\begin{equation}
\int_0^{+\infty}\left(\int_{-\infty}^{+\infty}\frac{e^{i\alpha(t-s)}}{i\alpha-\beta}d\alpha\right)_t e^{\beta s} b(s) ds.    \notag
\end{equation}
Because the term 
\begin{equation}
\left(\int_{-\infty}^{+\infty}\frac{e^{i\alpha(t-s)}}{(i\alpha-\beta) i\alpha}d\alpha\right)_t    \notag
\end{equation}
is uniformly bounded, by taking the difference it suffices to estimate
\begin{equation}
\left(\left(\int_{-\infty}^{-A}+ \int_A^{+\infty}\right)\frac{e^{i\alpha(t-s)}}{\alpha}d\alpha\right)_t,   \notag
\end{equation}
which, thanks to the function $k(p)$, equals to $k'(t-s)$, therefore  is  obviously uniformly bounded.\\

Thus, combining the preceding estimates on $l$ and $v$, we get 
\begin{equation}
\lVert u(t)\lVert_{\mathcal{H}^1}\lesssim e^{-\beta t} \int_0^{+\infty} e^{t_0\beta} \left(|b(t_0)|+|b'(t_0)|\right)d t_0, \; \forall t\geq 0.     \notag
\end{equation}
\begin{remark}
It should be pointed out that by  the choice of $b(t)$, it is  essential  to let  
\begin{equation}
\chi_t u(t, L)_t\in L^1, \label{trace-H1}
\end{equation}
for  $u$, some  solution of 
\begin{gather}
\Box u- u=h(t, x)  \textrm{ in }  \Omega, \label{wave-free-1}\\
(u_t+ a u_{\nu})=0   \textrm{ on } \partial \Omega,   \\
u(0, x)= u_0,\; u_t(0, x)= v_0. \label{wave-free-3}
\end{gather}
Indeed, this is exactly the trace estimate that we have previously proved in Lemma \ref{lem-a-es}.
\end{remark}

\subsection{Construction of feedbacks: exponential stabilization of the linearized system}\label{sec:constructionfeedback}
Thanks to the energy estimates in the preceding two sections, we are in a position to conclude the following bound for the solution $u|_{(0,2)}$ of \eqref{wave-free-1}--\eqref{wave-free-3} and the solution $w=\chi u$ of \eqref{wave-op-1}--\eqref{wave-op-3}. \\

For $t\leq 2$, thanks to Lemma \ref{lem-a-es},  there is
\begin{equation}
\lVert u\lVert_{\mathcal{H}^1}\lesssim  \int_0^{t} \lVert h(s, \cdot)\lVert_{L^2(\overline{\Omega})}  ds+ \lVert u_0\lVert_{H^1(\overline{\Omega})}+ \lVert v_0\lVert_{L^2(\overline\Omega)}.
\end{equation}

 For $t\geq 2$, thanks to the estimates that obtained in Section \ref{subsec-inh}--\ref{subsec-boum}, we get exponential decay of $u$ with a decay rate $\beta$:
\begin{align*}
\lVert u(t)\lVert_{\mathcal{H}^1}&= \lVert w(t)\lVert_{\mathcal{H}^1},\\
&\lesssim  e^{-\beta t}\int_0^{+\infty} e^{s \beta} \lVert h_0(s)\lVert_{L^2(\overline{\Omega})} ds+ e^{-\beta t} \int_0^{+\infty} e^{s\beta} \left(|b_0(s)|+|b_0'(s)|\right)d s, \\
&\lesssim e^{-\beta t}\left(\int_0^{+\infty} e^{s \beta} \lVert h(s)\lVert_{L^2(\overline{\Omega})} ds+ \lVert u_0\lVert_{H^1(\overline{\Omega})}+ \lVert v_0\lVert_{L^2(\overline{\Omega})}\right)\\
&\;\;\;\;\;+ e^{-\beta t}\left(\int_0^{2} e^{s \beta} \lVert h(s)\lVert_{L^2(\overline{\Omega})} ds+ \lVert u_0\lVert_{H^1(\overline{\Omega})}+ \lVert v_0\lVert_{L^2(\overline{\Omega})}\right) \\
&\;\;\;\;\;   +e^{-\beta t} \int_0^{+\infty} e^{s\beta} \left(|b(s)|+|b'(s)|\right)d s,\\
&\lesssim e^{-\beta t}\left(\int_0^{+\infty} e^{s \beta} \left(\lVert h(s)\lVert_{L^2(\overline{\Omega})}+|b(s)|+|b'(s)|\right)  ds+ \lVert u_0\lVert_{H^1(\overline{\Omega})}+ \lVert v_0\lVert_{L^2(\overline{\Omega})}\right),
\end{align*}
provided that the control $b(t)$ satisfies {\bf Assumption 1}  that is raised in Section \ref{subsec-str}. To be more specific:
\begin{itemize}
\item[(i)] \textit{$b(t)= 0$ on the time interval $t\in [0, 2]$;}
\item[(ii)]  \textit{$b(t)$ is real valued;}
\item[(iii)]  \textit{$b(t)$ and $b'(t)$ decreases exponentially;}
\item[(iv)]  \textit{for any pole $\omega_j$ satisfying $\Im \omega_j<\beta$, we have that $B(\omega_j)$ coincidences the value of $i\omega_j U+ aU_{\nu}$ for the solution $U$ of
\begin{equation}
\Delta U+ (1+\omega_j^2)U= -H  \textrm{ in }  \Omega.  \notag
\end{equation} }
\end{itemize}

We mainly concentrate on the choice of $b(t)$ such that  the  property $(iv)$ holds, as the others  are simple  to achieve. Actually, by taking $\psi= rU$ we get 
\begin{gather*}
\partial_r^2 \psi + (1+ \omega^2)\psi= -r\int_0^{+\infty} e^{-i\omega_j t}\big(\chi(t)h(t)+ \chi_{tt}u(t)+ 2\chi_t u_t(t)\big) dt, \\
\psi(0)=0, \; (iL\omega-a)\psi(L)+ aL \psi'(L)=L^2\int_0^{+\infty} e^{-i\omega_j t} \left(b(t)+ \chi_t u(t)\right) dt.
\end{gather*}
Assume that there are exactly $N$ poles in the region  $\Im z\leq \beta$, and there are given by $\{\omega_j\}_{j=1}^N$.   Because $\omega=\omega_j$ is a pole, the above equations admit solutions only for some well-chosen boundary values. More precisely,  for any given $f(r)$ we need to find the unique value $l(\omega_j)$ such that 
\begin{gather*}
\partial_r^2 \psi + (1+ \omega_j^2)\psi= f(r), \\
\psi(0)=0, \; (iL\omega_j-a)\psi(L)+ aL \psi'(L)=l(\omega_j),
\end{gather*}
has a $C^2$ solution.  Since $\omega_j$ is a pole such that the above equations with $f=l=0$ admit non-trivial solutions, we are allowed to set $\psi'(0)=0$.  Therefore, the function $\psi$ satisfies
\begin{gather*}
\partial_r^2 \psi + (1+ \omega_j^2)\psi= f(r), \\
\psi(0)=\psi'(0)= 0, \; (iL\omega_j-a)\psi(L)+ aL \psi'(L)=l(\omega_j).
\end{gather*}
By ignoring the third boundary condition and solving the preceding ODE, we obtain 
\begin{equation}
|\psi(r)|\lesssim r\lVert f\lVert_{L^2},\;    |\psi'(r)|\lesssim  \lVert f\lVert_{L^2},   \notag
\end{equation}
thus the value $l$ satisfies
\begin{equation}
|l(\omega_j)|\lesssim \lVert f\lVert_{L^2}.   \notag
\end{equation}
Therefore, the value on right hand side should linearly depend on  $(h, u_0, v_0)$:
\begin{equation*}
l(\omega_j)=  L_{\omega_j} \left(-r\int_0^{+\infty} e^{-i\omega_j t}\left(\chi(t)h(t)+ \chi_{tt}u(t)+ 2\chi_t u_t(t)\right) dt\right)=: \mathcal{L}_{\omega_j}^1(h, u_0, v_0),
\end{equation*}
which also means that 
\begin{equation}
\left|\int_0^{+\infty} e^{-i\omega_j t} \left(b(t)+ \chi_t u(t)\right) dt\right| \lesssim \left\lVert r\int_0^{+\infty} e^{-i\omega_j t}\left(\chi(t)h(t)+ \chi_{tt}u(t)+ 2\chi_t u_t(t)\right) dt\right\lVert_{L^2_r}.   \notag
\end{equation}
Moreover, the functionals  also satisfy the  conjugate property: $l(-\bar \omega_j)= \overline{l}(\omega_j)$, $\mathcal{L}^1_{-\bar{\omega}_j}= \overline{\mathcal{L}^1}_{\omega_j}$. \\

On the other hand, by the definition of $\chi_t u(t)$ we have 
\begin{equation}
L^2\int_0^{+\infty} e^{-i\omega_j t} \left(b(t)+ \chi_t u(t)\right) dt=: L^2\int_0^{+\infty} e^{-i\omega_j t} b(t) dt+ \mathcal{L}_{\omega_j}^2(h, u_0, v_0),  \notag
\end{equation}
thus  
\begin{equation}
\int_0^{+\infty} e^{-i\omega_j t} b(t) dt=\frac{1}{L^2} \left(\mathcal{L}_{\omega_j}^1(h, u_0, v_0)-\mathcal{L}_{\omega_j}^2(h, u_0, v_0)\right)= {\bf r}_{\omega_j}(h, u_0, v_0),   \notag
\end{equation}
with $\bf r_{\omega_j}$ linearly depending on $(h, u_0, v_0)$, such that ${\bf r}_{-\bar{\omega_j}}= \bar{{\bf r}}_{\omega_j}$.\\
Because 
\begin{align*}
&\;\;\;\;|\chi_t u(t, L)|+\left\lVert r\int_0^{+\infty} e^{-i\omega_j t}\left(\chi_{tt}u(t)+ 2\chi_t u_t(t)\right) dt\right\lVert_{L^2_r}\\&\lesssim \lVert (u_0, v_0)\lVert_{\mathcal{H}^1}+ \int_0^2\lVert h(s, \cdot)\lVert_{L^2(\overline{\Omega})} ds,
\end{align*}
we obtain
\begin{equation*}
\left|\int_0^{+\infty} e^{-i\omega_j t} b(t)dt\right| \lesssim \left\lVert \int_0^{+\infty} e^{-i\omega_j t}h(t) dt\right\lVert_{L^2(\Omega)}+ \lVert (u_0, v_0)\lVert_{\mathcal{H}^1}+ \int_0^2\lVert h(s, \cdot)\lVert_{L^2(\overline{\Omega})} ds.
\end{equation*}
Therefore, condition $(iv)$ means that  for every pole $\omega_j$ there is
\begin{equation}
\int_0^{+\infty} e^{-i\omega_j t} b(t)dt= {\bf r}_{\omega_j}(h, u_0, v_0)    
\end{equation}
with the functional ${\bf r}_{\omega_j}$ satisfying conjugate property and linearly depending on $(h, u_0, v_0)$,
\begin{align*}
|{\bf r}_{\omega_j}(h, u_0, v_0)|&\lesssim \int_0^{+\infty} e^{\beta_j t}\lVert h(t, \cdot)\lVert_{L^2(\overline{\Omega})} dt+ \lVert (u_0, v_0)\lVert_{\mathcal{H}^1}+ \int_0^2\lVert h(s, \cdot)\lVert_{L^2(\Omega)} ds,\\
&\lesssim \int_0^{+\infty} e^{(\beta-\varepsilon)t}\lVert h(t, \cdot)\lVert_{L^2(\overline{\Omega})} dt+ \lVert (u_0, v_0)\lVert_{\mathcal{H}^1}.
\end{align*}
Hence, there exists some constant $C_N$  such that 
\begin{equation*}
|{\bf r}_{\omega_j}(h, u_0, v_0)|\leq C_N \|(h, u_0, v_0)\|_{\mathcal{B}}, \forall (h, u_0, v_0)\in \mathcal{B}, \forall j\in \{1, 2,..., N\}.
\end{equation*}

Now we prove the following lemma which eventually will lead to the feedback $b(t)$ that we are seeking for.  Recalling \eqref{norm-B} on  the  norm for the triple $(h, u_0, v_0)$, we have 
\begin{lemma}\label{lem:bconstruction}
There exists some $C_0>0$ such that  for any $ (h, u_0, v_0)\in \mathcal{B}$, 
we can  find some control $b(t)\in\mathbb{R}$ compactly supported in $(2, 4)$ that verifies
\begin{itemize}
\item[(i)] $\int_0^{+\infty}e^{s\beta} (|b(s)|+|b'(s)|) ds\leq C_0\lVert (h, u_0, v_0)\lVert_{B}$;
\item[(ii)] $\int_0^{+\infty} e^{-i\omega_j t} b(t)dt= r_{\omega_j}(h, u_0, v_0), \forall j\in \{1, 2,..., N\}$.
\end{itemize}
\end{lemma}
\begin{proof}
This is indeed a moment problem that usually appears in  control theory, except that here we only have finitely many moments while the moment problem concerning exact controllability of partial differential equations is always  composed of infinitely many moments (see \cite{2009-Tucsnak-Weiss-book, Beauchard-2005, Avdonin-Ivanov} for the moment theory).  However, normally the moment problem on bi-orthogonal sequences only provides complex-valued functions,  here we need to select a real-valued  control term  as we are working on dispersive equations.  Generally speaking, this is not possible, as  for conjugate pairs $(\omega_j, -\bar{\omega}_j)$ the integrals $\int e^{-i\omega t}b(t)dt$ should be conjugate for real controls $b(t)$, while, on the other hand, the moments $r_{\omega_j}$ are not necessarily selected to verify the same conjugate property. For us, however, it is exactly the conjugacy fact $r_{-\bar\omega_j} = \bar r_{\omega_j}$ that allows us to solve this problem, and that is the reason why we always ensured this conjugation property before. 
\\
Observing the conjugacy  on $(\omega_j, -\bar \omega_j)$,  it suffices to verify the condition on those $\omega$ such that their real part is non negative.  Indeed, after simple calculation we are allowed to  divide $\{\omega_j\}$ and $\{r_{\omega_j}\}$ into two sets: 
\begin{gather*}
r_{\omega_j}= r_{i \beta_j}= r(j), \textrm{ for } j\in\{1, 2,..., K\},\\
r_{\omega_j}= r_{\alpha_j+ i\beta_j}= r^1(j)+ i r^2(j), \alpha_j\neq 0,  \textrm{ for } j\in\{K+1, ..., J\},\\
r_{-\bar\omega_j}= r_{-\alpha_j+ i\beta_j}= r^1(j)- i r^2(j),  \alpha_j\neq 0,  \textrm{ for } j\in\{K+1, ..., J\}.
\end{gather*}

We are able to show the existence of  real-valued smooth functions $b_j(t), b_j^1(t), b_j^2(t)$ supported on $(2, 4)$  such that 
\begin{gather*}
\int_0^{+\infty} e^{-i\omega_k t} b_j(t)dt= \delta_{j, k},\; \forall j, k\in \{1, 2,..., K\},\\
\int_0^{+\infty} e^{-i\omega_k t} b_j(t)dt=\int_0^{+\infty} e^{i\bar\omega_k t} b_j(t)dt= 0, \;\forall k\in \{K+1,..., J\},\\
\int_0^{+\infty} e^{-i\omega_k t} b_j^1(t)dt= \delta_{j, k},\; \forall j\in \{K+1, ..., J\}, k\in \{1, 2,..., J\},  \\
\int_0^{+\infty} e^{-i\omega_k t} b_j^2(t)dt= i\delta_{j, k},\; \forall j\in \{K+1, ..., J\}, k\in \{1, 2,..., J\}.
\end{gather*}
Indeed, for example, for $\{(\alpha_k, \beta_k)\}_{k\leq N}$  given with $\alpha_k> 0$ and $(\alpha_i, \beta_i)\neq (\alpha_j, \beta_j)$, we need to find $b(t)$ supported on $(2, 4)$ such that 
\begin{gather*}
\int_2^4 e^{\beta_1 t}\cos (\alpha_1 t) b(t)=1, \;\; \int_2^4 e^{\beta_1t}\sin (\alpha_1 t) b(t)=0, \\
\int_2^4 e^{\beta_kt}\cos (\alpha_k t) b(t)=0, \;\; \int_2^4 e^{\beta_kt}\sin (\alpha_k t) b(t)=0,\; k\geq 2.
\end{gather*}
This is possible thanks to the fact that $\{e^{\beta_k t}\cos (\alpha_k t)|_{(2, 4)}, e^{\beta_k t}\sin (\alpha_k t)|_{(2, 4)}\}_k$ are linearly independent. 

Therefore, by the above choice of $\{b_j, b_j^1, b_j^2\}$  the control term  $b(t)$ that satisfies the moment problem $(ii)$ can be selected as a combination of them:
\begin{equation}
b(t):=\sum_{j=1}^K r(j)b_j(t)+ \sum_{j=K+1}^J \left(r^1(j)b_j^1+ r^2(j) b_j^2\right),       \notag
\end{equation}
which of course verifies  
\begin{align*}
\int_0^{+\infty}e^{s\beta} (|b(s)|+|b'(s)|) ds&\leq C\left(\sum_{j=1}^K |r(j)|+  \sum_{j=K+1}^J (|r^1(j)|+|r^2(j)|)\right)\\&\leq C_0\lVert (h, u_0, v_0)\lVert_{\mathcal{B}}. \notag
\end{align*}
\end{proof}
 
As a consequence of  the above lemma, we immediately derive   the exponential decay of the energy by selecting the control term concerning the lemma.
\begin{align*}
\lVert u(t)\lVert_{\mathcal{H}^1}&\lesssim \int_0^{t} \lVert h(t, \cdot)\lVert_{L^2(\overline{\Omega})} dt+ \lVert (u_0, v_0)\lVert_{\mathcal{H}^1},  \textrm{ for } t\leq 2,\\
\lVert u(t)\lVert_{\mathcal{H}^1}&\lesssim e^{-\beta t}\left(\int_0^{+\infty} e^{\beta t}\lVert h(t, \cdot)\lVert_{L^2(\overline{\Omega})} dt+ \lVert (u_0, v_0)\lVert_{\mathcal{H}^1}\right), \textrm{ for } t> 2.
\end{align*}

\section{Open-loop stabilization of the NLKG}\label{sec-stab-nl}
In this section, we stabilize the nonlinear Klein-Gordon equation around the static solution $u\equiv 1$.  The idea is,  as usual,  to regard the nonlinear term as a perturbed forcing term and to use the linearized stabilization.  By considering $u=1+ v$ as well as the same change of variables as in Section \ref{sec-stab-li},
\begin{equation}
\overline{v}(t, x):= v\left(\frac{t}{\sqrt{2}},\; \frac{x}{\sqrt{2}}\right),\; \overline{b}(t):= \frac{1}{\sqrt{2}}b\left(\frac{t}{\sqrt{2}}\right), \; \overline{\Omega}:= B_L(0),  \; \overline{\beta}:= \frac{1}{\sqrt{2}}\beta, \notag
\end{equation}
 and still denoting the new variables by the old ones if there is no risk of confusion, we derive the following nonlinear equation \begin{gather*}
 \begin{cases}
\Box u- u=\frac{3}{2}u^2+ \frac{1}{2}u^3 \textrm{ in }  \Omega,\\
u_t+ au_{\nu}=b(t) \textrm{ on } \partial  \Omega, \\
u(0, x)=u_0, u_t(0, x)=v_0,
\end{cases}
\end{gather*} 
where the initial state $(u_0, v_0)$ is small.  We will show that for sufficiently small initial data, we can find control $b(t)$ that is supported on $(2, 4)$ such that the solution of the above nonlinear equation decay exponentially with a decay rate $\beta$:  if $a$ belongs to $\mathcal{A}(L)$, then $\beta$ is chosen sufficiently small (as the value $\beta_*$ stated before the statement of Theorem \ref{thm-li-key}); if $a$ does not belong to $\mathcal{A}(L)$, then $\beta$ can be chosen as any positive number that is strictly smaller than  $\frac{1}{2L} \log{\frac{1+a}{1-a}}$.   

 Since we are dealing with sub-critical nonlinear terms, we perform the classical fixed point argument benefiting the linearized result  Theorem \ref{thm-li-key}  to get the solution.
Let us introduce a new function space,
\begin{align*}
&\mathcal{D}:= \big\{ u\in C^0([0, +\infty); H^1_{rad}(\Omega))\cap  C^1([0, +\infty); L^2_{rad}(\Omega)):\\
&\hspace{2.5cm} \sup_{t\geq 0} e^{\beta t}(\lVert u(t)\lVert_{H^1(\Omega)}+\lVert u_t(t)\lVert_{L^2(\Omega)})<+\infty\big\}    \notag
\end{align*}
with its norm given by 
\begin{equation}
\lVert f(t, x)\lVert_{\mathcal{D}}:= \sup_{t\geq 0} e^{\beta t}(\lVert u(t)\lVert_{H^1(\Omega)}+\lVert u_t(t)\lVert_{L^2(\Omega)}), \forall f\in \mathcal{D}. \notag
\end{equation}

Let us  select some positive number $\varepsilon_{\beta}$ that will be fixed later on, and let $\varepsilon\in (0, \varepsilon_{\beta})$.  For any given initial state $(u_0, v_0)$ such that 
\begin{equation}
\lVert (u_0, v_0)\lVert_{\mathcal{H}^1}\leq \varepsilon,    \notag
\end{equation}
and for any function $v\in \mathcal{D}(2C_{\beta} \varepsilon)$, $i. e. $
\begin{equation}
\lVert v(t, x)\lVert_{\mathcal{D}}\leq 2C_{\beta} \varepsilon,    \notag
\end{equation}
we define the map $\mathcal{T}$ that maps $v(t, x)$ to  $u(t, x)$ as the solution of 
\begin{gather*}
\begin{cases}
\Box u- u=\frac{3}{2}v^2+ \frac{1}{2}v^3 \textrm{ in }  \Omega,\\
u_t+ au_{\nu}=b\left(\frac{3}{2}v^2+ \frac{1}{2}v^3, u_0, v_0\right)(t)  \textrm{ on } \partial \Omega, \\
u(0, x)=u_0, u_t(0, x)=v_0,
\end{cases}
\end{gather*} 
where $b\left(\frac{3}{2}v^2+ \frac{1}{2}v^3, u_0, v_0\right)(t)$ is chosen by Theorem \ref{thm-li-key} in order to stabilize the linear system.  We will show that for a good choice of $\varepsilon_{\beta}$ the map  $\mathcal{T}$ is actually a contraction on the Banach space $\mathcal{D}(2C_{\beta} \varepsilon)$, hence admit a fixed point $u\in \mathcal{D}(2C_{\beta} \varepsilon)$ as a solution of the nonlinear system, and which decays exponentially. \\

First, we show that 
\begin{equation}
\mathcal{T}: \mathcal{D}(2C_{\beta} \varepsilon)\rightarrow \mathcal{D}(2C_{\beta} \varepsilon).     \notag
\end{equation}
Indeed, for any $v\in \mathcal{D}(2C_{\beta} \varepsilon)$ , thanks to Theorem \ref{thm-li-key}, the solution $u=\mathcal{T}(v)$ satisfies
\begin{equation}
\lVert u(t)\lVert_{\mathcal{H}^1}\leq C_{\beta}e^{-\beta t}\left(\int_0^{+\infty} e^{\beta t}\lVert h(t, \cdot)\lVert_{L^2(\overline{\Omega})} dt+ \lVert (u_0, v_0)\lVert_{\mathcal{H}^1}\right),      \notag
\end{equation}
where $h= \frac{3}{2}v^2+ \frac{1}{2}v^3$.
 Therefore,  
 \begin{align*}
\lVert u(t)\lVert_{\mathcal{D}}
 &= \sup_{t\geq 0} e^{\beta t}\lVert u(t)\lVert_{H^1(\overline{\Omega})}, \\
 &\leq \sup_{t\geq 0} C_{\beta}\left(\int_0^{+\infty} e^{\beta t}\lVert h(t, \cdot)\lVert_{L^2(\overline{\Omega})} dt+ \lVert (u_0, v_0)\lVert_{\mathcal{H}^1}\right), \\
 &=  C_{\beta}\left(\int_0^{+\infty} e^{\beta t}\lVert \frac{3}{2}v^2+ \frac{1}{2}v^3\lVert_{L^2(\overline{\Omega})} dt+ \lVert (u_0, v_0)\lVert_{\mathcal{H}^1}\right), \\
 &\leq C_{\beta}\left(\int_0^{+\infty} e^{\beta t}(\frac{3}{2}\lVert v(t)\lVert_{H^1}^2+ \frac{1}{2}\lVert v(t)\lVert_{H^1}^3)dt+ \varepsilon\right),   \\
 &\leq  C_{\beta}\left(\int_0^{+\infty} e^{\beta t}(\frac{3}{2}(e^{-\beta t}2\varepsilon C_{\beta})^2+ \frac{1}{2}(e^{-\beta t}2\varepsilon C_{\beta})^3)dt+ \varepsilon\right), \\
 &\leq C_{\beta}\left(\frac{6C_{\beta}^2\varepsilon^2}{\beta}+ \frac{2C_{\beta}^3\varepsilon^3}{\beta}    + \varepsilon\right)\leq 2C_{\beta}\varepsilon, 
 \end{align*}
 provided that 
\begin{equation}\label{cond-1}
\frac{6C_{\beta}^2\varepsilon^2}{\beta}+ \frac{2C_{\beta}^3\varepsilon^3}{\beta} \leq \varepsilon,
\end{equation} 
where we used the Sobolev embedding $H^1(\Omega)\hookrightarrow L^6(\Omega)$.\\

Next, we show that under a suitable choice of $\varepsilon$, the map $\mathcal{T}$ is indeed a contraction. We select two functions $v_1, v_2\in \mathcal{D}(2C_{\beta}\varepsilon)$ and  suppose that $u_i$, $i= 1, 2$, are  solutions of 
 \begin{gather*}
 \begin{cases}
\Box u_i- u_i=\frac{3}{2}v_i^2+ \frac{1}{2}v_i^3 \textrm{ in } \Omega,\\
u_{it}+ au_{i\nu}=b\left(\frac{3}{2}v_i^2+ \frac{1}{2}v_i^3, u_0, v_0\right)(t) \textrm{ on } \partial\Omega, \\
u_i(0, x)=u_0, u_{it}(0, x)=v_0.
\end{cases}
\end{gather*} 
By considering the difference $u:= u_1-u_2$, thanks to the fact that $b$ is linear with respect to $(h, u_0, v_0)$,  we get 
\begin{gather*}
\begin{cases}
\Box u- u=h, \textrm{ in } \Omega,\\
u_t+ au_{\nu}=b(h, 0, 0) \textrm{ on } \partial \Omega, \\
u(0, x)=0, u_t(0, x)=0,
\end{cases}
\end{gather*} 
where 
\begin{equation}
h:= \frac{3}{2}(v_1-v_2)(v_1+v_2)+ \frac{1}{2}(v_1-v_2)(v_1^2+v_1v_2+v_2^2).  \notag
\end{equation}
Therefore, 
\begin{align*}
 & \;\;\;\;\;\lVert u(t)\lVert_{\mathcal{D}}, \\
 &\leq \sup_{t\geq 0} C_{\beta}\left(\int_0^{+\infty} e^{\beta t}\lVert h(t, \cdot)\lVert_{L^2(\overline{\Omega})} dt+ \lVert (0, 0)\lVert_{\mathcal{H}^1}\right), \\
 &=  C_{\beta}\left(\int_0^{+\infty} e^{\beta t}\lVert \frac{3}{2}(v_1-v_2)(v_1+v_2)+ \frac{1}{2}(v_1-v_2)(v_1^2+v_1v_2+v_2^2)\lVert_{L^2(\overline{\Omega})} dt\right), \\
 &\leq C_{\beta}\left(\int_0^{+\infty} e^{\beta t}(\frac{3}{2}\lVert v_1(t)-v_2(t)\lVert_{H^1} \lVert v_1(t)+v_2(t)\lVert_{H^1}\right.\\
& \hspace{5.5cm}  \left. + \frac{1}{2}\lVert v_1(t)-v_2(t)\lVert_{H^1}\lVert v_1(t)+v_2(t)\lVert_{H^1}^2)dt\right),   \\
 &\leq  C_{\beta}\left(\int_0^{+\infty} e^{\beta t}(\frac{3}{2}(e^{-\beta t}4\varepsilon C_{\beta}) e^{-\beta t}\lVert v_1-v_2\lVert_{\mathcal{D}}  + \frac{1}{2}(e^{-\beta t}4\varepsilon C_{\beta})^2)e^{-\beta t}\lVert v_1-v_2\lVert_{\mathcal{D}} dt\right), \\
 &\leq \lVert v_1-v_2\lVert_{\mathcal{D}}  C_{\beta}\left(\frac{12C_{\beta}\varepsilon}{\beta}+ \frac{8C_{\beta}^2\varepsilon^2}{\beta}\right) \leq \frac{1}{2}\lVert v_1-v_2\lVert_{\mathcal{D}},
 \end{align*}
provided that 
\begin{equation}\label{cond-2}
\frac{12C_{\beta}^2\varepsilon}{\beta}+ \frac{8C_{\beta}^3\varepsilon^2}{\beta}\leq \frac{1}{2}.
\end{equation}
Condition  \eqref{cond-2},  together with \eqref{cond-1},  characterize the choice of $\varepsilon_{\beta}>0$.  
 Then, by Banach's fixed point theorem we find a fixed point $u\in \mathcal{D}(2C_{\beta} \varepsilon)$ which is exactly the solution. Thus for initial data $(u_0, v_0)$ that satisfies $\lVert (u_0, v_0)\lVert_{\mathcal{H}^1}=\varepsilon$,  the unique  solution $u\in \mathcal{D}(2C_{\beta} \varepsilon)$  verifies
\begin{equation}
\lVert u(t)\lVert_{H^1}+\lVert u_t(t)\lVert_{L^2}\leq e^{-\beta t} 2C_{\beta}\varepsilon.    \notag
\end{equation}
Moreover, by the definition of the control,  $b(t)$  verifies
\begin{equation}
b(t)= \sum_{k=1}^N \tilde l_k(u_0, v_0) b_k(t)    \notag
\end{equation}
with $\tilde l_k(u_0, v_0)$ depending continuously on $(u_0, v_0)\in \mathcal{H}^1$ satisfying
\begin{equation}
 \sum_{k=1}^N |\tilde l_k(u_0, v_0)|\leq C_{\beta} \left(\int_0^{+\infty}e^{\beta t} \lVert\frac{3}{2}u^2(t)+\frac{1}{2}u^3(t)\lVert_{L^2}+\varepsilon\right)\leq 2\varepsilon C_{\beta}.    \notag
\end{equation}

\section{Closed-loop stabilization of the NLKG}\label{sec-stab-nl-closed}

We are now in a position to prove the stronger closed-loop stabilization result.   Without loss of generality, we assume that $s\in [(M_1)T_{\beta}, MT_{\beta})$ for some integer $M$.  The proof will be  divided into two parts, in the first part for $t\in (s, MT_{\beta})$  we use trivial \textit{a priori} energy estimate as the feedback during this period is not the one that provides decay stabilization, while in the second part for $t\geq MT_{\beta}$ we apply Theorem \ref{thm-stab} on each interval $[KT_{\beta}, (K+1)T_{\beta}), K\geq M$ such that the energy of the solution decay at least  $e^{-\beta T_{\beta}}$ in each of them.

The well-posedness of the closed-loop system is trivial, as in each interval $t\in (s, MT_{\beta})$ and $[KT_{\beta},(K+1)T_{\beta})), K\geq M$ the value of $l_k$ does not change, which implies that  the system can be regarded as an open-loop system, hence, of course, admit an unique solution.  In the following, we only focus on the stability issues. 

For any given $\beta$, thanks to Theorem \ref{thm-stab}, we get the value of $C_{\beta}>0, \varepsilon_{\beta}>0$ and smooth functions $\{b_k(t)\}_{k=1}^{N_{\beta}}$.    Let us  also select  some constant $T_{\beta}>0$ that will be chosen explicitly later on.  Without loss of generality, we  assume that $s\in [(M-1)T_{\beta}, MT_{\beta})$.  \\

In the first part $t\in (s, MT_{\beta})$, for ease of notations we work on $(u-1, u_t)$ instead of $(u, u_t)$, while still use the  notation of  $(u, u_t)$. 
For $t\in (s, MT_{\beta})$, $l_k(t)= l_k^0$  and the system on $(u, u_t)$ reads as 
\begin{gather}
\begin{cases}
\Box u(t, x)-2u(t, x)= 3u^2(t, x)+ u^3(t, x)  , t\in(s, MT_{\beta}), x\in \Omega      \notag\\
u_t(t, x)+ au_{\nu}(t, x)=\sum_{k=1}^N l_k^0 b_k\left(t-[\frac{t}{T_{\beta}}]T_{\beta}\right),  t\in(s, MT_{\beta}),  x\in \partial\Omega,    \notag\\
u(s, x)= u_0-1, u_t(s, x)=v_0.
\end{cases}
\end{gather}
We then extend this equation on the time interval $(s, s+ T_{\beta})$.    By regarding the nonlinear term as a source term (thanks to the sub-critical setting) and by applying the direct energy estimate Lemma \ref{lem-a-es}, we know that the energy
\begin{equation}
\frac{1}{2}\lVert u(t)\lVert_{\mathcal{H}^1}^2\leq \tilde{E}(u(t)):= \frac{1}{2}\lVert u(t)\lVert_{\mathcal{H}^1}^2+ \frac{1}{a}\int_s^t \int_{\partial \Omega}u_t^2(\tau, x)d\sigma d\tau   \notag
\end{equation}
satisfies
\begin{align*}
\frac{d}{dt}\lVert u(t)\lVert_{\mathcal{H}^1}&\leq C\lVert u(t)\lVert_{\mathcal{H}^1}+ 3\lVert u(t)\lVert_{\mathcal{H}^1}^2+\lVert u(t)\lVert_{\mathcal{H}^1}^3+C\sum_k |l_k^0|, \\
&\leq C\lVert u(t)\lVert_{\mathcal{H}^1}+ 3\lVert u(t)\lVert_{\mathcal{H}^1}^2+\lVert u(t)\lVert_{\mathcal{H}^1}^3+ 2CC_{\beta} \lVert (u_0, v_0)-(1, 0)\lVert_{\mathcal{H}^1}, \\
&= C\lVert u(t)\lVert_{\mathcal{H}^1}+ 3\lVert u(t)\lVert_{\mathcal{H}^1}^2+\lVert u(t)\lVert_{\mathcal{H}^1}^3+ 2CC_{\beta} \lVert u(s)\lVert_{\mathcal{H}^1}.
\end{align*}
Though evolve as a nonlinear equation that may blow up for any nontrivial initial state, for any given bounded time interval we are allowed to set the initial data small enough such that the solution grows exponentially in that interval. More precisely, 
we will be selecting $T_{\beta}$ and $\tilde{\varepsilon}_{T_{\beta}}$ by that 
\begin{equation}
\lVert u(t)\lVert_{\mathcal{H}^1}\leq 1, \forall t\in (s, s+ T_{\beta}), \textrm{ if } \lVert u(t)\lVert_{\mathcal{H}^1}\leq \tilde{\varepsilon}_{T_{\beta}},    
\end{equation}
which is possible by choosing $\tilde{\varepsilon}_{T_{\beta}}$ for any given $T_{\beta}$. If the above condition holds, then for any $\lVert u(s)\lVert_{\mathcal{H}^1}\leq \tilde{\varepsilon}_{T_{\beta}}$, we have that 
\begin{equation}
\lVert u(t)\lVert_{\mathcal{H}^1}\leq C_{T_{\beta}} \lVert u(s)\lVert_{\mathcal{H}^1},  t\in (s, s+ T_{\beta}).     \notag
\end{equation}
Therefore, by changing back the notation of $(u, u_t)$, we have proved that for any $T_{\beta}$ there exists $\tilde{\varepsilon}_{T_{\beta}}>0$ and $C_{T_{\beta}}>0$ such that, for any initial state satisfying 
\begin{equation}
\lVert (u_0-1, v_0)\lVert_{\mathcal{H}^1}\leq \tilde{\varepsilon}_{T_{\beta}},    \notag
\end{equation}
we know that  the solution of the closed-loop system verifies
\begin{equation}
\lVert (u, u_t)(t)-(1, 0)\lVert_{\mathcal{H}^1}\leq C_{T_{\beta}} \lVert (u_0-1, v_0) \lVert_{\mathcal{H}^1},  t\in (s, MT_{\beta}).    
\end{equation}
\\
Next, we turn to the stabilization part. In this part, we will call directly the exponential decay Theorem \ref{thm-stab}. In fact, we only work on the first periodic  interval $[MT_{\beta}, (M+1)T_{\beta})$, for which the ``initial state" is $(u(MT_{\beta}), u_t(MT_{\beta}))$. By the definition of $l_k(MT_{\beta})$, in this interval $l_k(t)\equiv\tilde{l}_k(u(MT_{\beta}, u_t(MT_{\beta})$, then the system becomes 
\begin{gather}
\begin{cases}
\Box u(t, x)+ u(t, x)-u^3(t, x)=0, t\in[MT_{\beta}, (M+1)T_{\beta}), x\in \Omega      \notag\\
u_t(t, x)+ au_{\nu}(t, x)=  \sum_{k=1}^{N_{\beta}}  \tilde{l}_k(\tilde{u}_0, \tilde{v}_0) b_k(t-MT_{\beta}),  t\in[MT_{\beta}, (M+1)T_{\beta}),  x\in \partial\Omega,    \notag\\
u(MT_{\beta}, x)= \tilde u_0, u_t(MT_{\beta}, x)=\tilde v_0.
\end{cases}
\end{gather}

Observe that the preceding equation is exactly the equation that appears in Theorem \ref{thm-stab} with initial time $MT_{\beta}$ and initial state $(\tilde{u}_0, \tilde{v}_0)$.  Hence, once the initial state is smaller than $\varepsilon_{\beta}$ we will be able to apply the decay theorem, which is fulfilled if $(u_0, v_0)$ is selected to be smaller than some $\varepsilon_{T_{\beta}}< \tilde{\varepsilon}_{T_{\beta}}$ such that 
\begin{equation}
\varepsilon_{T_{\beta}} C_{T_{\beta}}\leq \varepsilon_{\beta}.  \notag
\end{equation}
 Thus we have the following exponential decay result, $\forall t\in [MT_{\beta}, (M+1)T_{\beta})$, 
\begin{equation}
\lVert (u(t), u_t(t))-(1, 0)\lVert_{\mathcal{H}^1}\leq 2C_{\beta} e^{-\beta (t-MT_{\beta})} \lVert (\tilde u_0, \tilde v_0)-(1, 0)\lVert_{\mathcal{H}^1}. \notag
\end{equation}
Notice that the constant $C_{\beta}$ is independent of $T_{\beta}$, we will select $T_{\beta}$ large enough such that 
\begin{equation}
2C_{\beta} e^{-\beta T_{\beta}}\leq e^{-(\beta-\varepsilon_0) T_{\beta}},
\end{equation}
which is of course possible as $\varepsilon_0>0$.\\

Since $u((M+1)T_{\beta})$ is even smaller than  $u(MT_{\beta})$, we are allowed to repeat the same  procedure in $[(M+1)T_{\beta}, (M+2)T_{\beta})$ to achieve exponential decay.  By repeating this piecewise exponential decay procedure, we let the solution tends to 0.  To be more precisely, we have the following decay result
\begin{align*}
\lVert (u, u_t)(t)-(1, 0)\lVert_{\mathcal{H}^1}&\leq C_{T_{\beta}} \left(e^{-(\beta-\varepsilon_0) T_{\beta}}\right)^{[\frac{t-MT_{\beta}}{T_{\beta}}]} 2C_{\beta}\lVert (u_0, v_0)-(1, 0)\lVert_{\mathcal{H}^1}, \\
&\leq \tilde C_{\beta}e^{-(\beta-\varepsilon_0) (t-s)}\lVert (u_0, v_0)-(1, 0)\lVert_{\mathcal{H}^1}, \forall t>s.
\end{align*}

The decay of $l_k$ then obviously comes from its definition.

\section{Further comments}\label{sec-com}
We believe that this paper presents more interesting open questions than answers.  Even if we do not  ask about other typical stabilization problems  such as rapid stabilization or finite time stabilization, the following questions come naturally. 

As we can see from Theorem \ref{thm-unstab} that the system around simple positive static solutions with any dissipative boundary condition is unstable, it is already of significant importance to understand the stability analysis for dispersive equations  around soliton-like solutions, for instance for the easiest case, is the system around $u=1$ always unstable with dissipative boundary condition $u_t(x)+ a(x) u_{\nu}(x)=0$? It is natural expect instability when $a(x)$ close enough to $a$. What about other cases, like $a(x)$ compactly supported in a part of the boundary?  What about  blow up situations?

 As stated in the introduction, we start by considering the focusing subcritical Klein-Gordon in the radial setting around the simplest static solution $u=1$.  We hope that our method that gives quantitative stabilization can be generalized to non radial cases (still close to the static solution).   Of course, it will be further important to stabilize focusing systems around soliton-like static solutions, for which the specific  situation could become more complicated and so would  the analysis. 
 
Due to the focusing setting, we are not able to adapt the famous Bardos-Lebeau-Rauch theory, and also as we can see in Theorem \ref{thm-unstab} that the dissipative boundary stabilization fails (maybe also fails for other focusing cases), it is thus natural to ask whether we can stabilize the system by adding some control term on a part of the dissipative boundary  under suitable geometric conditions.

\appendix
\section{Proof of Lemma \ref{lem:sesqsect} and  Lemma \ref{lem-underline}}\label{App-A}

\begin{proof}[Proof of Lemma \ref{lem:sesqsect}]

Consider 
\begin{align*}
p_{\omega}(\phi, \phi): = \int_{\Omega}\big|\nabla \phi\big|^2\,dx-\int_{\Omega}(V+\omega^2)\big|\phi\big|^2dx+ \frac{i\omega}{a}\int_{\partial \Omega}\big|\phi\big|^2d\sigma,\,\omega\in \mathbf{C}
\end{align*}
Our aim is to exhibit for each $\omega\in \mathbf{C}$ a real constant $\gamma$ as well as a positive constant $C$ with the property that for each $\phi$ in the domain of $p_{\omega}(\phi, \phi)$ we have 
\begin{align*}
\big|\Im p_{\omega}(\phi, \phi)\big|\leq C\cdot \big[\Re p_{\omega}(\phi, \phi) - \gamma\big\|\phi\big\|^2\big]. 
\end{align*}
Writing $\omega: = \alpha + i\beta$, we claim that in fact 
\[
\gamma = -\big(2\alpha^2 + \big\|V\big\|_{L^\infty(\Omega)} + C\big)
\]
for sufficiently large $C = C(\beta, a, L)$ works. Observe that 
\begin{align*}
\Re p_{\omega}(\phi, \phi) - \gamma\big\|\phi\big\|^2 \geq \int_{\Omega}\big|\nabla \phi\big|^2\,dx + (C+\alpha^2 + \beta^2)\int_{\Omega}\big|\phi\big|^2dx - \frac{\beta}{a}\int_{\partial \Omega}\big|\phi\big|^2d\sigma.
\end{align*}
It remains to estimate the third term on the right. Observe that given $\phi\in H^1(\Omega)$, which we assume real-valued for simplicity, there is $r_*\in [\frac{L}{4}, \frac{3L}4]$ with the property that 
\[
\big|\phi(r_*)\big|\leq C_1(L)\cdot \big\|\phi\big\|_{L^2(\Omega)}, 
\]
and then using the Cauchy-Schwarz inequality and the fundamental theorem of calculus we infer 
\begin{align*}
\big|\phi(1)\big|^2&\leq \big|\phi(r_*)\big|^2 + \big|\int_{r_*}^1 \phi_s(s)\cdot \phi(s)\,ds\big|\\
&\leq C_1^2(L)\cdot \big\|\phi\big\|_{L^2(\Omega)}^2 + C_2(L)\cdot \big\|\phi\big\|_{H^1(\Omega)}\cdot \big\|\phi\big\|_{L^2(\Omega)}. 
\end{align*}
Applying the AGM inequality to the last term, we infer that 
\begin{align*}
\frac{\beta}{a}\int_{\partial \Omega}\big|\phi\big|^2d\sigma\leq \frac12\cdot \big\|\phi\big\|_{H^1(\Omega)}^2 + C_3(\beta, a, L)\cdot \big\|\phi\big\|_{L^2(\Omega)}^2. 
\end{align*}
If we now define 
\begin{align*}
C: = C_3(\beta, a, L) + 1, 
\end{align*}
we conclude that 
\begin{align*}
\Re p_{\omega}(\phi, \phi) - \gamma\big\|\phi\big\|^2 \geq  \frac12\int_{\Omega}\big|\nabla \phi\big|^2\,dx + (1+\alpha^2 + \beta^2)\int_{\Omega}\big|\phi\big|^2dx. 
\end{align*}

To conclude the proof of the lemma, it suffices to combine the preceding estimate with the following one:
\begin{align*}
\big|\Im p_{\omega}(\phi, \phi) \big|&\leq 2|\alpha\beta|\cdot \int_{\Omega}\big|\nabla \phi\big|^2\,dx + \big|\frac{a}{\alpha}\big|\cdot \int_{\partial \Omega}\big|\phi\big|^2d\sigma\\
&\leq C_4(a,\alpha,\beta,L)\cdot \big\|\phi\big\|_{H^1(\Omega)}^2. 
\end{align*}

\end{proof}

\begin{proof}[Proof of Lemma \ref{lem-underline}]
Assume that for some $U$ we have that 
\begin{gather*}
\Delta U+ (V+\omega^2)U= 0, \textrm{ in } \Omega,\\
i\omega U+ aU_{\nu}= 0, \textrm{ on } \partial \Omega.
\end{gather*}
First we treat the case when $\beta=0$ and $\alpha\neq 0$. Via direct calculation we get
\begin{align*}
0&=\langle \Delta U+ (V+\omega^2)U, U\rangle,\\
&= \int_{\Omega} -|\nabla U|^2+ (V+\alpha^2)|U|^2 dx + \int_{\partial\Omega}\partial_{\nu}U \bar{U}d\sigma, \\
&= \int_{\Omega} -|\nabla U|^2+ (V+\alpha^2)|U|^2 dx -\frac{i\alpha}{a} \int_{\partial\Omega}|U|^2d\sigma,
\end{align*}
which clearly implies that $U=0,\,\partial_\nu U = 0$ on the boundary $\partial\Omega$. This in turn implies $U =0$ by unique continuation. 
\\
Next consider the case when $\beta<0$ and $\alpha\neq 0$. By using the same integration  by parts we have
\begin{align*}
0&=\langle \Delta U+ (V+\omega^2)U, U\rangle,\\
&= \int_{\Omega} -|\nabla U|^2+ (V+\alpha^2-\beta^2+ 2i\alpha\beta)|U|^2 dx + \int_{\partial\Omega}\partial_{\nu}U \bar{U}d\sigma, \\
&= \int_{\Omega} -|\nabla U|^2+ (V+\alpha^2-\beta^2+ 2i\alpha\beta)|U|^2 dx -\frac{i\alpha-\beta}{a} \int_{\partial\Omega}|U|^2d\sigma.
\end{align*}
Now we consider the imaginary part of the preceding formula, which leads to 
\begin{equation}
0= \int_{\Omega} 2i\alpha\beta|U|^2 dx -\frac{i\alpha}{a} \int_{\partial\Omega}|U|^2d\sigma= i\alpha \left( \int_{\Omega} 2\beta|U|^2 dx -\frac{1}{a} \int_{\partial\Omega}|U|^2d\sigma\right).    \notag
\end{equation}
Because $a>0$ and $\beta<0$, we know that $U=0$.

To get the bounds for the inhomogeneous problem, by taking the same calculation we observe  that 
\begin{equation}
\Im \int_{\Omega} H\bar U dx=  \alpha \left( \int_{\Omega} -2\beta|U|^2 dx +\frac{1}{a} \int_{\partial\Omega}|U|^2d\sigma\right),     \notag
\end{equation}
which, together with the Cauchy-Schwartz inequality, yields 
\begin{gather*}
\lVert U\lVert_{L^2(\Omega)}\leq \frac{1}{2|\alpha \beta|}\lVert H\lVert_{L^2(\Omega)}, \\
\lVert U\lVert_{L^2(\partial\Omega)}\leq \left(\frac{a}{2\alpha^2|\beta|}\right)^{\frac{1}{2}}\lVert H\lVert_{L^2(\Omega)}.
\end{gather*}

In the end, when $\omega=i\beta$ with $\beta\in (-C, 0)$, we have 
\begin{equation}
\int_{\Omega} V U^2 dx=\int_{\Omega}|\nabla U|^2+ \beta^2 U^2 dx- \frac{\beta}{a}\int_{\partial \Omega} U^2d\sigma,     \notag
\end{equation}
which leads to
\begin{equation}
\lVert U\lVert_{H^1(\Omega)}\lesssim \lVert U\lVert_{L^2(\Omega)}.     \notag
\end{equation}
It appears that we can only get a $\lesssim \frac{1}{|\beta|}$ bound for the trace, however, as we are working in the radial setting, the $H^1$ bound of $U$ gives the $H^1_r$ bounded of $rU(r)$, which further implies the trace bound
\begin{equation}
|L U(L)|\leq   \lVert rU(r)\lVert_{C(0, L)} \lesssim  \lVert rU(r)\lVert_{H^1_r} \lesssim \lVert U\lVert_{H^1(\Omega)}\lesssim \lVert U\lVert_{L^2(\Omega)}.  \notag
\end{equation}
\end{proof}

\section{On the asymptotic estimates on $\eta$ and $\Gamma_r$}\label{App-B}
\subsection{The  explicit asymptotic estimate of $\eta$}
By the definition of $\eta$ we have 
\begin{align*}
&\;\;\;\;\; (iL \omega-a) \left(e^{2i \langle\omega\rangle L}-1\right)   + i aL\langle\omega\rangle \left(e^{2i \langle\omega\rangle L}+ 1\right)   \\
&=-\eta \left(  i aL\langle\omega\rangle \left(e^{2i \langle\omega\rangle L}- 1  \right)+ (iL \omega-a) \left(e^{2i \langle\omega\rangle L}+ 1\right) \right),
\end{align*}
thus, {\small
\begin{align*}
&\;\;(i L -\frac{L\beta+a}{\alpha})\left( e^{i 2L \omega}-1 + \frac{iL}{\alpha}e^{i 2L \omega}+ O(\frac{1}{\alpha^2})\right)+ iaL (1+ \frac{i\beta}{\alpha}+O(\frac{1}{\alpha^2}))\\&\hspace{7.5cm}\cdot\left(e^{i 2L \omega}+1 + \frac{iL}{\alpha}e^{i 2L \omega}+ O(\frac{1}{\alpha^2})\right), \\
&=-\eta \left(( (i L -\frac{L\beta+a}{\alpha})\left( e^{i 2L \omega}+1 + \frac{iL}{\alpha}e^{i 2L \omega}+ O(\frac{1}{\alpha^2})\right)+ iaL (1+ \frac{i\beta}{\alpha}+O(\frac{1}{\alpha^2}))\right.\\ 
&\left.\hspace{7.5cm}\cdot\left(e^{i 2L \omega}-1 + \frac{iL}{\alpha}e^{i 2L \omega}+ O(\frac{1}{\alpha^2})\right)  \right)
\end{align*}
}
therefore,
{\small
\begin{align*}
&\;\;(1+ \frac{i}{\alpha}(\beta+ \frac{a}{L}))\left( e^{i 2L \omega}-1 + \frac{iL}{\alpha}e^{i 2L \omega}+ O(\frac{1}{\alpha^2})\right)+ a (1+ \frac{i\beta}{\alpha}+O(\frac{1}{\alpha^2}))\\&\hspace{7.5cm}\cdot\left(e^{i 2L \omega}+1 + \frac{iL}{\alpha}e^{i 2L \omega}+ O(\frac{1}{\alpha^2})\right), \\
&=-\eta\left(  (1+ \frac{i}{\alpha}(\beta+ \frac{a}{L})) \left( e^{i 2L \omega}+1 + \frac{iL}{\alpha}e^{i 2L \omega}+ O(\frac{1}{\alpha^2})\right)+ a (1+ \frac{i\beta}{\alpha}+O(\frac{1}{\alpha^2}))  \right.
\\&\hspace{7.5cm}\left.\cdot\left(e^{i 2L \omega}-1 + \frac{iL}{\alpha}e^{i 2L \omega}+ O(\frac{1}{\alpha^2})\right)    \right).
\end{align*}
}
To simplify the notations, we denote $e^{i 2L \omega}$ by $e$, which gives 
\begin{align*}
&\;\; e-1+ a(e+1)+ \frac{i}{\alpha} \left((\beta+\frac{a}{L})(e-1)+ Le+ a\beta(e+1)+ aLe\right)+ O(\frac{1}{\alpha^2}), \\
&= -\eta\left( e+1+ a(e-1)+ \frac{i}{\alpha} \left((\beta+\frac{a}{L})(e+1)+ Le+ a\beta(e-1)+ aLe\right)+ O(\frac{1}{\alpha^2}) \right).
\end{align*}
Therefore, $\eta= \eta_0+ \frac{\eta_1}{\alpha}+ O(\frac{1}{\alpha^2})$ satisfies {\small
\begin{align*}
 &\;\;\;\;\;\;\;\;\;\;\;\;\;\;\;\;\;\;\;\;\;\;\;\;\;\;\;\;e-1+ a(e+1)= -\eta_0\left( e+1+ a(e-1)\right), \notag\\
 &\;\;\;\;(\beta+\frac{a}{L})(e-1)+ Le+ a\beta (e+1)+ aLe\notag \\ &= i\eta_1 \left(e+1+a(e-1)\right) -\eta_0 \left((\beta+\frac{a}{L})(e+1)+ Le+ a\beta(e-1)+ aLe\right)+ O(\frac{1}{\alpha^2}),  \notag
\end{align*}
}
thus
\begin{align*}
\eta_0&= \frac{1-a- e(1+a)}{1-a+ e(1+a)}= \frac{1-c_0 e}{1+c_0e},\;  c_0= \frac{1+a}{1-a},  \\
\eta_1&= -i\frac{e}{(1+c_0e)^2}\left( \beta+\frac{a}{L}+L+aL+a\beta-c_0(\beta+\frac{a}{L})+c_0a\beta\right) \frac{1}{1-a}= c_1 \frac{e}{(1+c_0e)^2},
\end{align*}
where $c_0, c_1$ are constants.
Hence, there exist $\pi/L$ periodic functions $\eta_0, \eta_1$ such that 
\begin{align*}
\eta&= \eta_0+ \frac{\eta_1}{\alpha}+ O(\frac{1}{\alpha^2}),\\
&= \frac{1-c_0 e}{1+c_0e}+ c_1 \frac{e}{(1+c_0e)^2} \frac{1}{\alpha}+ O(\frac{1}{\alpha^2}), \\
&=  \frac{-1+d_0 e^{-i2L\alpha}}{1+d_0 e^{-i2L\alpha}}+ \frac{d_1 e^{-i2L\alpha}}{(
1+ d_0 e^{-i2L\alpha})^2}\frac{1}{\alpha}+ O(\frac{1}{\alpha^2})
\end{align*}
where constants $d_i$ satisfies
\begin{equation}
d_0= \frac{1-a}{1+a} e^{2L\beta}<1,\; d_1= c_1 \left(\frac{1-a}{1+a}\right)^2 e^{2L\beta}.      \notag
\end{equation}

\subsection{On the estimate of $\Gamma_r$}
Here we present  the explicit calculation on $\Gamma_r$.
For $r<s$, we have {\small
\begin{align*}
&\;\;\;\;\;4\Gamma(r, s)-4\tilde{\Gamma}(r, s), \\
&=\frac{i}{\alpha\eta}\left(\phi_1(r)\phi_2(s)- e^-(r)\big(e^-(s)+\eta e^+(s)\big)\right)- \frac{i}{\alpha}(\frac{1}{\eta_0}- \frac{1}{\eta})e^-(r)e^-(s)+ \frac{i}{\eta}\left(\frac{1}{\langle\omega\rangle}-\frac{1}{\alpha}\right)\phi_1(r)\phi_2(s), \\
&= \frac{i}{\alpha\eta} \left(\big( e^-(r)+\frac{ir}{2\alpha}e^+(r)+\mathbf{r}(r, \alpha)\big)\big(e^-(s)+\eta e^+(s)+O(\frac{1}{\alpha})\big)- e^-(r)\big(e^-(s)+\eta e^+(s)\big)\right) \\
&\;\;\;\;-  \frac{i}{\alpha}(\frac{\eta_1}{\eta_0^2}\frac{1}{\alpha}+O(\frac{1}{\alpha^2}))e^-(r)e^-(s)+ O(\frac{1}{\alpha^3})\phi_1(r)\phi_2(s) ,\\
&=  \frac{i}{\alpha\eta} \left( \big( e^-(r)+\frac{ir}{2\alpha}e^+(r)+\mathbf{r}(r, \alpha)\big)O(\frac{1}{\alpha})+\big(\frac{ir}{2\alpha}e^+(r)+ \mathbf{r}(r, \alpha)\big)\big(e^-(s)+\eta e^+(s)\big)\right),\\
&\;\;\;\;-  \frac{i}{\alpha}(\frac{\eta_1}{\eta_0^2}\frac{1}{\alpha}+O(\frac{1}{\alpha^2}))e^-(r)e^-(s)+ O(\frac{1}{\alpha^3})\phi_1(r)\phi_2(s),
\end{align*} }
thus  
{\small \begin{align*}
&\;\;\;\;\; 4\Gamma_r(r, s)-4\tilde{\Gamma}_r(r, s)\\
&=-\frac{\omega}{\alpha\eta}e^+(r)\left(\frac{is}{2\alpha}e^+(s)+\eta\frac{is}{2\alpha}e^-(s) \right)-\frac{\omega}{\alpha\eta}\frac{ir}{2\alpha}e^-(r)(e^-(s)+\eta e^+(s))+\frac{\omega}{\alpha^2}\frac{\eta_1}{\eta_0^2}e^+(r)e^-(s)  +O(\frac{1}{\alpha^2}), \\
&= -\frac{i}{2\eta \alpha}\left(s e^+(r)\left(e^+(s)+\eta  e^-(s)\right) +  r e^-(r)(e^-(s)+\eta e^+(s))    \right)+ \frac{1}{\alpha}\frac{\eta_1}{\eta_0^2}e^+(r)e^-(s)  +O(\frac{1}{\alpha^2}).
\end{align*}}
For $s\leq r$, we have  {\small
\begin{align*}
&\;\;\;\;\; 4\Gamma(r, s)-4\tilde{\Gamma}(r, s)\\
&=\frac{i}{\alpha\eta}\left(\phi_1(s)\phi_2(r)- e^-(s)\big(e^-(r)+\eta e^+(r)\big)\right)- \frac{i}{\alpha}(\frac{1}{\eta_0}- \frac{1}{\eta})e^-(r)e^-(s)+ \frac{i}{\eta}\left(\frac{1}{\langle\omega\rangle}-\frac{1}{\alpha}\right)\phi_1(s)\phi_2(r), \\
&= \frac{i}{\alpha\eta} \left(
O(\frac{1}{\alpha})\Big(e^-(r)+\eta e^+(r)+\frac{ir}{2\alpha}(e^+(r)+\eta e^-(r))+\mathbf{r}(r, \alpha)\Big)
\right. \\
&\;\;\;\;\left.+ e^-(s)\Big(\frac{ir}{2\alpha}(e^+(r)+\eta e^-(r))+\mathbf{r}(r, \alpha)\Big) \right)-  \frac{i}{\alpha}(\frac{\eta_1}{\eta_0^2}\frac{1}{\alpha}+O(\frac{1}{\alpha^2}))e^-(r)e^-(s)+ O(\frac{1}{\alpha^3})\phi_1(s)\phi_2(r),
\end{align*}}
hence  {\small 
\begin{align*}
&\;\;\;\;\; 4\Gamma_r(r, s) -4\tilde{\Gamma}_r(r, s), \\
&= \frac{i}{\alpha\eta}(\frac{is}{2\alpha}e^+(s)+\mathbf{r}(s, \alpha))\Big(e^-(r)+\eta e^+(r)+\frac{ir}{2\alpha}(e^+(r)+\eta e^-(r))+\mathbf{r}(r, \alpha)\Big)_r,\\
&\;\;\;\;+ \frac{i}{\alpha\eta}e^-(s)\Big(\frac{ir}{2\alpha}(e^+(r)+\eta e^-(r))+\mathbf{r}(r, \alpha)\Big)_r-  \left( \frac{i}{\alpha}(\frac{\eta_1}{\eta_0^2}\frac{1}{\alpha}+O(\frac{1}{\alpha^2}))e^-(r)e^-(s)\right)_r\\
&\;\;\;\;+  O(\frac{1}{\alpha^3})\phi_1(s)(\phi_2)_r(r)     ,\\
&=-\frac{\omega}{\alpha\eta} \frac{is}{2\alpha}e^+(s)(e^+(r)+\eta e^-(r))-\frac{\omega}{\alpha\eta} \frac{ir}{2\alpha}e^-(s)(e^-(r)+\eta e^+(r))+\frac{\omega}{\alpha^2}\frac{\eta_1}{\eta_0^2}e^+(r)e^-(s) +O(\frac{1}{\alpha^2}), \\
&= -\frac{i}{2\alpha\eta}\Big(se^+(s)(e^+(r)+\eta e^-(r))+re^-(s)(e^-(r)+\eta e^+(r)) \Big)+ \frac{1}{\alpha}\frac{\eta_1}{\eta_0^2}e^+(r)e^-(s)+ O(\frac{1}{\alpha^2}).
\end{align*}  }

\bibliographystyle{plain} 
\bibliography{wavestab}

\end{document}